\documentclass[reqno]{amsart}
\usepackage{amsthm,amsmath,amsfonts,amssymb,amscd,mathrsfs,diagrams}
\usepackage{enumerate}
\usepackage{bbm} 
\usepackage{hyperref}
\usepackage{eucal}

\newarrow {Dotsto} ....> 
\newarrow {Onto} ----{>>} 
\newarrow {Into} C---> 
\newarrow{Eq} ===={=} 

\newcommand{\gel}{g}

\newcommand{\Vf}{\mathfrak{C}\kern-.05em\mathfrak{R}}
\newcommand{\Af}{\mathfrak{A}}
\newcommand{\Sf}{\mathfrak{S}}

\renewcommand{\setminus}{\smallsetminus}
\DeclareMathOperator{\Pf}{Pf} 
\DeclareMathOperator{\Aut}{Aut}
\DeclareMathOperator{\reg}{reg}
\DeclareMathOperator{\B}{{B}\kern-.1em}

\DeclareMathOperator{\Fl}{Flag}
\DeclareMathOperator{\GL}{GL}

\DeclareMathOperator{\Gr}{Gr}
\DeclareMathOperator{\Hom}{Hom}

\DeclareMathOperator{\Obs}{Obs}
\DeclareMathOperator{\Pic}{Pic}
\DeclareMathOperator{\Proj}{Proj}
\DeclareMathOperator{\Res}{Res}
\DeclareMathOperator{\SL}{SL}

\DeclareMathOperator{\Spec}{Spec}

\DeclareMathOperator{\bun}{{\mathfrak{Bun}}}
\DeclareMathOperator{\codim}{codim}

\DeclareMathOperator{\crit}{Crit}
\DeclareMathOperator{\fix}{Fix}

\DeclareMathOperator{\pt}{pt}

\DeclareMathOperator{\spec}{Spec}
\DeclareMathOperator{\stab}{Stab}

\newcommand{\0}{\boldsymbol{0}}
\newcommand{\Acal}{{\mathscr A}}
\newcommand{\Bcal}{{\mathscr B}}
\newcommand{\Bchi}{\B\kern-.1em\chi}
\newcommand{\BchiR}{\B\kern-.1em\chiR}

\newcommand{\CC}{\mathbb C} 
\newcommand{\Ccal}{\mathscr C}
\newcommand{\Ccalo}{\widehat{\Ccal}} 
\newcommand{\Ccalt}{\widetilde{\Ccal}} 

\newcommand{\EE}{\mathbb E}
\newcommand{\Ecal}{\mathscr{E}}

\newcommand{\Fcal}{\mathscr{F}}
\newcommand{\Gammah}{\widehat{\Gamma}} 
\newcommand{\Gh}{\widehat{G}} 
\newcommand{\GQ}{B} 

\newcommand{\Gmax}{G_{\operatorname{\max}}}
\newcommand{\Gq}{b} 

\newcommand{\Hcal}{\mathscr H}

\newcommand{\II}{{\mathbb{{I}}}\!}
\newcommand{\IX}{\I\X}
\newcommand{\I}{\II}

\newcommand{\Kcal}{\mathscr K}
\newcommand{\LGB}{\mathfrak{L}} 
\newcommand{\LGQ}{{\operatorname{LGQ}}} 
\newcommand{\Lcal}{\mathscr L}

\newcommand{\Mf}{\mathfrak{M}} 
\newcommand{\Mcal}{\mathscr M}

\newcommand{\Ocal}{\mathscr O}
\newcommand{\PP}{\mathbb P}
\newcommand{\Pcal}{\mathscr P}
\newcommand{\Pcalt}{\widetilde{\Pcal}}

\newcommand{\QQ}{\mathbb Q}
\newcommand{\Qcal}{\mathscr Q}

\newcommand{\RR}{\mathbb R}

\newcommand{\Rq}{c} 
\newcommand{\Scal}{\mathscr{S}}
\newcommand{\TLB}{\Fcal} 
\newcommand{\TT}{\mathbb T}

\newcommand{\Uo}{\widehat{U}} 
 
\newcommand{\LL}{\mathbb{L}} 
\newcommand{\LV}{\mathbf{L}} 

\newcommand{\Vb}{\operatorname{Vb}}
\newcommand{\WP}{W\PP} 

\newcommand{\Wcal}{\mathscr W}
\newcommand{\XLB}{{L}}  

\newcommand{\XS}{\X^{\operatorname{sympl}}}
\newcommand{\Xcal}{\mathscr{X}}

\newcommand{\X}{\Xcal}
\newcommand{\LGQcal}{\mathscr{Q}}

\newcommand{\ZZ}{\mathbb Z}

\newcommand{\Zst}{\mathscr{Z}} 
\newcommand{\aff}{/_{\!\!\operatorname{aff}}}
\newcommand{\age}{\operatorname{age}} 
\newcommand{\alphab}{\boldsymbol{\alpha}}

\newcommand{\bGq}{\mathbf{\Gq}} 

\newcommand{\Gdeg}{\deg}

\newcommand{\bgamma}{\boldsymbol{\gamma}}

\newcommand{\bigK}{M} 
\newcommand{\bigN}{{K}} 

\newcommand{\cb}{x} 
\newcommand{\cf}{p} 
\newcommand{\chat}{\hat{c}}
\newcommand{\chiR}{\zeta}

\newcommand{\cjcl}[1]{[ \kern-.15em [ #1 ]\kern-.15em ]}

\newcommand{\crst}{\mathscr{C \kern-.25em R}} 

\newcommand{\dsand}{{\quad \text{and} \quad}} 
\newcommand{\e}{\mathbf{e}}

\newcommand{\efp}{{\mathfrak{e}'}}

\newcommand{\fie}{\varphi}
\newcommand{\genj}{\langle J \rangle} 

\newcommand{\git}[1]{{\!/\!\!/_{\kern-.2em #1}}} 
\newcommand{\sympl}[1]{{\!/\!\!/_{\kern-.2em #1}^{\operatorname{spl}}}} 

\newcommand{\hfrak}{\mathfrak{h}}

\newcommand{\kc}{\kk_{\Ccal}}  
\newcommand{\kk}{\omega} 
\newcommand{\pkk}{\mathring{\omega}}  
\newcommand{\klog}{\kk_{\log}}

\newcommand{\klogc}{\kk_{\log,\Ccal}}    
\newcommand{\pklogcp}{\pkk_{\log,\Ccal'}}    
\newcommand{\pklogct}{\pkk_{\log,\widetilde{\Ccal}}}    
\newcommand{\pklogceta}{\pkk_{\log,\Ccal_\eta}} 
\newcommand{\pklogc}{\pkk_{\log,\Ccal}}    

\newcommand{\littlem}{m} 
\newcommand{\lift}{\vartheta} 

\newcommand{\mmapvalue}{\tau}
\newcommand{\mmap}{\mu}
\newcommand{\mrkp}{y} 

\newcommand{\one}{\varphi}

\newcommand{\qmp}{\mathcal{Q}} 

\newcommand{\smalln}{{n}}

\newcommand{\spl}{\xi} 
\newcommand{\spn}{\varkappa}  

\newcommand{\thetab}{\bar{\theta}} 

\newcommand{\thetaweight}{e} 
\newcommand{\tmrkp}{\widetilde{\mrkp}}
\newcommand{\trp}{^{\mathsf T}}  

\renewcommand{\u}{\sigma} 
\newcommand{\ut}{\widetilde{\u}} 
\newcommand{\ufrak}{\mathfrak{u}}

\newcommand{\ve}{\varepsilon}

\renewcommand{\Re}{\mathfrak{Re}}

\renewcommand{\hom}{\operatorname{Hom}}

\newtheorem{thm}{Theorem}[subsection] 
\newtheorem{pro}[thm]{Proposition} 
\newtheorem{lem}[thm]{Lemma} 
\newtheorem{cor}[thm]{Corollary} 

\theoremstyle{definition} 
\newtheorem{defn}[thm]{Definition} 

\theoremstyle{remark} 
\newtheorem{rem}[thm]{Remark}
\newtheorem{exa}[thm]{Example}

\begin{document}

\title[A Mathematical Gauged Linear Sigma Model]{A Mathematical Theory of the\\ Gauged Linear Sigma Model}
\author{Huijun Fan, Tyler Jarvis and Yongbin Ruan}
\date{\today}
\thanks{The second author's work was partially supported by NSF grant DMS 1564502.The third author's work was partially supported by NSF grants DMS 1159265 and DMS 1405245.}
\begin{abstract}
We construct a mathematical theory of Witten's Gauged Linear Sigma Model (GLSM). Our theory applies to a wide range of examples, including many cases with non-Abelian gauge group.

Both the Gromov-Witten theory of a Calabi-Yau complete intersection $X$ and the
Landau-Ginzburg dual (FJRW-theory) of $X$ can be expressed as gauged linear sigma
models. Furthermore, the Landau-Ginzburg/Calabi-Yau correspondence can be interpreted as a
variation of the moment map or a deformation of GIT in the GLSM. 
This paper focuses primarily on the algebraic theory, while a companion article \cite{FJR:15} will treat the analytic theory.  
\end{abstract}

\maketitle
\setcounter{tocdepth}{1}
\tableofcontents

\section{Introduction}

In 1991 a celebrated conjecture of Witten \cite{Wit:91} asserted that
the intersection theory of Deligne-Mumford moduli space is governed by
the KdV hierarchy. His conjecture was soon proved by Kontsevich
\cite{Kon:92}. The KdV hierarchy is the first of a family of integrable
hierarchies (Drinfeld-Sokolov/Kac-Wakimoto hierarchies) associated to
integrable representations of affine Kac-Moody algebras. Immediately
after Kontsevich's solution of Witten's conjecture, a great deal of
effort was spent in investigating other integrable hierarchies in
Gromov-Witten theory. In fact, this question was very much in Witten's
mind when he proposed his famous conjecture in the first place. Around
the same time, he also proposed a sweeping generalization of his
conjecture \cite{Wit:92a, Wit:93}.  The core of his generalization is
a remarkable first-order, nonlinear, elliptic PDE associated to an
arbitrary quasihomogeneous singularity. It has the simple form
\begin{equation}\label{eq:Witten}
\bar{\partial}u_i+\overline{\frac{\partial W}{\partial u_i}}=0,
\end{equation}
where $W$ is a quasihomogeneous polynomial, and $u_i$ is interpreted
as a section of an appropriate orbifold line bundle on an orbifold Riemann
surface $\Ccal$.
  
During the last decade, a comprehensive treatment of the Witten
equation has been carried out, and a new theory like Gromov-Witten has been
constructed by Fan-Jarvis-Ruan \cite{FJR:07b, FJR:08, FJR:07a}. In
particular, Witten's conjecture for ADE-integrable hierarchies has
been verified (for the A series by \cite{Lee,FSZ}, and for the D and E series by \cite{FJR:07a}).

The so-called \emph{FJRW-theory} has applications beyond the ADE-integrable
hierarchy conjecture.  For example, it can be viewed as the
Landau-Ginzburg dual of a Calabi-Yau hypersurface
\[
X_W=\{W=0\} \subset \WP^{N-1}
\]
in weighted projective space.  The
relation between the Gromov-Witten theory of $X_W$ and the FJRW-theory
of $W$ is the subject of the Landau-Ginzburg/Calabi-Yau
correspondence, a famous duality from physics. More recently, the
LG/CY correspondence has been reformulated as a precise
mathematical conjecture, 
and a great deal of  progress has
been made on this conjecture \cite{CIR:12, ChiRu:10, ChiRu:11, PS, LPS}.

A natural question is whether the LG/CY correspondence can be
generalized to complete intersections in projective space, or more
generally to toric varieties. The physicists' answer is ``yes.'' In
fact, Witten considered this question in the early 90s \cite{Wit:92a}
in his effort to give a physical derivation of the LG/CY
correspondence. In the process, he invented an important model in
physics called the \emph{Gauged Linear Sigma Model (GLSM)}. From the
point of view of partial differential equations, the gauged linear
sigma model generalizes the Witten Equation \eqref{eq:Witten} to the
\emph{Gauged Witten Equation}
\begin{align}\label{eq:GaugedWitten}
\bar{\partial}_A u_i+\overline{\frac{\partial W}{\partial
    u_i}}&=0,\\ *F_A&=\mmap,
\end{align}
where $A$ is a connection of certain principal bundle, and $\mmap$ is the moment map of the
GIT-quotient, viewed as a symplectic quotient.  In general, both the
Gromov-Witten theory of a Calabi-Yau complete intersection $X$ and the
LG dual of $X$ can be expressed as gauged linear sigma
models. Furthermore, the LG/CY correspondence can be interpreted as a
variation of the moment map $\mmap$ (or a deformation of GIT) in the GLSM. 

The main purpose of this article and its companion \cite{FJR:15} is to construct a rigorous 
mathematical theory for the gauged linear sigma model. This new model has many applications
and some of them are already under way (see, for example \cite{RR,RRS,CJR}).

An important phenomenon in FJRW-theory is that the state space 
is a direct sum of \emph{narrow} and \emph{broad} sectors. The theory for the narrow sectors admits a purely algebraic construction in terms of cosection localization.  A similar situation holds for the GLSM, we have both broad and narrow sectors, but the narrow sectors are a subset of a larger class called \emph{compact type}.  We show in this paper how to use cosection localization to describe the GLSM algebraically for sectors of compact type. The analytic theory for more general broad sectors, and the relation to other approaches like \cite{TX}, will appear in a companion article \cite{FJR:15}.

\subsection{Brief description of the theory}

The \emph{input data} of our new theory is 
\begin{enumerate}
\item A finite dimensional vector space $V$ over $\CC$.
\item A reductive algebraic group $G\subseteq GL(V)$.
\item A $G$-character $\theta$ with the property $V^s_G(\theta)=V^{ss}_G(\theta)$. We say that it defines a \emph{strongly regular phase} $\X_\theta = [V\git{\theta}G]$. 
\item A choice of $\CC^*$ action ($R$-charge) on $V$ (denoted $\CC^*_R$ that  is \emph{compatible} with $G$, i.e, commuting with $G$-action, and such that $G\cap \CC^*_R = \langle J \rangle$ has finite order $d$.  Denote the subgroup of $\GL(V)$ generated by $G$ and $\CC^*_R$ by $\Gamma$.
\item A $G$-invariant \emph{superpotential} $W:V\to \CC$ of degree $d$ with respect to the $\CC^*_R$ action with the property that the GIT quotient $\crst_\theta$ of the critical locus $\crit(W)$ is compact.
\item A stability parameter $\ve>0$ in $\QQ$.  We also often write $\ve = 0+$ to indicate the limit as $\ve \downarrow 0$ or $\ve = \infty$ to indicate the limit as $\ve \to \infty$.
\item If $\ve>0$, a $\Gamma$ character $\lift$ that defines a \emph{good lift} of $\theta$, meaning that $\lift|_G= \theta$ and $V^{ss}_{\Gamma}(\lift)=V^{ss}_G(\theta)$.  The good lift provides some stability conditions for the moduli space. But in the case of $\ve=0+$ the good lift is unnecessary.
\end{enumerate}
With the above input data we construct a theory with following main ingredients: 

\begin{enumerate}
\item \emph{A state space}, which is the relative Chen-Ruan cohomology of the quotient $\X_\theta = [V\git{\theta} G]$ with an
additional shift by $2q$: 
\[
\Hcal_{W, G}=\bigoplus_{\alpha\in \QQ}\Hcal^{\alpha}_{W, G} = \bigoplus_{\Psi}\Hcal_{\Psi},
\]
where the sum runs over those conjugacy classes $\Psi$ of $G$ for which  $\X_{\theta,\Psi}$ is nonempty, and where 
\[
\Hcal^{\alpha}_{W, G}=H^{\alpha+2q}_{CR}(\X_{\theta},W^{\infty},
 \QQ)=\bigoplus_{{\Psi}} H^{\alpha-2\age{(\gamma)}+2q}(\X_{\Psi},
  W^{\infty}_{\Psi},\QQ),
\]
  and 
\[
\Hcal_{\Psi}=H^{\bullet+2q}_{CR}(\X_{\theta,\Psi},W^{\infty},
 \QQ)= \bigoplus_{\alpha\in \QQ} H^{\alpha-2\age{(\gamma)}+2q}(\X_{\theta,\Psi},
  W^{\infty}_{\Psi},\QQ).
\] 
Here $W^\infty = \Re(W)^{-1}(M,\infty) \subset [V\git{\theta} G]$ for some large, real $M$ (see Section~\ref{sec:state-space} for details).

\item 
\emph{The moduli space of LG-quasimaps:}

We denote by $\crst_\theta = [\crit_G^{ss}(\theta)/G] \subset  [V\!\git{\theta} G] = [V^{ss}_G(\theta)/G]$ 
the GIT quotient (with polarization $\theta$) of the critical locus of $W$. Our main object of study is the stack 
\[
\LGQ_{g,k}^{\ve,\lift}(\crst_{\theta}, \beta)
\] 
of \emph{Landau-Ginzburg quasimaps to $\crst_\theta$} 
(see the precise definition in Section~\ref{sec:moduli-space}).  The definition of the stack works equally well if the vector space $V$ is replaced by a closed subvariety, but we have focused on the case of $V = \CC^n$ for simplicity.

The main technical theorem of the article is
\begin{thm}
$\LGQ_{g,k}^{\ve,\lift}(\crst_\theta, \beta)$ is a proper Deligne-Mumford stack whenever $\crst_\theta$ is proper.
\end{thm}

\item \emph{A virtual cycle:}

$\LGQ_{g,k}^{\ve,\lift}(\crst_\theta, \beta)$ is naturally embedded into $\LGQ_{g,k}^{\ve,\lift}([V\git{\theta}G], \beta)$. The latter is not compact, but it admits a two-term perfect obstruction theory with a cosection whose
degeneracy locus is precisely $\LGQ_{g,k}^{\ve,\lift}(\crst_\theta, \beta)$. 

Applying Kiem-Li's theory of cosection localized virtual cycles \cite{KiLi:10}, and adapting the cosection introduced to the LG-model by Chang-Li-Li \cite{ChaLi:11,CLL:13} to $\LGQ_{g,k}^{\ve,\lift}([V\git{\theta}G], \beta)$, we can construct a virtual cycle
$$[\LGQ_{g,k}^{\ve,\lift}(\crst_\theta, \beta)]^{vir}\in H_*(\LGQ_{g,k}^{\ve,\lift}(\crst_\theta, \beta), \QQ)$$
with virtual dimension
$$
\dim_{vir}=\int_{\beta} c_1(V\git{\theta} G)+(\chat_{W,G}-3)(1-g)+k-\sum_i (\age(\gamma_i)-q),
$$
where $\chat_{W, G}$ is the \emph{central charge} (see Definition~\ref{def:chat}).

\item \emph{Numerical invariants:}
Once we construct the virtual cycle, we can define correlators  
$$\langle \tau_{l_1}(\alpha_1),\cdots, \tau_{l_k}(\alpha_k\rangle=\int_{[\LGQ_{g,k}^{\ve,\lift}(\crst_\theta,\beta)]^{vir}}\prod_i ev_i^*(\alpha_i)\psi^{l_i}_i,$$
where $\alpha_i \in \Hcal_{W,G}$ is of compact type (see Definition~\ref{def:compact-type}). 
One can define a generating function in the standard fashion.
These invariants satisfy the usual gluing axioms whenever all insertions are narrow.
\end{enumerate}

Almost all known examples in physics satisfy the conditions of our input data, and hence our theory applies.
We list several examples in the paper.  To keep this article to a reasonable length, we
will not spend much time on the many applications, but rather we focus on the algebraic construction of the theory in this paper and on the analytic construction in  its companion article \cite{FJR:15}.  

We should mention that the equation for the case $W=0$ has been studied
already in mathematics under the name of \emph{symplectic vortex
  equation}.  There is a large amount of work on this in both the
algebraic and symplectic setting. A particularly important piece of
work for us is the theory of stable quotients \cite{MOP:11} and stable
quasimaps \cite{CCFK:14,CFKM:11,CFKi:10, Kim:11}. 
In fact, our new theory can be treated as a unification of
FJRW-theory with stable quasimaps. 

There are two important special cases which we use 
to check the consistency of our theory. The first one is the theory of stable maps with $p$-fields
by Chang-Li \cite{ChaLi:11}, which corresponds to the geometric phase of our theory with an $\ve=\infty$
stability condition. The other one is the 
\emph{hybrid model}  of Clader 
\cite{Cla:13, ClaDis:14}. Unfortunately, that hybrid model only works for
a very restrictive situation.  The
theory we describe here corresponds to a much more general situation,
including complete intersections of toric varieties and even quotients by 
non-Abelian groups. But an understanding of the failure of the hybrid model for general complete intersections
motivated much of our construction. In this article, we will focus on the sectors of compact type 
and our construction will be completely algebraic. 
Finally, the virtual cycle construction relies on Kiem-Li's theory of cosection localized virtual cycles \cite{KiLi:10}, and using the cosection introduced to the LG-model by Chang-Li-Li \cite{CLL:13}.

The GLSM can be viewed as a generalization of FJRW-theory from a hypersurface with a finite Abelian gauge group $G$ to more general spaces with an arbitrary reductive gauge groups.

The results of this article were first announced by the second author on the Workshop of Geometry and Physics of Gauged Linear Sigma Model
in March, 2013 in Michigan. In the lecture, the second author gave a complete construction of the moduli space. The only thing missing
was the full detail of the proof of various properties of the moduli space. We apologize for the long delay in producing those details.  

\subsection{Acknowledgments}
The second author thanks Emily Clader, Dan Edidin and Bumsig Kim for helpful conversations. The third author would like to thank Kentaro Hori for 
many helpful conversations on GLSM, Huai-Liang Chang, Jun Li, 
Wei-Ping Li for many helpful discussions on the cosection technique, and Emily Clader for helpful discussions on the hybrid model. A special thank goes to E.~Witten for introducing us to the gauged Witten equation and for many
insightful conversation over the years. Finally, we thank the referee for many helpful suggestions, which greatly improved the paper.

\section{Brief Review of FJRW-theory}
In this section, we review the basic elements of FJRW theory. Our
new generalization will follow the blueprint of this older
case closely.

\subsection{The Basic Construction}
The basic starting point is a $\CC^*$-action on $\CC^N$ with positive weights $(c_1,\dots, c_N)$ and a nondegenerate polynomial $W\in \CC[x_1,\dots, x_N]$ of degree $d>1$ with respect to the $\CC^*$-action. We also choose a \emph{gauge group} $G$ of diagonal symmetries of $W$.  We think of both $\CC^*$ and $G$ as subgroups of $\GL(N,\CC)$.  Let $J = (\exp(2\pi i c_1/d), \dots, \exp(2\pi i c_N/d)) \in \CC^* \subset \GL(N,\CC)$.  We require that $\CC^* \cap G = \langle J \rangle$.

In order for the Witten equation~\eqref{eq:Witten} to make sense, we work with roots of the log-canonical bundle 
\[
\klogc=\kc\left(\sum_{i=1}^k \mrkp_i\right).
\]
Specifically, we work on the space of $W$-curves, which are tuples  
\[
\big( \Ccal, \fie_j\colon \Lcal_j^{\otimes {d}}\to \klogc^{\otimes
  {c}_j}\big) \qquad \text{ for $j\in \{1,\dots,N\}$},
\]  
where $\Ccal$ is a stable orbifold curve, each $\Lcal_j$ is an orbifold line bundle on $\Ccal$, and  each  $\fie_j\colon \Lcal_j^{\otimes {d}}\to \klogc^{\otimes{c}_j}$ makes $\Lcal_j$ in to a $d$th root of the $c_j$th power of $\klogc$. 
Some additional conditions are also required of the $W$-structure, namely
\begin{enumerate}
\item\label{cond:faithful} At each point $\mrkp$ the induced
  representation $\rho_p:G_\mrkp \to (\CC^*)^N$ of the local group $G_\mrkp$
  of $\Ccal$ at $\mrkp$ on the sum $\bigoplus_{i=1}^N\Lcal_i$ is faithful.
\item If $s$ is the number of
  monomials $W_1,\dots, W_s$ of $W$, then for each $i=1,\dots, s$ the
  isomorphisms $\{\fie_j\}_{i=1}^N$ induce isomorphisms:
\[
W_i(\Lcal_1,\dots, \Lcal_N)=\bigotimes_{j=1}^N \Lcal_j^{\otimes
  e_{ij}} \rTo^{\sim} \klogc,
\] where the $e_{ij}$ are the exponents of  $W_i$.
\end{enumerate}

Let $\Wcal_{g,k}^{W,G}$ be the stack of stable $W$-curves with the property that at each marked point $\mrkp$ the image of the local group $G_\mrkp$ under the representation $\rho_\mrkp:G_\mrkp\to (\CC^*)^N$ lies in $G$.  In the formulation we have given here, FJRW theory naturally corresponds to the orbifolded Landau-Ginzburg A-model for the superpotential $W$ on the orbifold $[\CC^N/\Gmax]$.  As we will describe below, it is possible to generalize it to $[\CC^N/G]$ for any subgroup $G$ containing the element
$J$.  But the current theory does not work for any group smaller than
$\genj $ in any generality.

A marked point $\mrkp_j$ of a $W$-curve is called \emph{narrow} if the
fixed point locus $\fix(\rho_\mrkp(G_\mrkp)) \subseteq \CC^N$ is just $\{0\}$. The
point $\mrkp_j$ is called \emph{broad} otherwise, and any coordinates $z_i$
for $\CC^N$ fixed by $G_\mrkp$ are called broad variables.

There are several natural morphisms of $\Wcal_{g,k}^{W,G}$ analogous to the
morphisms of the stack $\overline{\Mcal}_{g,k}(X,\beta)$ of stable
maps, including a \emph{stabilization} map.  Forgetting the
$W$-structure and the orbifold structure gives a
morphism $$st:\Wcal_{g,k}^{W,G} \rightarrow \overline{\Mcal}_{g,k}.$$
A key result in the theory states that 
$\Wcal_{g,k}^{W,G}$ is a compact, smooth complex orbifold with projective
  coarse moduli space, and $st$ is a finite morphism (but not
  representable).

\subsection{The Polishchuk-Vaintrob Construction}

\ 

Polishchuk and Vaintrob \cite{PoVa:11} have given an alternative formulation for the $W$-structures in terms of principal bundles.  Although it is maybe not quite as easy to see how to define the Witten equation in this construction, it simplifies the description of the stack of $W$-curves and has the advantage of making clear that the resulting stacks depend only on the (finite, Abelian) group $G$ and not on the superpotential $W$.  This construction also inspires part of our generalization to the more general theory for arbitrary (infinite and possibly non-Abelian) groups.

Let $G\subset \Aut(W)$ be a finite subgroup containing $J$, and let $\Gamma$ be the
subgroup of $(\CC^*)^N$ generated by $G$ and
$\CC^*_R=\{(\lambda^{c_1}, \dots, \lambda^{c_N}) | \lambda\in
\CC^*\}$, where this $\CC^*_R$ corresponds the
quasihomogeneity of $W$.  It is easy to see that
\begin{equation}\label{eq:GintersectCR}
G\cap \CC^*_R = \genj .
\end{equation}
We can define a surjective homomorphism
  $$\chiR: \Gamma\rightarrow \CC^*$$ by sending $G$ to $1$ and
$(\lambda^{c_1}, \dots, \lambda^{c_N})$ to
$\lambda^d$. Equation~\eqref{eq:GintersectCR} shows that the map $\chiR$
is well-defined and that $\ker(\chiR)=G$.  Let $\pklogc$ denote the
principal $\CC^*$-bundle associated to $\klogc$.
   
\begin{defn}
  A $\Gamma$-structure on an orbicurve $\Ccal$ is
\begin{enumerate}
\item A principal $\Gamma$-bundle $\Pcal$ on $\Ccal$ such that the
  corresponding map $\Ccal \to \B\Gamma$ to the classifying stack
  $\B\Gamma$ is representable.
\item A choice of isomorphism $\spn:\chiR_*\Pcal \cong
  \pklogc$.  Here $\chiR_*\Pcal$ denotes the principal $\CC^*$
  bundle on $\Ccal$ induced from $\Pcal$ by the homomorphism $\chiR$.
\end{enumerate}
\end{defn}
An equivalent way to state (2) is to recognize that the homomorphism
$\chiR$ induces a morphism of stacks $\BchiR:\B\Gamma \to \B \CC^*$
and (2) is equivalent to the requirement that the composition $\BchiR
\circ \Pcal:\Ccal \to \B\CC^*$ be equal to the morphism of stacks
$\Ccal \to \B\CC^*$ induced by the line bundle $\klogc$.

Let's match this new definition with the definition of a $W$-structure. The
projection $\pi_i:\Gamma \subseteq (\CC^*)^N \to \CC^*$ to the $i$th
factor for each $i\in \{1,\dots N\}$ defines a collection of line
bundles $(\Lcal_1, \cdots, \Lcal_N)$. It is easy to check that
$\pi^d_i=\chiR^{c_i}$. And thus we have
   $$\Lcal^d_i=\omega^{c_i}_{\Ccal, \log}.$$
   
Let $W=\sum_j W_j$. We want to show that for each $j\in \{1,\dots,
N\}$ we have $W_j(\Lcal_1, \cdots, \Lcal_N)=\omega_{\Ccal,log}$.  The
monomial $W_j$ induces a homomorphism $(\CC^*)^N\rightarrow \CC^*$. By
our initial assumptions, we have $W_j|_G=1$. Therefore, $W_j:
\Gamma/G\rightarrow \CC^*$.  By checking $W_j$ on the subgroup
$\CC_R^* = \{(\lambda^{c_1}, \dots, \lambda^{c_N})\}$ we can easily
show that the above homomorphism is an isomorphism. Hence, $W_j(\pi_1,
\cdots, \pi_N)=\chiR$. This implies that $W_j(\Lcal_1, \dots,
\Lcal_N)=\omega_{C, \log}$.
   
Let $G_{\mrkp_i}$ be the local group of $\Ccal$ at the marked point $\mrkp_i$
the morphism $\Ccal \to \B\Gamma$ implies that each $G_{\mrkp_i}$ has a
homomorphism to $\Gamma$.  Let $\gamma_{\mrkp_i}$ be the canonical
generator of $G_{\mrkp_i}$. Its image $(\gamma_1, \dots, \gamma_N)$ in
$(\CC^*)^N$ gives us the familiar presentation of the local group.
The fact that $\omega_{C, \log}$ has no orbifold structure implies that
$G_{\mrkp_i}$ actually maps to $\ker(\chiR)=G\subset \Gamma$. And
representability of the morphism $\Ccal \to \B\Gamma$ implies that the
map $G_{\mrkp_i} \to \ker(\chiR)=G$ is injective, so we have $(\gamma_1,
\dots, \gamma_N) \in G \subset (\CC^*)^N$.

A complete proof of the equivalence of this definition with our original definition is given in \cite[Prop 3.2.2]{PoVa:11}.

\subsection{The Virtual Cycle}

A choice of W-structure does not solve the problem completely.
Suppose that $u_i\in \Omega^0(\Lcal_i)$ and $\Lcal_1, \dots, \Lcal_N$
is a W-structure. Then,
\[
\bar{\partial}u_i\in \Omega^{0,1}(\Lcal_i), \overline{\frac{\partial
    W}{\partial u_i}}\in \Omega^{0,1}_{\log}(\bar{\Lcal}^{-1}_i),
\]
where $\Omega_{\log}^{0,1}$ means a $(0,1)$-form with possible
singularities of order $1$.  So the Witten equation~\eqref{eq:Witten} has singular
coefficients! This is a fundamental phenomenon for the application of
the Witten equation. One of the most difficult conceptual advances in
the entire theory was to generate the A-model state space from the
study of the Witten equation.  Now it is understood that the
singularity of the Witten equation is the key. Unfortunately, the
appearance of singularities makes the Witten equation very difficult
to study analytically. The general construction of the FJRW virtual cycle is analytic.
However, there is a subsector (the \emph{narrow sector}) which admits a purely
algebraic treatment in terms of cosection localization.  

 In any case, our treatment of the moduli space of solutions of Witten equation allowed us to construct a
 virtual cycle 
 $$[\Wcal_{g,k}(\gamma_1, \dots, \gamma_k)]^{vir}\in
 H_*(\Wcal_{g,k}(\gamma_1, \dots, \gamma_k), \QQ)\otimes \prod_i
 H_{N_{\gamma_i}}(\CC^{N_{\gamma_i}}, W^{\infty}_{\gamma_i},
 \QQ)^{{G}}.
$$  
This naturally leads us to the state space   $$\Hcal_{W,G}=\prod_i H^{N_{\gamma_i}}(\CC^{N_{\gamma_i}},
 W^{\infty}_{\gamma_i}, \QQ)^{{G}}.$$
The space $\Hcal_{W,G}$ in FJRW theory is analogous to the cohomology of the target in Gromov-Witten theory. 

Pushing down
$\left[\Wcal_{g,k}(\bgamma)\right]^{vir}$ to the stack of stable
curves $\overline{\Mcal}_{g,k}$ and Poincar\'e dualizing 
$$ \Lambda^W_{g,k}(\alpha_1, \dots, \alpha_k) :=
 \frac{|{G}|^g}{\deg(st)} PD
 st_*\left(\left[\Wcal_{g,k}(\bgamma)\right]^{vir} \cap \prod_{i=1}^k
 \alpha_i \right).$$
gives a cohomological field theory, in the sense of Kontsevich and Manin.

The general construction of \cite{FJR:07a} is analytic.  However,
for narrow sectors the Witten equation has only the zero
solution. This leads to an algebraic treatment in this subsector.

In this case, the Witten equation breaks into two separate equations
$$\bar{\partial} u_i=0, \ \overline{\frac{\partial W}{\partial
    u_i}}=0.$$ The first equation says that $u_i$ is a holomorphic
section. The second equation implies all the $u_i$ vanish, by
nondegeneracy of $W$.  In this case, the virtual cycle can be
formulated in terms of the topological Euler class.  Our original construction of this cycle was not quite algebraic because we used the complex
conjugate at one point.  The effort to remove it leads to several
algebraic treatments, including those of Polishchuk-Vaintrob
\cite{PoVa:01,PoVa:11}, Chiodo \cite{Chi:06a} and Chang-Kiem-Li-Li
\cite{KiLi:10,ChaLi:11,CLL:13}.  We will use many of their ideas in this 
paper to construct the virtual cycle for the compact type sector of the gauged linear sigma model.

\section{Gauged Linear Sigma Model (GLSM)}

We will describe a broad generalization of FJRW theory and use it to
provide a mathematical theory of gauged linear sigma models.  

\subsection{Quotients}
Geometric invariant theory (GIT) is a fundamental tool in our constructions. It is also often useful to describe quotients in terms of symplectic reductions.  Here we briefly fix notation and conventions and also describe the connection between the GIT and symplectic pictures.

\subsubsection{Geometric Invariant Theory}

Unless otherwise indicated, we will always work with a reductive algebraic group $G$ acting on a finite-dimensional vector space $V$.  For a given character $\theta:G \to \CC^*$, we write $\LV_\theta$ for the
line bundle $V\times \CC$ with the induced linearization. 

We call a point of $v\in V$ \emph{stable} with respect to the linearization $\theta$ (or $\theta$-stable) if 
\begin{enumerate}
\item The stabilizer $\stab_G(v) = \{g\in G \mid gv = v\}$ is finite, and 
\item There exists a $k>0$ and an $f\in H^0(V,\LV_\theta^k)$ such that $f(v) \neq 0$, and such that every $G$-orbit in $D_f = \{f \neq 0\}$ is closed.\footnote{This always holds if all the points in $D_f$ have finite stabilizer.}
\end{enumerate}
Mumford-Fogarty-Kirwan\cite{MFK:94} use the name \emph{properly stable} to describe what we call stable.

For a closed $G$-invariant subvariety $Z\subset V$, we are interested in several different quotients:
\begin{itemize}
\item $[Z/G]$ the stack quotient of $Z$ by $G$.   
\item $Z\aff G$, the \emph{affine quotient} given by $Z\aff G = \Spec(\CC[Z^*]^G)$, where $\CC[Z^*]$ is the ring of regular functions on $Z$.
\item $[Z\git{\theta} G] = [Z^{ss}_G(\theta)/G]$, the GIT quotient stack.
\item $Z\git{\theta} G = \Proj_{Z\aff G}\left(\bigoplus_{k\ge 0} H^0(Z,\LV_\theta^k)^G\right)$, the underlying coarse moduli space of $[Z\git{\theta} G]$.
\end{itemize}
In this paper we are primarily concerned with characters $\theta\in \Gh_\QQ = \Hom(G,\CC^*)\otimes \QQ$ such that every semistable point of $Z$ is stable: 
$Z^{s}_G(\theta) = Z^{ss}_G(\theta)$. This implies that the GIT quotient is a Deligne-Mumford stack.
\begin{defn}\label{def:StronglyRegular}
We say that $\theta\in\Gh_\QQ$ (or the corresponding linearization $\LV_\theta$) is \emph{strongly regular} on $Z$  
if $Z^{ss}_G(\theta)$ is not empty and $Z^{s}_G(\theta) = Z^{ss}_G(\theta)$.
\end{defn}

The linearization $\LV_{\theta}$ induces a line bundle on
$[Z\git{\theta} G]$, which we denote by $\XLB_{\theta}$.  GIT guarantees that there is a line bundle $M$ on $Z\git{\theta} G$ that is relatively ample over the affine quotient and that pulls back to $\XLB_\theta^k$ for some positive integer $k$.

For a fixed $Z$, changing the linearization gives a different
quotient.  The space of (fractional) linearizations is divided into chambers, and
any two linearizations lying in the same chamber have isomorphic GIT
quotients.  We will call the isomorphism classes of these quotients
\emph{phases}.  If the linearizations lie in distinct chambers, the
quotients are birational to each other, and are related by flips
\cite{Tha:96, DoHu:98}.  This variation of GIT and the way the quotients change when crossing a wall of a chamber is important in the theory of the gauged linear sigma model.

\subsubsection{Symplectic Reductions}
It is often useful to think of GIT quotients as symplectic
reductions.  Take $Z \subseteq \CC^{\smalln}$ with the 
standard K\"ahler form $\omega=\sum_i dz_i\wedge
d\bar{z}_i.$ Since $G$ is reductive, it is the complexification of a
maximal compact Lie subgroup $H$, acting on $Z$ via a faithful unitary
representation $H\subseteq U({\smalln})$. Denote the Lie algebra of $H$ by
${\hfrak}$.

We have a Hamiltonian action of $H$ on $Z$ with moment map $\mmap_Z: Z
\to {\hfrak}^*$ for the action of $H$ on $Z$, given by
\[
\mmap_Z(v)(Y) = \frac{1}{2} \overline{v}\trp Y v = \frac{1}{2}
\sum_{i,j \le {\smalln}} \bar{v}_i Y_{i,j} v_j
\]
for $v \in Z$ and $Y\in {\hfrak}$.  If $\mmapvalue \in {\hfrak}^*$ is a
 value of the moment map, then the locus $\mmap^{-1}(H\mmapvalue)$ is
an $H$-invariant set, and the symplectic orbifold quotient of $Z$ at
$\mmapvalue$ is defined as
\[
\left[Z\sympl{\mmapvalue} H\right] = \left[ \mmap_Z^{-1}(H\mmapvalue)/H
  \right] = \left[ \mmap_Z^{-1}(\mmapvalue)/H_\mmapvalue
  \right],
\]
where $H_\mmapvalue$ is the stabilizer in $H$ of $\mmapvalue$.
The symplectic quotients 
$\left[Z\sympl{\mmapvalue} H\right]$ depend on a choice of 
$\mmapvalue\in \hfrak^*$. As in the GIT case, there is a chamber structure for the image
of $\mmap$ such that
 \begin{enumerate}
 \item For any two regular values $\mmapvalue$ and $\mmapvalue'$ in the same
   chamber, the quotients 
   $\left[Z\sympl{\mmapvalue} H\right]$ and $\left[Z\sympl{\mmapvalue'} H\right]$
are isomorphic.
 \item The quotients associated to regular values in different
   chambers are birational to each other.
 \end{enumerate}
(See \cite[\S8]{MFK:94} for details).

\subsubsection{Relation Between GIT and Symplectic Quotients}

Although we are primarily interested in GIT quotients, identifying the phases is sometimes easier in the symplectic setting, so it is useful to understand the relation between the two formulations.  

To do this, we first observe we can $G$-equivariantly compactify the vector space $V$ by embedding it into $\overline{V} = \PP(V\oplus \CC)$ in the obvious way, with the trivial $G$-action on the factor $\CC$.  For any integer $n>0$, define a $G$-linearization on $\overline{V}$ by letting $G$ act on the fiber of $\Ocal(n) = \Ocal_{\overline{V}}(n)$ by multiplication by $\theta$. 

\begin{pro}
For each $n>0$, let $\overline{V}^{ss}_{\theta, n}$ denote the semistable locus in $\overline{V}$ with respect to the previously defined linearization on $\Ocal(n)$.  There exists a finite $M>0$ such that $V \cap \overline{V}^{ss}_{\theta, n}$ is equal to the affine semistable locus $V^{ss}_G(\theta)$ for all $n\ge M$.
\end{pro}
\begin{proof}
We have $V^{ss}_G(\theta) =\bigcup_{t} D_t$, where the union runs over all $G$-invariant global sections $t$ of $\LV^k$ for all $k>0$, and $D_t$ is the distinguished open set $\{x \mid t(x) \neq 0\}$.  Any such $t$ corresponds to a polynomial $g\in \CC[V^*]$ such that $G$ acts on $g$ as $\theta^{-k}$.  

Similarly, we have $\overline{V}^{ss}_{\theta,n}=\bigcup_{s} D_s$, where the union runs over all $G$-invariant sections $s$ of $\Ocal(kn)$ for all $k>0$.  Any such $s$ corresponds to a polynomial $f\in \CC[V^*]$ of degree at most $kn$ such that $G$ acts on $f$ as $\theta^{-k}$.  Clearly, every such section $s$ defines a section of $\LV_\theta^{k}$ on $V$, and hence $(\overline{V}^{ss}_{\theta,n} \cap V) \subset V^{ss}_G(\theta)$ for every $n>0$.  

Conversely, since $V$ is quasicompact in the Zariski topology, we may choose a finite number of $G$-invariant sections $t_1,\dots, t_m$ such that $V^{ss}_G(\theta) =\bigcup_{i=1}^m D_{t_i}$.  For each $i$, let $g_i\in \CC[V^*]$ be the polynomial corresponding to the section $t_i$ of $\LV_{\theta}^{k_i}$, and let $d_i$ be the degree of $g_i$.  Letting $M = \max(d_1/k_1,\dots, d_m/k_m)$ implies that each $g_i$ has degree no more than $M k_i$ and thus defines a $G$-invariant section of $\Ocal(n k_i)$ for every $n\ge M$.  Therefore, $V^{ss}_G(\theta) \subset (\overline{V}^{ss}_{\theta,n} \cap V)$ for all $n\ge M$. 
\end{proof}

We can also extend the action of $H$ to  a Hamiltonian action of $H$ on $\overline{V}$ with an extended moment map $\widetilde{\mmap}: \overline{V}
\to {\hfrak}^*$ such that $V \cap \mmap^{-1}(\mmapvalue) = V \cap \widetilde{\mmapvalue}$.   

To relate the GIT quotient $[\overline{V}\git{} G]$ to the symplectic quotient
$\left[\overline{V}\sympl{\mmapvalue} H\right]$
we use the Kempf-Ness Theorem  and the so-called
\emph{shifting trick}.
For our purposes, these can be combined into the following theorem, which is essentially  \cite[Thm 2.2.4]{DoHu:98}.
\begin{thm}
Taking derivations of the character $\theta$ defines a weight
$\mmapvalue_\theta\in\hfrak^*$ and a very ample line bundle $\LGB_\theta$ on
$G/B$ for some Borel subgroup $B$ of $G$.  The manifold $G/B$
inherits the Fubini-Study symplectic structure via the projective
embedding of $G/B$ defined by $\LGB_\theta$.  Let
$\mmap_{\LGB_\theta}:G/B \to \hfrak^*$ be the corresponding moment map.
This also defines a line bundle $pr_2^*(\LGB_\theta)$ on $\overline{V}\times G/B$
and a moment map $\mmap_{\theta}: \overline{V} \times G/B \to \hfrak^*$ by
$\mmap_{\theta}(v,gB) = \mmap_V(v) + \mmap_{\LGB_\theta}(gB)$.  We have
\[
[\overline{V}\times G/B)\git{pr_2^*(\LGB_\theta)}G] = [\mmap_\theta^{-1}(0)/H] 
= [\mmap^{-1}(-\mmapvalue_\theta)/H_{-\mmapvalue_\theta}] = \left[\overline{V} \sympl{-\mmapvalue_\theta} H\right]
\]
where $H_{-\mmapvalue_\theta}$ is the stabilizer of $-\mmapvalue_\theta$ in $H$.  This can be
extended to rational characters $\theta \in \Gh_\QQ$
by taking appropriate powers of the corresponding line bundles.
\end{thm}
\begin{cor}
Whenever the coadjoint orbit of $\tau_\theta$ in $\hfrak^*$ is trivial, so that $G/B$ is a single point (e.g., in the case that $G$ is Abelian), then we have $pr_2^*\LGB_\theta = \LV_\theta$
and
\[
\left[Z\sympl{-\mmapvalue_\theta} H\right]
= [Z\git\theta G ]
\]
for any $G$-invariant quasiprojective subvariety $Z\subseteq \overline{V}$.
\end{cor}
For us the main use of this corollary is that it allows us to identify the phases of the GIT quotient by examining the critical points of the moment map. 
  
\subsection{GLSM}\label{sec:GLSM}

The Gauged Linear Sigma Model (GLSM) requires an additional $\CC^*$ action on $V$ called the \emph{R-charge} and a superpotential on the 
quotient.  We will be especially interested in the critical locus of the superpotential.

Our basic setup is the following. Let $V$ be an $\smalln$-dimensional vector space over $\CC$, and let $G\subset \GL(V)$ be a
reductive algebraic group over $\CC$  with identity component $G_0$
such that $G/G_0$ is finite.  We call $G$ the \emph{gauge group}.  If the gauge group action on $V$ factors through $\SL(V)$ then we say that it satisfies the \emph{Calabi-Yau condition.}

Assume that $V$ also admits a $\CC^*$ action $(z_1, \cdots, z_{\smalln})
\rightarrow (\lambda^{c_1}z_1, \cdots, \lambda^{c_{\smalln}}z_{\smalln})$, which we
denote by $\CC^*_R$.  We think of $\CC_R^*$ as a subgroup of
$\GL(V,\CC)$.  This means we require $\gcd(c_1,\dots,c_\smalln) = 1$.  Unlike the case of FJRW theory, we allow the weights $c_i$ of $\CC^*_R$ to be zero or negative.

\begin{defn}
Fix a polynomial $W{\colon}V\to \CC$ of degree $d\neq 0$
with respect to the $\CC^*_R$ action (i.e., quasihomogeneous) and
invariant under the action of $G$.  The polynomial $W$ will be called the \emph{superpotential} for our theory.
\end{defn}   

\begin{rem}
For any strongly regular phase $\theta$, the complex dimension of $\X_\theta = [V\!\git{\theta} G]$ is $n-\dim(G)$.
\end{rem}

\begin{defn}\label{def:chat}
Let $N = n-\dim(G)$.  We define the \emph{central charge} of the theory for the pair $(W,G)$ to be
\begin{equation}\label{eq:def-chat}
\chat_{W,G}=N-2\sum_{j=1}^\smalln c_j/d.
\end{equation}
And we define
\begin{equation}\label{eq:J}
J = (\exp(2\pi i c_1/d), \dots, \exp(2 \pi i c_{\smalln}/d)),
\end{equation}
 which is an automorphism of $W$ of order $d$.  
 
 It will sometimes be convenient to write $q_i = c_i/d$ and $q = \sum_{i=1}^n q_i$ so that 
\[
\chat_{W,G}=N-2q \dsand J = (\exp(2\pi i q_1), \dots, \exp(2 \pi i q_{\smalln})).
\]
\end{defn}
Note that the $\CC^*_R$ action is closely related to what the physics literature calls
\emph{R-charge}.  More precisely, R-charge is the $\CC^*$-action given by the weights $(2c_1/d,\dots,2c_{\smalln}/d)$; but for our purposes, $\CC^*_R$ is more useful, and we will sometimes abuse language and call it the R-charge.
 
\begin{defn}
We say that the actions of $G$ and $\CC^*_R$ are \emph{compatible} if
\begin{enumerate}
\item They commute: $g r = r g$ for any $g\in G$ and any $r \in \CC^*_R$.
\item We have $G \cap \CC^*_R=\genj$.
\end{enumerate}
\end{defn} 
\begin{defn}
We define $\Gamma$ to be the subgroup of $\GL(V,\CC)$ generated by $G$
and $\CC_R^*$.  If $G, \CC^*_R$ are compatible, then every element $\gamma$ of
$\Gamma$ can be written as $\gamma = g r$ for $g\in G, r\in \CC^*_R$; that is,
\[
\Gamma = G \CC_R^*.
\]
\end{defn}

The representation $\gamma = gr$ is unique up to an element of $\genj
$. Moreover, there is a well-defined homomorphism 
\begin{align}\label{eq:defzeta}
\chiR{\colon} \Gamma = G\CC^*_R &\rightarrow \CC^*\\
 g(\lambda^{c_1}, \cdots, \lambda^{c_{\smalln}}) &\mapsto \lambda^d. \notag
\end{align}
We denote the target of $\chiR$ by $H = \chiR(\CC^*_R)= \CC^*$, to distinguish it from $\CC^*_R$. 
This gives the following
exact sequence:
\begin{equation}\label{eq:chiR-seq}
1 \rTo G \rTo \Gamma \rTo^{\chiR}H \rTo 1
\end{equation}
Moreover, there is another homomorphism 
\begin{align}\label{eq:defspl}
\spl{\colon}\Gamma &\rightarrow  G/\genj  \\
g r &\mapsto g\genj.\notag
\end{align}
This is also well defined, and gives another exact
sequence:
\[
1 \rTo \CC^*_R \rTo \Gamma \rTo^{\spl}G/\genj \rTo 1
\]

\begin{defn}\label{defn:WGnondegen}
Let $\theta:G\to \CC^*$ define a  strongly regular phase $\X_\theta = \left[V\!\git{\theta} G\right]$. 
The superpotential $W$ descends to a holomorphic function $W{\colon} \X_{\theta}\rightarrow
\CC$.  Let  $\crit^{ss}_G(\theta) = \{v\in V^{ss}_G(\theta) \mid \frac{\partial W}{\partial x_i}=0 \text{ for
  all $i =1, \dots, {\smalln}$}\} \subset V^{ss}$ denote the semistable points of the critical
locus.  The group $G$ acts on $\crit^{ss}_G(\theta)$ and the stack quotient is
\[
\crst_\theta = 
\left[\crit^{ss}_G(\theta)/G\right] = \{x\in \X_{\theta} \mid d W = 0 \}\subset \X_{\theta},
\]
where $d W\colon T\X_{\theta} \to T\CC^*$ is the differential of $W$ on $\X_\theta$.  
We say that the pair $(W, G)$ is \emph{nondegenerate
  for $\X_{\theta}$} if the critical locus $\crst_\theta \subset \X_{\theta}$  is
  compact.
\end{defn}

\subsubsection{Characters, Lifts, and GIT Stability}

\begin{defn}
Given any $G$-character $\theta\in \Gh$, we say that a character $\lift\in \Gammah = \Hom(\Gamma,\CC^*)$ is a \emph{lift} of $\theta$ if its restriction to $G$ is equal to $\theta$:
\[
\lift |_{G} = \theta.
\]
\end{defn}

\begin{pro}\label{prop:lift-chars}
Given any character $\theta\in \Gh$, there is a lift of $\theta$ to some $\lift\in \Gammah$.  Composition of this lift with the inclusion $\CC^*_R \subset \Gamma$ induces a character $\CC^*_R \to \CC^*$.  

Given any two lifts $\lift,\lift'\in \Gammah$, the ratio $\lift^{-1}\lift'$ induces a character on $\CC^*_R$ of weight divisible by $d$, which factors through the composition $\CC^*_R \subset \Gamma \rTo^{\chiR}\CC^*$.  

Conversely, given any lift $\lift$ of $\theta$ and given any $\ell\in \ZZ$, there is a unique lift $\lift'$ of $\theta$ such that $\lift^{-1}\lift'$ induces a character on $\CC^*_R$ of weight $\ell d$.

Finally, the $d$th power  $\theta^d$ of any character $\theta$ factors through $G/\genj$, inducing a character $\thetab:G/\genj \to \CC^*$. This gives a lift of $\theta^d$ via $\Gamma \rTo^{\spl} G/\genj \rTo^{\thetab}\CC^*$, which we denote $\lift_0^d$.  The induced character on $\CC^*_R$ has weight $0$.  
\end{pro}
\begin{proof}
Given a character $\theta\in \Gh$, the element $J\in G\cap \CC^*_R$ must satisfy $\theta(J) = \exp(2\pi i a/d)$ for some $a\in \ZZ$.  For any $\gamma \in \Gamma$, write $\gamma = g r$ with $g\in G$ and $r = (\lambda^{c_1},\dots,\lambda^{c_n})\in \CC^*_R$.  Define $\lift(g r) = \theta(g)\lambda^a$.  This is a well-defined group homomorphism because the only possible ambiguity in the representation of $\gamma$ 
is due to elements in $G\cap \CC^*_R = \genj$.  That is to say, the only ambiguity is whether an element is written as $g \cdot rJ^k$ or $gJ^k \cdot r$.
We calculate:
\[
\lift(g \, rJ^k) = \theta(g) \lambda^a J^{ka} = \theta(gJ^k) \lambda^a = \lift(gJ^k\, r).
\]
So the lift $\lift$ is well defined.  This proves the existence of a lift of $\theta$.  

The ratio of any two lifts of $\theta$ is a lift of the trivial $G$-character. The induced character on $\CC^*_R$ must therefore be trivial on $J$, and hence must have weight divisible by $d$.  
Moreover, given any lift $\lift$ with $\CC^*_R$-weight $a$, we can define a new character by $\lift'(g r) = \theta(g) \lambda^{a+\ell d}$.  It is immediate that $\lift^{-1}\lift'$ induces a character on $\CC^*_R$ of weight $\ell d$

The final statement about $\theta^d$ is immediate from the definition.
\end{proof}

It will also be useful to consider fractional characters rather than just integral characters.
\begin{defn}
We write $\Gh_\QQ$ as a shorthand for $\Gh\otimes_\ZZ\QQ =\hom(G,\CC^*)\otimes_\ZZ\QQ$ and $\Gammah_\QQ$ as a shorthand for $\Gammah\otimes_\ZZ\QQ = \hom(\Gamma,\CC^*)\otimes_\ZZ\QQ$.
A \emph{lift of $\theta\in \Gh_\QQ$} is a fractional character $\lift\in \Gammah_\QQ$ such that $\lift|_G = \theta$.
 \end{defn}

\begin{cor}\label{cor:unique-lift}
Given any $\theta\in \Gh_\QQ$ and any $\xi \in \widehat{\CC^*_R}\otimes \QQ\cong\QQ$, there exists a unique lift $\lift \in \Gammah_\QQ$ of $\theta$ that induces  $\xi$.
\end{cor}

In the next proposition we list many of the properties of $\Gamma$- and $G$-actions and characters  that are relevant for our use of geometric invariant theory.  Many of these are simple, but we find it useful to state them explicitly.

\begin{pro}\label{thm:G-Gamma-ss}
\hfill
\begin{enumerate}
\item Given any character $\lift\in \Gammah$, the $\lift$-semistable locus $V^{ss}_\Gamma(\lift)$ for the $\Gamma$ action on $V$ is a subset of the $\lift|_G$-semistable locus $V^{ss}_G(\lift)$ for the $G$ action on $V$
\[
V^{ss}_\Gamma(\lift) \subseteq V^{ss}_G(\lift).
\] 
\item \label{item:chars-agree-on-G}For any two characters $\lift, \lift'\in\Gammah$ such that $\lift|_G = \lift'|_G$ we have 
\[
H^0(V,\LV_\lift^k)^G = H^0(V,\LV_{\lift'}^k)^G
\]
for every nonnegative integer $k$. Furthermore, the $G$-semistable loci agree:
\[
V^{ss}_G(\lift) = V^{ss}_G(\lift').
\]
\item \label{item:isotypes-d} For any character $\lift\in \Gammah$ and any nonnegative integer $k$, the group $\Gamma$ acts on the $G$-invariant space of sections $H^0(V,\LV_\lift^k)^G$, hence this space gives a representation of $\CC^*_R$ and can be decomposed into eigenspaces: 
\[
H^0(V,\LV_\lift^k)^G = \bigoplus_{\ell \in \ZZ} E_{\ell,\lift},
\]
where $\CC^*_R$ acts on $E_{\ell,\lift}$ with weight $\ell\in \ZZ$.

Moreover,  $d$ must divide $\ell$ for any nontrivial (nonzero) component $E_{\ell,\lift}$.

\item For any two characters $\lift, \lift'\in \Gammah$ that agree when restricted to $G$, and for any integer $\ell$, the eigenspace  
$E_{\ell,\lift}$ is equal to an eigenspace  
$E_{\ell',\lift'}$ for some integer $\ell'$.  That is, the decomposition into components is the same for $\lift$ and $\lift'$, but the weight of the $\CC^*_R$ action on each component depends on  
the choice of character.  

\item\label{item:extend-theta} For any character $\theta\in \Gh$ and any positive integer $k$, let $\lift$ be any lift of $\theta^k$ ($\lift$ is not necessarily equal to the $k$th power of a lift of $\theta$).
We have 
\[
H^0(V,\LV_\lift)^{\Gamma} \subseteq H^0(V,\LV_\theta^k)^{G}
\]
and 
\[
V^{ss}_\Gamma(\lift) \subseteq V^{ss}_G(\theta^k) = V^{ss}_G(\theta)
\]
\item\label{item:G-Gamma-ss} Given any character $\theta \in \Gh$, and any nonnegative integer $k$, the set of $G$-invariant sections $H^0(V,\LV^{k}_\theta)^{G}$ is the direct sum, over all $\lift$ lifting $\theta^k$, of the $\Gamma$ invariant sections:
\[
H^0(V,\LV^k_\theta)^{G} = \bigoplus_{\text{$\lift$ lifting $\theta^k$}}  H^0(V,\LV_\lift)^{\Gamma} 
\]

Moreover, the $\theta$-semistable locus $V^{ss}_G(\theta)$  is the union of all the semistable loci for the $\Gamma$ action with characters $\lift$, where $\lift$ ranges over all lifts of $\theta^k$ and $k$ ranges over all positive integers. 
\[
V^{ss}_G(\theta) = \bigcup_{k\in \ZZ^{>0}} \bigcup_{{\text{$\lift$ lifting $\theta^k$}}} V^{ss}_\Gamma(\lift). 
\]
\item\label{item:unstable-Gamma-inv} For any character $\theta\in \Gh$ the $G$-semistable locus $V^{ss}_G(\theta)$ and its complement, the $G$-unstable locus $V^{un}_G(\theta)$, are both preserved by $\Gamma$.   
\end{enumerate}
\end{pro}
\begin{proof}\hfill
\begin{enumerate}
\item This is immediate from the definition of semistable.
\item Again, this is immediate from the definitions.
\item The fact that $\Gamma$ acts on $H^0(V,\LV_\lift^k)^G$ is a straightforward computation which follows from the fact that the action of $\CC^*_R$ and $G$ commute. 

This implies, in particular, that $H^0(V,\LV_\lift^k)^G$ is a finite-dimensional representation of $\CC^*_R$ and can be decomposed into eigenspaces: 
\[
H^0(V,\LV_\lift^k)^G = \bigoplus_{\ell\ \in \ZZ} E_{\ell,\lift},
\]
where $\CC^*_R$ acts on $E_{\ell,\lift}$ with weight $\ell\in \ZZ$.

Finally, we note that if $f\in E_{\ell,\lift}$ is nontrivial, then since $f$ is $G$-invariant it must also be fixed by $J\in G$, but since $J\in \CC^*_R$ we must have $J\cdot f = J^{\ell}f = f$, and hence $d$ divides $\ell$.
 
\item 
The character $\lift:\Gamma \to \CC^*$ induces a character of $\CC^*_R$ with some weight $w\in \ZZ$.  The action of $r = (\lambda^{c_1},\dots,\lambda^{c_n})\in \CC^*_R$ on any section $f\in H^0(V,L_\lift^k)^G$ is given by $(r\cdot f)(v') = \lift(r^{-1}) f(rv')$ for every $v'\in V$, so for $f \in E_\ell$ we have $\lambda^{-w}f(rv') = \lambda^\ell f (v')$ and thus $f (rv') = \lambda^{\ell+w} f (v')$.  That is, there exists an integer $m$ such that $f(rv') = \lambda^m f(v')$ for every $v'\in V$.   This last result is independent of $\lift$.  Applying this in the case of $\lift'$,  the action of $r$ on $f$ is  $(r\cdot f)(v') = \lift'(r^{-1}) f(rv') = \lambda^{m-w'} f(v')$, where $w'$ is the weight of the character of $\CC^*_R$ induced by $\lift'$.  

Thus any eigenspace 
$E_{\ell,\lift}$ of $H^0(V,\LV_{\lift})^G$ is also an eigenspace 
of $H^0(V,\LV_{\lift'})^G$ but with possibly a different weight $\ell'$.

\item \label{item:varpi-varrho}
As a $G$-linearization, the line bundle $\LV_{\lift}$ is identical to $\LV_{\theta^k} = \LV_{\theta}^{k}$, so any $\Gamma$-invariant section $\sigma' \in H^0(V,\LV_{\lift})^{\Gamma}$ is also a $G$-invariant section of $H^0(V,\LV_{\theta}^{k})$, and hence $V^{ss}_\Gamma(\lift ) \subseteq V_G^{ss}(\theta^k) = V_G^{ss}(\theta)$.  

\item Given any lift $\lift$ of $\theta^d$ we have 
$H^0(V,\LV^k_\theta)^{G}  = \bigoplus_{\ell}  E_{d\ell,\lift}$.
For each $\ell$ let $\lift'$ be the character $\lift'(gr) = \lift(gr)\lambda^{-d\ell}$, where $g\in G$ and $r = (\lambda^{c_1},\dots,\lambda^{c_n}) \in \CC^*_R$.  This shows that $ E_{d\ell,\lift} =  E_{0,\lift'} = H^0(V,\LV_{\lift'})^{\Gamma}$.  
By Proposition~\ref{prop:lift-chars}, there is precisely one such lift for each $\ell$.  Thus we have \[
H^0(V,\LV^k_\theta)^{G} = \bigoplus_{\text{$\lift'$ lifting $\theta^k$}}  H^0(V,\LV_{\lift'})^{\Gamma} 
\]

Now we obviously have 
\[
\bigcup_{k\in \ZZ^{>0}} \bigcup_{{\text{$\lift$ lifting $\theta^k$}}} V^{ss}_\Gamma(\lift) \subseteq V^{ss}_G(\theta). 
\]
Conversely, given any $v\in  V^{ss}_G(\theta)$ and any $f\in H^0(V,\LV_{\theta}^k)^G$  for some positive integer $k$ with $f(v) \neq 0$, fix a choice of lift $\lift$ of $\theta^k$.  We can decompose $f$ as a sum $f_i + \cdots + f_n$ with each $f_j$ in the eigenspace 
 $E_{j,\lift}$.  Since $f$ does not vanish at $v\in V$, then $f_\ell (v) \neq 0$ for at least one integer $\ell$, so we may assume that, with respect to the character $\lift$, the group $\CC^*_R$ acts on $f$  by multiplication by $\lambda^\ell$ for some integer $\ell$.  By (\ref{item:isotypes-d})  the integer $\ell$ is divisible by $d$.  Choosing $\lift'(g r) = \lift(g r)\lambda^{-ell}$ shows that $f\in E_{0,\lift'} = H^0(V,\LV_{\lift'})^{\Gamma}$, so $v\in V^{ss}_{\Gamma}(\lift')$ as desired.

\item By (\ref{item:G-Gamma-ss}), given any $v\in V^{ss}_{G}(\theta)$ there is a $\lift$ lifting some $\theta^k$ such that $v\in V^{ss}_{\Gamma}(\lift)$.  But $V^{ss}_{\Gamma}(\lift)$ is preserved by $\Gamma$, and hence the $\Gamma$-orbit of $V$ must lie in $V^{ss}_{\Gamma}(\lift) \subset V^{ss}_{G}(\theta)$.
\end{enumerate}
\end{proof}

\begin{lem}\label{lem:three-lifts-suffice}
For any $\theta \in \Gh_\QQ$, let $\lift_-, \lift_0,\lift_+ \in \Gammah_\QQ$ be the unique lifts such that the induced characters of $\CC^*_R$ have weight $-1,0,1$, respectively.  We have 
\[
V^{ss}_{G}(\theta) = V^{ss}_{\Gamma}(\lift_-) \cup V^{ss}_{\Gamma}(\lift_0) \cup V^{ss}_{\Gamma}(\lift_+).
\]
\end{lem}
\begin{proof}

Taking powers as necessary, we may assume that $\theta\in \Gh$.  The algebra $\bigoplus_{k>0} H^0(V,\LV_{\theta}^k)^G = \bigoplus_{k>0} H^0(V,\LV_{\theta^k})^G$ is finitely generated, so there exists a finite set $f_1,\dots f_K \in \bigoplus_{k>0} H^0(V,\LV_{\theta^k})^G$ such that for every $v\in V^{ss}_G(\theta)$ at least one of the $f_i$ does not vanish on $v$.  We may further assume that each $f_i$ is an element of some $H^0(V,\LV_{\theta}^{k_i})^G$.  Taking appropriate powers, we may assume that $k_i$ is the same for all $i$ and is divisible by $d$.  Let $k = k_i$ be that common choice of $k_i$.  

Decompose  $H^0(V,\LV_{\theta^k})^G = H^0(V,\LV_{\lift_0^k})^G= \bigoplus_{\ell} E_{\ell,\lift_0^k}$ into isotypical pieces. By Proposition~\ref{thm:G-Gamma-ss}, each  $\ell$ is divisible by $d$.  When $\ell=0$, we have $E_{0,\lift_0^k} \subset H^0(V,\LV_{\lift_0}^k)^{\Gamma}$.  When $\ell >0$ it is straightforward to see that 
\[
E_{\ell,\lift_0^k} \subset H^0(V,\LV_{\lift_+}^{k\ell/d})^{\Gamma},
\]
and when $\ell<0$, we have 
\[
E_{\ell,\lift_0^k} \subset H^0(V,\LV_{\lift_-}^{-k\ell/d})^{\Gamma}.
\]
Therefore, we have
\[
V^{ss}_G(\theta) = V^{ss}_\Gamma(\lift_-) \cup V^{ss}_\Gamma(\lift_0) \cup V^{ss}_\Gamma(\lift_+) .
\]
\end{proof}

In many cases, however, we can do much better than the previous lemma. 
\begin{defn}\label{def:good-lift}
We say that a lift $\lift\in \Gammah_\QQ$ of  $\theta\in\Gh_\QQ$ is a \emph{good lift}, if $V^{ss}_{\Gamma}(\lift) = V^{ss}_G(\theta)$.
\end{defn}
Although not every $\theta\in \Gh$ has a good lift for every choice of ($G$-compatible) $\CC^*_R$-action, most of the examples we discuss in this paper have this property. 

\begin{rem}
Even when a point is both $\theta$-stable and $\lift$-semistable for some lift $\lift$ of $\theta$, the stabilizer in $\Gamma$ of the point will often be infinite.  Hence $\lift$-stability and $\theta$-stability are not easily related, even if $\lift$ is a good lift of $\theta$.
\end{rem}

\subsubsection{Input data}
From now on we will assume that we have the following input data:
\begin{enumerate}
\item A finite dimensional vector space $V$ over $\CC$.
\item A reductive algebraic group $G\subseteq GL(V)$.
\item A choice of $\CC^*_R$ action on $V$ which is compatible with $G$, and such that $G\cap \CC^*_R = \langle J \rangle$ has order $d$.
\item A $G$-character $\theta$ defining a strongly regular phase $\X_\theta = [V\git{\theta}G]$
\item A good lift $\lift$ of $\theta$, except when the stability parameter $\ve$ is $0+$ (otherwise any lift will work and all give the same results).
\item A nondegenerate, $G$-invariant superpotential $W:V\to \CC$ of degree $d$ with respect to the $\CC^*_R$ action.
\end{enumerate}

Here we provide one simple example to illustrate the ideas.  In Section~\ref{sec:examples} we consider many more important examples studied by Witten in \cite{Wit:97}.   The reader who wants to get right to the main results may skip this example on first reading; whereas, others may wish to look at the additional examples in \ref{sec:examples} before proceeding.

\begin{exa}\underline{Hypersurfaces:}\label{ex:hypersurf}

Suppose that $G = \CC^*$ and $F \in \CC[x_1,\dots,x_\bigN]$ is a nondegenerate quasihomogeneous polynomial of $G$-weights $(\Gq_1, \dots, \Gq_{\bigN})$ and
 total degree $\Gq$, as in FJRW-theory. Let
$$W=p F{\colon} \CC^{{\bigN}+1}\rightarrow \CC.$$ Here, we assign $G$-weight $-\Gq$
to the variable $p$, so that $W$ is $G$ invariant.  

The critical set of $W$ is given by the equations:
$$\partial_p W=F=0,\quad \partial_{\cb_i}W=p\partial_{\cb_i}F=0.$$ This implies that either $p\neq 0$ and $(\cb_1, \dots, \cb_{\bigN})=(0, \dots, 0)$
 or that $p=0$ and $F(\cb_1, \dots, \cb_{\bigN})=0$.  
 Suppose that $\Gq_i >0$ for $i=1, \cdots, {\bigN}$ and $\Gq>0$. Consider
the quotient of $\CC^{{\bigN}+1}$ by $G=\CC^*$ with weights $(\Gq_1, \dots, \Gq_{\bigN},-b)$. If 
 $b = \sum_{i=1}^{\bigN} \Gq_i$, then we have a Calabi-Yau weight system, but we do not assume that here. The affine
 moment map
 $$\mmap=\frac12 \left(\sum_{i=1}^{\bigN} \Gq_i|x_i|^2- b |p|^2\right)$$ is a quadratic function
 whose only critical point is at zero. Therefore, the only critical
 value is $\mmapvalue=0$ and there are two phases $\mmapvalue>0$ or $\mmapvalue<0$.

  \begin{enumerate}
\item[Case of $\mmapvalue>0$:]
We have
 $$\sum_i^{\bigN} \Gq_i |x_i|^2=b|p|^2+2\mmapvalue.$$ For each choice of
 $p$, the set of $(x_1, \dots, x_{\bigN}) \in \CC^{\bigN}$, such that
 $(x_1,\dots, x_\bigN, p) \in \mmap^{-1}(\mmapvalue)$, is a nontrivial
 ellipsoid $E$, isomorphic to $S^{2{\bigN}-1}$; and we
 obtain a map from the quotient $\XS_\mmapvalue$ to $\X_{base} =
 [E/U(1)] = \WP(\Gq_1,\dots,\Gq_{\bigN})$, corresponding to the maximal
 collection of column vectors $(\Gq_1, \cdots, \Gq_{\bigN})$ of $\GQ$.  The
 space $\XS_\mmapvalue$ can be expressed as the total space of the line
 bundle $\Ocal(-b)$ over $\XS_\mmapvalue$.  If $\sum_i \Gq_i = b$, this is the canonical bundle $\kk_{\WP(\Gq_1, \dots,  \Gq_{\bigN})}$.
 
Alternatively, we can consider the GIT quotient $\left[\CC^{{\bigN}+1}\git\theta
G\right]$, where $\theta$ has weight $-\thetaweight$, 
with $\thetaweight>0$.  One can easily see
that the $\LV_\theta$-semistable points are $((\CC^{\bigN}-\{\0 \})\times
\CC) \subset \CC^{\bigN} \times \CC = \CC^{{\bigN}+1}$, and the first projection
$pr_1:(\CC^{\bigN}-\{\0 \})\times \CC \to (\CC^{\bigN}-\{\0 \})$ induces the map
$[V\!\git{\theta} G] \to \WP(\Gq_1,\dots,\Gq_{\bigN})$.

Now we choose  $\CC^*_R$-weights $\Rq_{x_i}=0$ and $\Rq_{p}=1$, so that $W$ has $\CC^*_R$-weight $d=1$.  The element $J$ is trivial, and the group $\Gamma$ is a direct product $\Gamma \cong G \times \CC^*_R$, with $\spl$ and $\chiR$  just the first and second projections, respectively.

The critical locus $\crst_\theta = \{p=0 = F(\cb_1, \dots, \cb_{\bigN})\}$ is a degree-$b$ hypersurface in the image of the zero section of $[V\!\git{\theta} G] \cong \Ocal(-d) \to \WP(\Gq_1, \dots, \Gq_{\bigN})$. 
We call this phase the \emph{Calabi-Yau phase} or \emph{geometric phase}.

We wish to find a good lift of $\theta$.  To do this, consider any $v\in V^{ss}_G(\theta) = ((\CC^{\bigN}-\{\0 \})\times
\CC)$.  If $\ell$ is a generator of $\LV_\theta^*$ over $\CC[V^*]$ with $G$ acting on $\ell$ with weight $-\thetaweight$, if we choose the trivial lift $\lift_0$ of $\theta$, which corresponds to $\CC^*_R$ acting trivially on $\ell$, then a monomial of the form  $x_i^{k \thetaweight}\ell^k$ is $\Gamma$-invariant and does not vanish on points with $x_i\neq 0$, so every point of $\CC^{\bigN}\times
\CC$ with $x_i\neq 0$ is in $V^{ss}_\Gamma(\lift_0)$.   Letting $i$ range from $1$ to $\bigN$ shows that $V^{ss}_\Gamma(\lift_0) = V^{ss}_G(\theta)$. Thus  $\lift_0$ is a good lift of the character $\theta$.  It is easy to see that $\lift_0$ is the only good lift of $\theta$.

\item[Case of $\mmapvalue<0$:]
We have
\[
\mmap^{-1}(\mmapvalue) = \left\{(x_1,\dots,x_\bigN,p) \,\middle |\, \sum_{i=1}^{\bigN} \Gq_i|x_i|^2-\mmapvalue= b |p|^2\right\}
\]
For each choice of $x_1,\dots,x_{\bigN} \in \CC^{\bigN}$ the set of $p\in \CC$
such that $(x_1,\dots, x_\bigN, p) \in \mmap^{-1}(\mmapvalue)$ is the circle
$S^1\subset \CC$, corresponding to the maximal collection $(-b)$, and
we obtain a map $\XS_\mmapvalue \to [S^1/U(1)]$.  If we choose the basis of
$U(1)$ to be $\lambda^{-1}$, then $p$ can be considered to have
positive weight $b$. Moreover, every $p$ has isotropy equal to
the $b$th roots of unity (isomorphic to $\ZZ_b$). The quotient
$[S^1/U(1)]$ is $\WP(b)=\B\ZZ_b = [\pt/\ZZ_b]$.
 
In the GIT formulation of this quotient, with $\theta$ of weight $-\thetaweight$,
and $\thetaweight<0$, the $\LV_\theta$-semistable points are equal to
$(\CC^{{\bigN}} \times \CC^*) \subset \CC^{{\bigN}+1}$. The second projection
$pr_2: (\CC^{{\bigN}} \times \CC^*) \to \CC^*$ induces the map $[V\!\git{\theta} G]\to
\B\ZZ_b$.

The toric variety $\X_\theta = [V\!\git{\theta} G]$ can be viewed as the
total space of a rank-${\bigN}$ orbifold vector bundle over $\B\ZZ_b$.  This bundle is actually just a $\ZZ_b$
bundle, where  $\ZZ_b$ acts by  
\[
(x_1, \dots, x_{\bigN}) \mapsto (\xi_b^{\Gq_1} x_1, \dots, \xi_b^{\Gq_{\bigN}} x_{\bigN}) \qquad \xi_b = \exp(2\pi i/b).
\]
If $W$ has $\CC^*_R$-weight $b$, then this is exactly the
action of the element $J$ in FJRW-theory.  So the bundle $\X_\theta$ is
isomorphic to $[\CC^{\bigN}/\genj ]$.  This is a special phase which is sort of like a toric variety of a finite group instead of
$\CC^*$. 

We can choose $\CC^*_R$ to have weights
 $\Rq_{x_i}=\Gq_i$ and $\Rq_{p}=0$.
 Now $W$ has $\CC^*_R$-weight $d=\Gq$, and 
 $J = (\xi^{\Gq_1}, \dots, \xi^{\Gq_\bigN},1)$, where $\xi = \exp(2\pi i /d)$. We have $\Gamma = \{((st)^{\Gq_1}, \dots, (st)^{\Gq_\bigN}, t^{-d}) \mid s,t \in \CC^*\} =  \{(\alpha^{\Gq_1}, \dots, \alpha^{\Gq_\bigN}, \beta) \mid \alpha,\beta \in \CC^*\}$, with $\chiR:\Gamma \to \CC^*$ given by $(\alpha^{\Gq_1}, \dots, \alpha^{\Gq_\bigN}, \beta) \mapsto \alpha^d\beta$.  Also the map $\spl:\Gamma \to G/\genj$ is given by $(\alpha^{\Gq_1}, \dots, \alpha^{\Gq_\bigN}, \beta)\mapsto \beta$.

A similar argument to the one we gave above (for the geometric phase) shows that the trivial lift $\lift_0$ is again a good lift of $\theta$.
 
 The critical subset is the single point $\{(0, \dots, 0)\}$ in the quotient $\X_\mmapvalue = [\CC^{\bigN}/\ZZ_d]$.
It is clearly compact, so the the polynomial $W$ is nondegenerate.  We call $\X_\mmapvalue$ a \emph{Landau-Ginzburg phase} or a \emph{pure
  Landau-Ginzburg phase} \cite{Wit:97}. This example underlies
Witten's physical argument of the Landau-Ginzburg/Calabi-Yau
correspondence for Calabi-Yau hypersurfaces of weighted projective
spaces.
\end{enumerate}
\end{exa}

\subsubsection{Choice of $\CC^*_R$}\label{sec:choice}

Our theory does not really depend on $\CC^*_R$, but rather only on the embedding of the groups $G\subseteq \Gamma \subseteq GL(V)$, on the sum $q = \sum_{i=1}^\smalln q_i = \sum_{i=1}^\smalln c_i/d$ of the $\CC^*_R$ weights, and on the choice of lift $\lift$.   Of course the choice of $q$ and the embedding of $\Gamma$ in $\GL(V)$ put many constraints on $\CC^*_R$; but they still allow some flexibility.

For an example of this, consider the case when the gauge group 
$G=(\CC^*)^\littlem$ is an algebraic torus.
Let the action of
the $i$th copy of $\CC^*$ on $V = \CC^{\smalln}$ be given by
 $$\lambda_i(x_1, \dots, x_{\smalln})=(\lambda^{\Gq_{i1}}_i x_1, \dots,
\lambda^{\Gq_{i \smalln}}_i x_{\smalln}).$$  We call the integral matrix $\GQ=(\Gq_{i j})$ the \emph{gauge weight matrix}. If the weight matrix $\GQ=(\Gq_{i j})$ satisfies the \emph{Calabi-Yau condition} $\sum_j \Gq_{ij} = 0$ for each $i$, then we have a lot of flexibility in our choice of $\CC^*_R$, as shown by the following lemma.
\begin{lem}
If the gauge group $G$ is a torus with weight matrix $\GQ=(\Gq_{i j})$, and if we have a compatible $\CC^*_R$ action with weights $(\Rq_1,\dots,\Rq_\smalln)$, such that $W$ has $\CC^*_R$-weight $d$, then for any $\QQ$-linear combination $(\Gq'_1, \dots, \Gq'_\smalln)$ of rows 
of the gauge weight matrix $\GQ$, we define a new choice of R-weights $({\Rq'_1},\dots,{\Rq'_\smalln}) = (\Rq_1+\Gq'_1, \dots, \Rq_\smalln+\Gq'_\smalln)$.  Denote the corresponding $\CC^*$ action by $\CC^*_{R'}$. 

Since the group $\Gamma'$ generated by $G$ and $\CC^*_{R'}$ lies inside the maximal torus of $\GL(\smalln,\CC)$, it is Abelian; and so we automatically have that $G$ and $\CC^*_{R'}$ commute.  We also have the following:
\begin{enumerate}
\item The group $\Gamma'$ generated by $G$ and $\CC^*_R$ is the same as the group $\Gamma$ generated by $G$ and $\CC^*_R$.
\item The $\CC^*_{R'}$-weight of $W$ is equal to $d$.
\item $G \cap \CC^*_{R'} = G\cap \CC^*_R = \genj$, where $J$ is the element defined by Equation~\eqref{eq:J} for the original $\CC^*_R$ action.
\item If $\GQ$ is a Calabi-Yau weight system, then  for both $\CC^*_R$ and $\CC^*_{R'}$ the sum of the weights $q = \sum q_i = \sum c_i/d$ is the same and the central charge $\chat_W$ is the same.
\end{enumerate}
\end{lem}
\begin{proof}
For any element $h'\in \CC^*_{R'}$ we have $h' = (t^{\Rq'_1},\dots,t^{\Rq'_\smalln}) \in \CC^*_{R'}$ for some $t\in \CC^*$.  Letting $h = (t^{\Rq_1},\dots,t^{\Rq_\smalln}) \in \CC^*_R$ and $g = (t^{\Gq'_1},\dots,t^{\Gq'_\smalln}) \in G$, we have $h' = gh$.   
\begin{enumerate}
\item From the equation $h'= gh$, it is now immediate that $G\CC^*_R = G\CC^*_{R'}$. 
\item Since the $G$-weight of $W$ is zero we also have that $\CC^*_{R'}$-weight of $W$ is the same as the $\CC^*_R$-weight of $W$.
\item If $h' \in G\cap\CC^*_{R'}$ then $\gamma = gh$ for some $g\in G$ and $h\in \CC^*_R$, but $\gamma\in G$ implies that $h\in G$, so $G\cap \CC^*_{R'} \subseteq G\cap \CC^*_R$, and a similar argument shows that $G\cap \CC^*_{R} \subseteq G\cap \CC^*_{R'}$.
\item For a Calabi-Yau weight system we have $\sum_j \Gq_{ij} = 0$ for each $i$, hence $\sum_j \Gq'_{j} = 0$, and the invariance of $q$ and $\chat_W$ follows.
\end{enumerate} 
\end{proof}

\begin{rem}
Since $\Gamma$ is preserved in the preceding lemma and lifts depend only on $\Gamma$, any good lift $\lift$ of $\theta\in \Gh$ for the original $\CC^*_R$ action is also a good lift for the new $\CC^*_{R'}$ action.
\end{rem}

\section{Moduli Space and Evaluation Maps}
Throughout this section we assume that we have a reductive $G\subseteq \GL(V)$ and that
$\CC^*_R\subset \GL(V)$ is a diagonal embedding of $\CC^*$ into
$\GL(V)$ such that $G$ and $\CC^*_R$ are compatible. 
 Let $\Gamma\subset \GL(V)$ be the subgroup generated by $G$ and $\CC^*_R$.

We further assume that we have chosen a superpotential $W{\colon}V\to \CC$
which is $G$-invariant and has degree $d$ with respect to the
$\CC^*_R$ action.

Assume that 
$\theta\in\Gh$ defines a polarization $\LV_{\theta}$ such that $V^{ss}_G(\theta)$ is nonempty and is equal to $V^{s}_G(\theta)$.    Denote by $\X_\theta = [V\!\git{\theta} G]$ the corresponding phase of the quotient of $V$ by the action of $G$ and by $\crst_\theta=[\crit(W)\git{\theta} G]$ the phase of the critical locus of $W$. Furthermore assume that $\lift$ is a good lift of $\theta$ if the stability parameter
$\ve$ is not $0+$. 

Finally, assume that  $W$ defines a nondegenerate holomorphic
map $W{\colon}[V\!\git{\theta} G] \to \CC$.

\subsection{State Space}\label{sec:state-space}

    The GLSM has a state space similar to that of FJRW-theory. For complete intersections, it has already been studied by  Chiodo-Nagel \cite{CN}. 
    
 \begin{defn}\label{df:statespace}
 Let 
 \[
 \IX  = \left[\{(v,g)\in V^{ss}_\theta \times G\mid gv = v\}/G \right] 
 \]
denote the the inertia stack of $\X$ (the group $G$ acts on the second factor in the quotient by conjugation).

For each conjugacy class $\Psi\subset G$, let 
\[
I(\Psi) = \{(v,g)\in V^{ss}_\theta \times G | gv=v,\ g \in \Psi\}
\]
and 
\[
\X_{\theta,\Psi} = [I(\Psi)/G].
\]
\end{defn}
We have 
\begin{equation}\label{eq:IX-coprod}
\IX = \coprod_{\Psi} \X_{\theta,\Psi},
\end{equation}
where $\Psi$ runs over all conjugacy classes of $G$.  However, since the action of $G$ on $V^{ss}_{\theta} = V^{s}_{\theta}$ is proper (see 
\cite[\S2.1]{EJK:10} for more on proper group actions), the set $I(\Psi)$ is empty unless all the elements of $\Psi$ are of finite order.  Moreover, by \cite[Lem 2.10]{EJK:10} all but finitely many of the $I(\Psi)$ are empty, so the union in \eqref{eq:IX-coprod} has only a finite number of nonempty terms.

\begin{defn}\label{def:state-space}
  We will abuse notation and denote the map induced by $W$ on $\X_\theta$ as $W{\colon}\X \to
 \CC$. 
  Let ${W}^{\infty}$ be the set $W^{\infty}=(\Re W)^{-1}(M, \infty) \subseteq
 [V\!\git{\theta} G]$ for some large, real $M$.  Similarly, for each conjugacy class $\Psi$ in $G$, denote by  ${W}_{\Psi}^{\infty}=(\Re W)^{-1}(M, \infty) \subseteq \X_{\Psi}$.
 
 We define the \emph{state space} to be the vector space
 $$\Hcal_{W, G}=\bigoplus_{\alpha\in \QQ}\Hcal^{\alpha}_{W, G} = \bigoplus_{\Psi}\Hcal_{\Psi},$$
 where the sum runs over those conjugacy classes $\Psi$ of $G$ for which  $\X_{\theta,\Psi}$ is nonempty, and where
 $$\Hcal^{\alpha}_{W, G}=H^{\alpha+2q}_{CR}(\X_{\theta},W^{\infty},
 \QQ)=\bigoplus_{{\Psi}} H^{\alpha-2\age{(\gamma)}+2q}(\X_{\Psi},
  W^{\infty}_{\Psi},\QQ),$$
  and 
\[
\Hcal_{\Psi}=H^{\bullet+2q}_{CR}(\X_{\theta,\Psi},W^{\infty},
 \QQ)= \bigoplus_{\alpha\in \QQ} H^{\alpha-2\age{(\gamma)}+2q}(\X_{\theta,\Psi},
  W^{\infty}_{\Psi},\QQ),
\] 
That is, the state space is the relative Chen-Ruan cohomology with an
additional shift by $2q$.

For each element $g\in G$ we write $\cjcl{g}\subset G$ for the conjugacy class of $g$ in $G$.
We often call the factor $\Hcal_{\cjcl{g}}$ the \emph{$\cjcl{g}$-sector}, and we call the factor $\Hcal_{\cjcl{1}}$ the \emph{untwisted sector}.
\end{defn}

Recall (see Definition~\ref{def:chat}) that $N$ is the complex dimension of the GIT quotient $\X_\theta=[V\!\git{\theta} G]$
\[
N = \dim([V\!\git{\theta} G]) = n -\dim(G).
\]
And similarly, for each $\cjcl{\gamma}$ we let $N_\gamma$ denote the complex dimension of the sector $\X_{\cjcl{\gamma}}$: 
\[
N_\gamma = \dim(\X_{\theta,\cjcl{\gamma}}) = \dim(\fix(\gamma)) - \dim(Z_G(\gamma)),
\] 
where $Z_G(\gamma)$ is the centralizer of $\gamma$ in $G$.
 
Similar to the classical case, for every $i\in \QQ$, there is a perfect pairing
$$H^i(\X_{\cjcl{\gamma}}, W^{\infty}_{\cjcl{\gamma}})\otimes
H^{2{N}_\gamma-i}(\X_{\cjcl{\gamma}}, W^{\infty}_{\cjcl{\gamma}})\rightarrow
\CC,$$
dual to the intersection pairing of relative homology (see \cite[\S3]{FJR:07a} for more details).  Recall that the age satisfies 
\[
\age(\gamma) + \age(\gamma^{-1}) = \codim(\X_{\theta,\cjcl{\gamma}})  = N-N_\gamma,
\]
so applying the previous pairing to each sector, we obtain a nondegenerate pairing
$$\langle \, , \rangle {\colon} \Hcal_{W,G}^p\otimes \Hcal_{W,
  G}^{2\chat-p} \to \CC,$$ where $\chat = \chat_{W,G}=N-2q$ (see Definition~\ref{def:chat}).

\begin{defn}
 An element $\gamma \in G$ is called \emph{narrow} if the corresponding component $\X_{\cjcl{\gamma}}=[V^{ss,\gamma}/ Z_G(\gamma)]\subset \IX_{\theta}$ is compact (or, equivalently, if its underlying coarse moduli space is compact).  In this case we also say that the corresponding sector $\Hcal_{\cjcl{\gamma}}$ is narrow.
If $\gamma$ is not narrow, we call it (and the corresponding sector) \emph{broad}.
\end{defn}
The theory for narrow sectors is generally much easier to understand than for the broad sectors, but some elements of the broad sectors also behave well, namely those which are supported on a compact substack of  of $\IX_\theta$.
\begin{defn}\label{def:compact-type}
If $W,G$ are nondegenerate for $\X_\theta$ (that is, if $\crst_\theta$ is compact) then we say an element of $\Hcal_{W,G}$ is \emph{of compact type} if its Poincare dual is supported on $[H^{ss}/ Z_G(\gamma)] \subset [V^{ss,\gamma}/ Z_G(\gamma)]$ where $H$ is a $Z_G(\gamma)$-invariant vector subspace of $V$ and $[H^{ss}/Z_G(\gamma)]$ is compact. Any narrow element is of compact type.
  Define $\Hcal_{W,G, \mbox{comp}} \subset \Hcal_{W,G}$ to be the span of all the compact type elements.  \end{defn}
If $G$ is finite and $W$ is nondegenerate, then narrow insertions are the only nonzero elements of compact type.

\subsection{Moduli Space}\label{sec:moduli-space}

Our moduli space will be a sort of unification of the quasimaps of
\cite{CFKi:10, CFKM:11, Kim:11, CCFK:14} with an extension of the Polishchuk-Vaintrob
description of the FJRW moduli space \cite{PoVa:11} to reductive algebraic groups. 

As before, we denote by $\crst_\theta = \X_\theta = [\crit_G^{ss}(\theta)/G] \subset [V\!\git{\theta} G] = [V^{ss}_G(\theta)/G]$ 
the GIT quotient (with polarization $\theta$) of the critical locus of $W$. It will be useful also to consider other affine varieties, so we let  $Z\subseteq V$ be a closed subvariety of $V$ such that $Z^{s}_G(\theta) = Z^{ss}_G(\theta) \neq \emptyset,$ and we denote by $\Zst_\theta$ the quotient $\Zst_\theta = [Z\git{\theta}G] =[Z^{ss}_G({\theta})/G]$.

Our main object of study is the stack of \emph{Landau-Ginzburg quasimaps to $\Zst_\theta$}
\[
\LGQ_{g,k}^{\ve,\lift}(\Zst_{\theta}, \beta),
\]
with a special interest in the case of $\Zst_\theta = \crst_\theta$.
We will embed $\LGQ_{g,k}^{\ve,\lift}(\crst_{\theta},\beta)$ into $\LGQ^{\ve,\lift}_{g,k}(\X_{\theta}, \beta),$ which plays a role analogous to the stack of stable maps with $p$-fields \cite{ChaLi:11, CLL:13}.  

Before we define our moduli problem, we recall the definition of a prestable orbicurve.
\begin{defn}
A \emph{prestable orbicurve} is a balanced twisted curve $\Ccal$ (see \cite[\S4]{AbVi:02}).   
\end{defn}
A prestable orbicurve has a prestable underlying coarse curve (i.e., the only singularities are nodes) and there is a
contraction $\Ccal \rightarrow \Ccal'$ to a stable orbicurve $\Ccal'$ (see \cite[\S9]{AbVi:02}).

\begin{defn}\label{defn:LGquasimap}
Assume that the actions of $G$ and $\CC^*_R$ are compatible and that
$\theta\in\Gh$ defines a  polarization $\LV_{\theta}$  such
that the stable and semistable loci of $Z\subset V$ are nonempty and coincide.  
A \emph{prestable, $k$-pointed, genus-$g$,
 LG-quasimap to $\Zst_\theta$} is a tuple $(\Ccal,\mrkp_1,\dots,\mrkp_k, \Pcal, u, \spn)$ consisting of
\begin{enumerate}[     A.) ]
   \item A prestable, $k$-pointed orbicurve $(\Ccal, \mrkp_1,\dots,\mrkp_k)$
     of genus $g$.
   \item A principal (orbifold) $\Gamma$-bundle $\Pcal{\colon} \Ccal \to
     \B\Gamma$ over $\Ccal$.
   \item A global section $\u{\colon} \Ccal \to \Ecal = \Pcal\times_{\Gamma} V$.
   \item An isomorphism $\spn \colon \chiR_*\Pcal \to \pklogc$ of principal $\CC^*$-bundles ($\pklogc$ indicates the principle $\CC^*$-bundle associated to the line bundle $\klogc$).
   
\end{enumerate}
 such that
\begin{enumerate}
    \item The morphism of stacks $\Pcal{\colon}\Ccal \to \B\Gamma$ is
      representable (i.e., for each point $\mrkp$ of $\Ccal$, the induced
      map from the local group $G_\mrkp$ to $\Gamma$ is injective).
    \item \label{item:basepoints}
      The set of points $b \in \Ccal$ such that any point $p$ of the fiber $\Pcal_b$ over $b$ is mapped by $\u$
      into an $\LV_\theta$-unstable $G$-orbit of $V$ is finite, and this set is 
      disjoint from the nodes and marked points of $\Ccal$.
    \item\label{item:no-fields} The image  of the induced map $[\u]{\colon}\Pcal \to V$ lies 
    in $Z$.   
\end{enumerate}
\end{defn}

\begin{defn}
The points $b$ occurring in condition (\ref{item:basepoints}) above are called 
\emph{base points} of the quasimap.  That is, 
$b\in\Ccal$ is a base point if there is at least one point of the fiber $\Pcal_b$ over $b$ that is mapped by $\u$
      into an $\LV_\theta$-unstable $G$-orbit of $V$.
\end{defn}

\begin{defn}
Any $G$-character $\chi$ defines a $G$-linearized line bundle $\LV_\chi$ on $V$, and hence  a line bundle  on $[V\!\git{\theta} G]$.    We denote this line bundle by $\XLB_\chi$.  

Alternatively, we may construct $\XLB_\chi$ as follows. 
Note that the stable locus $V^{ss}_G(\theta)$ is a principal
$G$-bundle over $[V\!\git{\theta} G]$ and thus defines a morphism $[V\!\git{\theta} G] \to
\B G$ to the classifying stack of $G$.  The character $\chi$ induces a
map of classifying stacks ${\Bchi}{\colon}\B G \to \B \CC^*$.  Composing
these maps gives a morphism $[V\!\git{\theta} G] \to \B\CC^*$ and hence a line
bundle on $[V\!\git{\theta} G]$. This is $\XLB_\chi$.
\end{defn}

\begin{defn}\label{def:CLBchi}
For any prestable LG-quasimap $\qmp = (\Ccal, \mrkp_1,\dots,\mrkp_k, \Pcal, \u, \spn)$, a $\Gamma$-equivariant line bundle $\LV \in \Pic^{\Gamma}(V)$ determines a line bundle $\Lcal = \Pcal\times_\Gamma \LV$ over $\Ecal = \Pcal\times_\Gamma V$, and pulling back along $\u$ gives a line bundle $\u^*(\Lcal)$ on $\Ccal$.  

In particular, any character $\alpha\in\Gammah$ determines a $\Gamma$-equivariant line bundle $\LV_\alpha$ on $V$ and hence a line bundle $\u^*(\Lcal_\alpha)$ on $\Ccal$.  Alternatively, we may construct $\u^*(\Lcal_\alpha)$ by composing the map $\Pcal:\Ccal \to \B\Gamma$ with the map $\B\alpha:\B\Gamma \to \B\CC^*$ to get 
$\u^*{\Lcal_\alpha}:\Ccal \rTo^{\B\alpha \circ \Pcal} \B\CC^*$.
\end{defn}

\begin{defn}\label{def:LGQ-deg}
For any $\alpha \in \Gammah$ we define the \emph{degree} of $\alpha$ on $\qmp$ to be  
\[
\Gdeg_\qmp(\alpha) = {\deg_{\Ccal}(\u^*(\Lcal_{\alpha}))} \in \QQ.
\]
This defines a homomorphism $\Gdeg_\qmp:\Gammah \to \QQ$.

For any $\beta\in \Hom(\Gammah,\QQ)$ we say that an 
LG-quasimap  $\qmp = (\Ccal,x_1,\dots,x_k,\Pcal,\u,\spn)$ has \emph{degree  $\beta$} if $\Gdeg_\qmp = \beta$.
\end{defn}

\begin{rem}\label{rem:relation-among-lift-bundles}  
If $\lift\in\Gammah_\QQ$ is any character of $\Gamma$, then Geometric Invariant Theory guarantees the existence of a line bundle $M$ on $Z\git{\lift}\Gamma$ such that $M$ is relatively ample over $Z\aff\Gamma$ and such that for some $n>0$ we have $\overline{\phi}^*M = \XLB_{\lift}^{\otimes n}$ on $[Z\git{\lift}\Gamma]$, or equivalently, 
\begin{equation}\label{eq:good-pullback1}
p^*\pi^*\overline{\phi}^*M = \LV_{\lift}^{\otimes n}
\end{equation}
 as a $\Gamma$-equivariant bundle on $Z^{ss}$ (see, for example, \cite[Thm 11.5]{Alp:13}). 

If $\lift$ is also a good lift of $\theta\in \Gh_\QQ$ and $Z \subseteq V$ is a closed subvariety of $V$, we have the following diagram of quotients
\begin{equation}\label{diagram:quotients}
\begin{diagram}
Z^{ss}& 
\rTo^{p} & 
[Z\git{\theta}G] & \rTo^{\phi}                & Z\git{\theta}G\\
&&\dTo^{\pi}       &                            & \dTo^{\pi'}\\
&&[Z\git{\lift}\Gamma] & \rTo^{\overline{\phi}} &Z\git{\lift}\Gamma.
\end{diagram} 
\end{equation}
The bundle $\pi'^* (M)$ is ample over $Z\aff G$, and we have 
\begin{equation}\label{eq:good-pullback2}
\pi^*\overline{\phi}^*M = \phi^*\pi'^*(M) = \XLB_\theta^{\otimes n}.
\end{equation}
\end{rem}

\begin{defn}\label{def:prestable-family}
A \emph{family of prestable, $k$-pointed, genus-$g$,
 LG-quasimaps to $\Zst_\theta$ over a scheme $T$} is a tuple $(\varpi\colon \Ccal\to T,\mrkp_1,\dots,\mrkp_k, \Pcal, \u, \spn)$ consisting of
\begin{enumerate}[     A.) ]
   \item  A flat family of prestable, genus-$g$, $k$-pointed orbicurves $(\varpi:\Ccal\to T, \mrkp_1,\dots,\mrkp_k)$ over $T$ with (gerbe) markings $\Scal_i \subset \Ccal$, and sections $\mrkp_i: T \rTo \Scal_i$ which induce isomorphisms between
$T$ and the coarse moduli of $\Scal_i$ for each $i\in\{1,\dots,k\}$ 
   \item A principal $\Gamma$-bundle $\Pcal{\colon} \Ccal \to  \B\Gamma$ over $\Ccal$
   \item A section $\u{\colon} \Ccal \to \Ecal = \Pcal\times_{\Gamma} V$
   \item An isomorphism $\spn \colon \chiR_*\Pcal \to \pklogc$ of principal $\CC^*$-bundles 
   \end{enumerate}
such that the restriction to every geometric fiber of $\varpi\colon \Ccal \to T$ induces a prestable, $k$-pointed, genus-$g$,
LG-quasimap to $\Zst_\theta$.
\end{defn}

\begin{defn}
A morphism between LG-quasimaps $(\varpi\colon \Ccal\to T,\mrkp_1,\dots,\mrkp_k, \Pcal, \u, \spn)$ and $(\varpi'\colon \Ccal'\to T',\mrkp'_1,\dots,\mrkp'_k, \Pcal', \u', \spn')$, is a tuple of morphisms $(\tau, \xi, \rho)$, where $(\tau,\xi)$ form a morphism of prestable orbicurves
\begin{diagram}
\Ccal 			& \rTo^{\xi} 	& \Ccal'\\
\dTo^{\varpi} 	& 				& \dTo^{\varpi'}\\
T				& \rTo^{\tau}	& T'
\end{diagram}
and $\rho\colon \Pcal \to \xi^*(\Pcal')$ is a morphism of principal $\Gamma$-bundles such that the obvious diagrams commute:
\begin{diagram}
\chiR_*(\Pcal)			&	\rTo^\spn			& \pklogc\\
\dTo^{\chiR_*(\rho)}	&						& \dTo\\
\chiR_*(\xi^*(\Pcal'))	&	\rTo^{\xi^*(\spn')}	& \xi^*(\pklogcp)
\end{diagram}
and 
\begin{diagram}
\Ccal			& \rTo^\u			& \Pcal\times_\Gamma Z\\
				& \rdTo_{\xi^*(\u)}	& \dTo_{\rho \times \mathbbm{1}}\\
				&					& \xi^*(\Pcal') \times_\Gamma Z
\end{diagram}
\end{defn}

We now wish to define a stability condition for LG-quasimaps.  To do this we must first define the \emph{length} of an LG-quasimap at a point.
\begin{defn} Choose a polarization $\theta\in \Gh$ and a lift $\lift$ of $\theta$. 
Given a prestable LG-quasimap $\qmp = (\Ccal,\mrkp_1,\dots,\mrkp_k,\Pcal,\u,\spn)$ to $[Z\git{\theta} G]$, and any point $\mrkp \in \Ccal$ such that the generic point of the component of $\Ccal$ containing $\mrkp$ maps to a $\lift$-semistable point, we define the \emph{length} of $\mrkp$ with respect to $\qmp$ and $\lift$ to be
\[
\ell(\mrkp) = \min \left\{\frac{(\u^*(s))_\mrkp}{m} \, \middle|\, s\in H^0(Z,\LV_\lift^m)^\Gamma ,\, m>0\right\},
\]
where $(\u^*(s))_\mrkp$ is the order of vanishing of the section $\u^*(s)\in H^0(\Ccal, \u^*\Lcal_\lift^{\otimes m})$ at $\mrkp$.
\end{defn}
This definition differs from that in \cite[Def 7.1.1]{CFKM:11}, in that it depends on the choice of the lift $\lift\in \Gammah$ rather than on the polarization $\theta\in \Gh$, but the following properties listed in \cite[\S7.1]{CFKM:11} still hold.

\begin{enumerate}
\item For every $\mrkp \in \Ccal$, if the generic point of the component of $\Ccal$ containing $\mrkp$ maps to a $\lift$-semistable point, then we have
\[ \deg_\Ccal(\u^*(\Lcal_\lift)) \ge \ell(\mrkp) \ge 0
\]
with $\ell(\mrkp) >0$ if and only if $\mrkp$ is a $\lift$-basepoint of $\qmp$.

\item If $\lift$ is a good lift, and if $B$ is the set of basepoints of $\qmp$, then the map $\u$, when restricted to $\Ccal\setminus B$ defines a map 
\[
\u:\Ccal\setminus B  \to [Z\git{\lift}\Gamma] \rTo^{\overline{\phi}}  Z\git{\lift}\Gamma.
\]
Since $B$ is disjoint from nodes and marks, and since $Z\git{\lift}\Gamma$ is projective over $Z\aff\Gamma$, this extends to a morphism $\u_{\reg}:\Ccal \to Z\git{\lift}\Gamma$.  Choose $M \in \Pic(Z\git{\lift}\Gamma)$, as in Remark~\ref{rem:relation-among-lift-bundles}, with $p^*\pi^*\overline{\phi}^* M = \LV^{\otimes n}$ for some $n>0$.  We have 
\[
\deg_{\Ccal}(\u^*(\Lcal_{\lift})) - \frac{1}{n}\deg_{\Ccal}(\u_{\reg}^*(M)) = \sum_{\mrkp \in \Ccal} \ell(\mrkp).
\]

\item For any family of prestable LG-quasimaps $(\Ccal/T, \mrkp_1,\dots, \mrkp_k,\Pcal,\spn)$ over $T$,  the function $\ell:\Ccal\to \QQ$ is upper semicontinuous.
\end{enumerate}

\begin{defn} Choose a polarization $\theta\in \Gh$ and a good lift $\lift$ of $\theta$ (See Definition~\ref{def:good-lift}).

Given a prestable LG-quasimap $\qmp = (\Ccal,x_1,\dots,x_k,\Pcal,\u,\spn)$, and given any positive rational $\ve$ we say that $\qmp$ is \emph{$\ve$-stable} (for the lift $\lift$) if 
\begin{enumerate}
\item $\klogc \otimes \u^*(\Lcal_\lift)^{\ve}$ is ample, and 
\item $\ve \ell(\mrkp) \le 1$ for every $\mrkp \in \Ccal$.
\end{enumerate}

We say that $\qmp$ is $\infty$-stable if there exists an $n>0$ such that $\qmp$ is $\ve$-stable for all $\ve >n$.
\end{defn}
\begin{rem}
The $\infty$-stability condition is equivalent to saying that there are no basepoints (by condition (2) when $\ve$ is large) and that on each component of $\Ccal$ the line bundle $\u^*(\Lcal_{\lift})$ has nonnegative degree (by condition (1) when $\ve$ is large), with the degree only being able to vanish on components where $\klog$ is ample.
\end{rem}

We also wish to define another stability condition we call $0+$ stability.  This is the limiting stability condition as $\ve\downarrow 0$; but where $\ve$-stability requires a good lift, $0+$ stability does not.

\begin{defn}
Given a polarization $\theta\in \Gh$ and a  prestable LG-quasimap $\qmp = (\Ccal,x_1,\dots,x_k,\Pcal,\u,\spn)$, we say that $\qmp$ is \emph{$0+$-stable} if there exists a lift $\lift$ (not necessarily a good lift), such that  
\begin{enumerate}
\item Every rational component has at least two special points (a mark $\mrkp_i$ or a node), and 
\item\label{it:zero-plus2} On every component $\Ccal'$ with trivial $\omega_{log,\Ccal'}$, the line bundle $\u^*(\Lcal_{\lift})$ has positive degree.
\end{enumerate}
\end{defn}
It turns out that condition~(\ref{it:zero-plus2}) holds for some lift if and only if it holds for all lifts.  This follows from the next proposition and its corollary.
\begin{pro}
For any two lifts $\lift$ and $\lift'$ of $\theta$, the bundles $\u^*(\Lcal_\lift)$ and $\u^*(\Lcal_{\lift'})
$ differ by a power of $\klogc$:
\[
\u^*(\Lcal_{\lift})^{-1} \otimes \u^*(\Lcal_{\lift'}) = \klogc^a
\] for some $a\in \QQ$.
\end{pro}
\begin{proof}
We have (after clearing denominators, if necessary) that $(\lift^{-1}\lift')(g) = \theta^{-1}(g)\theta(g) = 1$ for any $g\in G$.   Hence $\lift^{-1}\lift'$ factors through $\chiR$, and in fact, we have $\lift^{-1}\lift' = \chiR^\ell$ for some $\ell$.  This gives $\u^*(\Lcal_{\lift^{-1}\lift'}) = \Lcal_{\chiR}^\ell = \klogc^{\ell/d}$.
\end{proof}
\begin{cor}
A prestable LG-quasimap $\qmp = (\Ccal,\mrkp_1,\dots, \mrkp_k,\Pcal)$ satisfies condition~(\ref{it:zero-plus2}) for $0+$-stability for one lift of $\theta$ if and only if it satisfies that condition for every lift of $\theta$.
\end{cor}
\begin{proof}
By the previous proposition the difference between the various lifts is a power of $\klogc$ and hence is trivial on these components.
\end{proof}
\begin{defn}
A  \emph{family of $\ve$-stable, $k$-pointed, genus-$g$,
 LG-quasimaps to $\Zst_\theta$ over a scheme $T$} is  
is a family of prestable $k$-pointed, genus-$g$,
 LG-quasimaps to $\Zst_\theta$ over $T$ (see Definition~\ref{def:prestable-family}) such that the induced LG-quasimap on each geometric fiber is $\ve$-stable. 
\end{defn}

\begin{pro} 
The automorphism group of any $\ve$-stable LG-quasimap $\qmp = (\Ccal,x_1,\dots,x_k,\Pcal,\u,\spn)$ is finite and reduced.  
\end{pro}
\begin{proof} 
Observe that we have an exact sequence
\[
1\to \Aut_{\Ccal} (\qmp) \to \Aut (\qmp) \to \Aut_{\Ccal},
\]
where $\Aut_{\Ccal}(\qmp)$ is the group of automorphisms of $\qmp$ fixing $\Ccal$.  Thus we may break the proof into two parts. First, the same argument as given in \cite[Prop 7.1.5]{CFKM:11} shows that if $\Ccal$ is irreducible but unstable (i.e., $\Aut(\Ccal)$ is infinite), then $\Aut(\qmp)$ is finite.  Second, we prove that $\Aut_{\Ccal}(\qmp)$ is finite.

The quasimap $\qmp$ induces a morphism $\bar{\u}:\pklogc\smallsetminus F \to[V\!\git{\theta}G]$, where $F$ is the fiber in $\pklogc$ over the set of basepoints $B$ of $\u$.  
Any element of $\Aut_{\Ccal}(\qmp)$  must fix $\pklogc$ and the morphism 
$\bar{\u}:\pklogc\smallsetminus F \to[V\!\git{\theta}G]$.  Since $[V\!\git{\theta}G]$ is a DM stack, the set of automorphisms of $\bar{\u}$ restricted to the generic point must be finite.  But any automorphism of the $\Gamma$-bundle $\Pcal$ over a curve is completely determined by its value on the generic point.  
Hence $\Aut_{\Ccal} \qmp$ is finite.

Finally, the automorphism group is reduced because we have restricted ourselves to characteristic $0$. 
\end{proof}

\begin{defn}
For a given choice of compatible $G$- and $\CC^*_R$-actions on a closed affine variety $Z\subseteq V$, a strongly regular character $\theta\in \Gh$, a good lift $\lift$ of $\theta$, and a nondegenerate $W$,  we denote the corresponding stack of $k$-pointed, genus-$g$,
$\ve$-stable (for $\lift$) LG-quasimaps into $\Zst_\theta$ of degree $\beta$
by 
\[
\LGQ_{g,k}^{\ve,\lift}(\Zst_{\theta}, \beta).
\]
If $\ve=0+$ we can dispense with the good lift and instead define
\[
\LGQ_{g,k}^{0+}(\Zst_{\theta}, \beta)
\]
to be the stack of $k$-pointed, genus-$g$, LG-quasimaps into $\Zst_\theta$ of degree $\beta$ that are $0+$-stable for any (and hence every) lift of $\theta$.
\end{defn}

\subsection{Example: Hypersurfaces}\label{sec:hypersurf-stack}

We illustrate these ideas with the example of the hypersurface described in Example~\ref{ex:hypersurf}. The reader who wishes to move directly to the main results of this paper may skip this example on first reading. For many more examples see Section~\ref{sec:examples}.

\begin{exa}\underline{Hypersurfaces (geometric phase):}\label{ex:hypersurf-geom}
Consider again the situation of a hypersurface in weighted projective space, as in Example~\ref{ex:hypersurf}, 
where 
$G = \CC^*$ and $V = \CC^{\bigN}\times\CC$ with coordinates $(x_1,\dots,x_\bigN,p)$.  Let $W = Fp {\colon} \CC^{\bigN+1}\rightarrow \CC$ have $G$-weights $(\Gq_1, \dots, \Gq_{\bigN},-\Gq)$. 

In the geometric phase we have semistable locus $(\cb_1, \cdots, \cb_\bigN)\neq (0,\cdots, 0)$, and critical locus $\{p=0, F(\cb_1, \dots, \cb_{\bigN})=0$. So the quotient $\crst_\theta = \{p=0 = F(\cb_1, \dots, \cb_{\bigN})\}$ of the critical locus is a degree-$\Gq$ hypersurface in $\WP(\Gq_1, \dots, \Gq_{\bigN})$.

Choosing the $\CC^*_R$-weights $(0,\dots,0,1)$ gives a hybrid model in which $W$ has $\CC^*_R$-weight $d=1$ and $\Gamma$ is a direct product $\Gamma \cong G \times \CC^*_R$, with $\spl$ and $\chiR$  just the first and second projections, respectively.  We use the trivial lift $\lift_0$ as our good lift.

A principal $\Gamma$-bundle $\Pcal$ on $\Ccal$ with $\chiR_*(\Pcal) \cong \pklogc$ is equivalent to a line bundle $\Lcal$ on $\Ccal$ with $\Pcal = \mathring{\Lcal} \times \pklogc$, where $\mathring{\Lcal}$ is the principal $\CC^*$-bundle associated to the line bundle $\Lcal$.

The vector bundle $\Pcal\times_\Gamma V$ is $\Lcal^{\oplus \bigN} \oplus (\Lcal^{\otimes(-\Gq)} \otimes \klogc)$, so the stack is 
      \[
      \{(\Ccal, \Lcal, s_1, \cdots, s_\bigN, p)| s_i\in H^0(\Ccal,\Lcal),\  p\in H^0(\Lcal^{-\Gq}\otimes \klogc\}
      \]
      satisfying the stability conditions. Here $\Ccal$ is a marked orbicurve and $\Lcal$ is a line bundle over $\Ccal$. 
      
 A particularly simple case is the $\infty$-stable LG-quasimaps to the critical locus $\crst_\theta$.   Since there are no base points in this case, $(s_1, \cdots, s_{\bigN})\neq 0$. The critical locus requires  $p=0, F=0$, the quasimap $\u=(s_1, \cdots, s_\bigN, p)$ corresponds to a map $\Ccal \to \WP(\Gq_1,\dots,\Gq_\bigN)$.  Moreover, the image of the map must lie in $X_F = \{F=0\} \subset \WP(\Gq_1,\dots,\Gq_\bigN)$ and we have $\Lcal=\u^*\Ocal(1) = \u^* \Lcal_{\lift_0}$.  So the $\infty$-stability condition for the trivial lift exactly corresponds to this map's being a stable map to $X_F$.  Therefore, $\LGQ_{g,k}^{\infty,\lift_0}(\crst_{\theta}, \beta)$ is the stack of stable maps to the critical locus $\crst_\theta = \{F=0\} \subseteq \WP(\Gq_1, \cdots, \Gq_\bigN)$ of degree $\beta$.  Moreover, the stack $\LGQ_{g,k}^{\infty,\lift_0}([V\!\git{\theta}G], \beta)$ is the space of stable maps with $p$-fields, studied in \cite{ChaLi:11, CLL:13}.

There is a parallel theory of quasimaps into $X_F$ that has the same moduli space as our construction in this example (the geometric phase of the hypersurface), but the virtual cycle constructions are different. For $\ve=\infty$, Chang-Li-Li-Liu \cite{CLL:15} 
      proved equivalence of the two theories using a sophisticated degeneration argument. A similar argument probably works for other choices of $\ve$.
  \end{exa}

  \begin{exa}\underline{Hypersurfaces (LG phase):}\label{ex:hypersurf-LG}

  Let's now consider the LG-phase of the hypersurface in weighted projective space.  The unstable locus is $\{p=0\}$. We first consider the same R-charge as before, i.e., $c_{x_i}=0, c_p=1$.
  We have a similar moduli space 
   \[
      \{(\Ccal, \Lcal, s_1, \cdots, s_\bigN, p)\mid s_i\in H^0(\Ccal,\Lcal),\  p\in H^0(\Lcal^{-\Gq}\otimes \klogc)\},
      \]
satisfying the stability condition that $p\neq0$. For the LG-quasimaps to lie in the critical locus requires $s_i=0$. The base points are precisely the zeros of $p$, and the base locus forms an effective divisor $D$ with  $\Lcal^{-\Gq}\otimes \klogc\cong  O(D)$. So we can reformulate the moduli problem as 
   \[
      \{(\Ccal, \Lcal, s_1, \cdots, s_\bigN) \mid s_i\in H^0(\Ccal,\Lcal),\  \Lcal^\Gq \cong \klogc(-D)\}.
      \]
and can be viewed as a \emph{weighted $\Gq$-spin} condition (see \cite{RR}). When $\ve =\infty$, there is no base point, i.e., $D=0$, and we obtain the usual $\Gq$-spin moduli space corresponding to 
$$\Lcal^d\cong \klogc.$$

  There are other choices of R-charge. For example,  we can choose the $\CC^*_R$-action to have weights $\Rq_{x_i}=\Gq_i$ and $\Rq_{p}=0$.  
 We have $\Gamma = \{(\alpha^{\Gq_1}, \dots, \alpha^{\Gq_\bigN}, \beta) \mid \alpha,\beta \in \CC^*\}$, with $\chiR:\Gamma \to \CC^*$ given by $(\alpha^{\Gq_1}, \dots, \alpha^{\Gq_\bigN}, \beta) \mapsto \alpha^d\beta$.  Also the map $\spl:\Gamma \to G/\genj$ is given by $(\alpha^{\Gq_1}, \dots, \alpha^{\Gq_\bigN}, \beta)\mapsto \beta$.
 
The stack $\LGQ_{g,k}^{\ve,\lift_0}(\X_{\theta}, \beta)$  consists of pointed orbicurves $\Ccal$ with line bundles $\Lcal$ and $\Bcal$ such that $\Bcal \cong \klogc\otimes\Lcal^{-d}$ and sections $s_1,\dots,s_N$ of $\Lcal$ and $p$ of $\Bcal$ satisfying the stability conditions. 
Again let's consider $\LGQ_{g,k}^{\infty,\lift_0}(\crst_{\theta}, \beta)$.  In this case, since the semistable locus $\crit^{ss}_{G}(\theta)$ consists of points of the form $(0,\dots,0,p)$ with $p\neq 0$, the sections $s_1,\dots, s_N$ must all vanish.
Again, we can identify $\Bcal=O(D)$ for an effective divisor. This implies
$$\Lcal^d\cong \klogc(-D).$$

Moreover, since $\theta$ has weight $-e$, for some $e<0$, the trivial lift $\lift_0$ corresponds to the map $\Gamma \to \CC^*$ given by $(\alpha^{\Gq_1}, \dots, \alpha^{\Gq_\bigN}, \beta)\mapsto \beta^{-e/d}$, and the pullback line bundle $\u^*(\Lcal_{\lift_0})$ is precisely $\Bcal^{-e/d}$, which is a $d$th root of $\Ocal_{\Ccal}$.  So the stability condition just reduces to the requirement that $\klogc$ be ample---that is, that the orbicurve $\Ccal$ be stable.
\end{exa}

  \subsection{Evaluation maps}

LG-quasimaps to $\Zst_\theta = [Z\git{\theta} G]$ are not quasimaps into $\Zst_\theta$. Their target is the Artin stack $\left[Z^{ss}_G(\theta)/\Gamma\right]$, so one might expect that evaluation maps would only land in the inertia stack $\I\left[Z^{ss}_G(\theta)/\Gamma\right]$ of the stack $\left[Z^{ss}_G(\theta)/\Gamma\right]$.  But we can define evaluation maps 
\[
\LGQ_{g,k}^{\ve,\lift}(\Zst_{\theta}, \beta) \rTo^{ev_i} 	\I\Zst_{\theta} = \coprod_{\Psi} \Zst_{\theta,\Psi}
\]
to the inertia stack of the GIT quotient stack $\Zst_\theta$, as follows.

Observe first that the log-canonical bundle $\klogc$ and its
corresponding principal $\CC^*$-bundle $\pklogc$ have a canonical section at
each marked point $\mrkp_i$ (call this section $dz/z$).  Since $G$ is the kernel of
$\chiR$, the preimage $\chiR_*^{-1}\spn^{-1}(dz/z) \subset \Pcal|_{\mrkp_i}$ is a
principal $G$-orbit in $\Pcal$, and hence defines a principal
$G$-bundle $\Qcal$ over the (orbifold) marked point $\mrkp_i$.  The
section $\u{\colon}\Ccal \to \Pcal\times_\Gamma V$ induces a section $\Ccal
\to \Qcal\times_G V$, which gives a map $\{\mrkp_i\} \to [Z/G]$.  Since
the section $\u$ is never $G$-unstable at nodes and marked points, this
actually gives a map to $\Zst_{\theta}$ and not just to $[Z/G]$.
Moreover, since $\mrkp_i$ is an orbifold point of the form $\mrkp_i =
[\tmrkp_i/G_{\mrkp_i}] \cong \B G_{\mrkp_i}$, the generator of the local group $G_{\mrkp_i}$ must map
to an element of the stabilizer of the image of $\tmrkp_i$.  That is, the
evaluation map takes values in the inertia stack $\I\Zst_\theta$.

Applying this construction to all LG-quasimaps gives the desired evaluation morphisms
\[
ev_i
{\colon}\LGQ_{g, k}^{\ve,\lift}(\Zst_{\theta}, \beta) \to \I\Zst_\theta.
\]

The existence of the evaluation maps shows that we can decompose
\begin{equation}\label{eq:stack-components}
\LGQ_{g, k}^{\ve,\lift}(\Zst_{\theta}, \beta)=\coprod_{\Psi_1,\dots,\Psi_k}\LGQ_{g,
   k}^{\ve,\lift}(\Zst_{\theta}, \beta)(\Psi_1, \cdots, \Psi_k),
\end{equation}
where
 $\Psi_i$ are conjugacy classes in $G$ indexing the twisted sectors of $\I\Zst_\theta$, and the factors  $\LGQ^{\ve,\lift}_{g,
   k}(\Zst_{\theta}, \beta)(\Psi_1, \cdots, \Psi_k)$ are the open and closed substacks where the $i$th evaluation morphism maps to the component (sector) $\Zst_{\theta,\Psi_i}$ of $\I\Zst_\theta$.

\begin{pro}\label{prop:uniform-bound}
There is an integer $\e$ depending only on $W$, $G$, and the action of $\Gamma$ on $V$ such that for any prestable LG-quasimap $\qmp = (\Ccal, \mrkp_1,\dots, \mrkp_k, \Pcal, \u,\spn)$
the degree of every line bundle on $\Ccal$ lies in $\frac{1}{\e}\ZZ$, and for any marked point or node $\mrkp$ of $\qmp$, the order of the local group $G_\mrkp$ at $\mrkp$ is bounded by $\e$.
\end{pro}
\begin{proof}
Recall that $\IX$ is indexed by a finite number of conjugacy classes, each of finite order (see the discussion after Definition~\ref{df:statespace}).  Let $\e$  be the least common multiple of these orders. 

Let $\mrkp$ be a marked point or node of $\Ccal$ and let $G_\mrkp$ be the local group of the orbifold $\Ccal$ at $\mrkp$.  Since $\Pcal:\Ccal\to B\Gamma$ is representable, the corresponding homomorphism  $G_\mrkp \to G \subset \Gamma$ must be injective, and hence $G_\mrkp \cong  \langle \gamma \rangle$ for some $\gamma \in G$ fixing $\u(\mrkp)\in V$. Therefore $\gamma$ must lie in one of the finite number of conjugacy classes corresponding to nonempty components of $\IX$, and hence the order of $G_\mrkp$ must divide $\e$.

This also shows that for any line bundle $\mathscr{N}$ on $\Ccal$ the tensor power $\mathscr{N}^{\otimes\e}$ is the pullback of a line bundle on the coarse curve underlying $\Ccal$, and hence $\e$ times the degree of $\mathscr{N}$ is an integer.
\end{proof}

\begin{exa}
Consider again the geometric phase of a hypersurface $X_F$ in weighted projective space of Examples~\ref{ex:hypersurf} and \ref{ex:hypersurf-geom}. The untwisted sector $\X_{\theta,\cjcl{1}}$ is broad and is the line bundle $\Ocal(-d)$ over weighted projective space.  Any subvariety of weighted projective space defines an element of the state space of compact type, and $\Hcal_{W,G,\mbox{comp}}$ can be identified with  the ambient classes of $H^*_{CR}(X_F, \QQ)$.

The elements of the state space which are not of compact type correspond to the so-called \emph{primitive cohomology} of $H^*_{CR}(X_F, \QQ)$. These correspond to broad insertions in FJRW-theory.

\end{exa}

\section{Properties of the Moduli Space}

\subsection{Boundedness}

In this section we develop some boundedness results that will be used in the proof of Theorem~\ref{thm:DM-Stack} (specifically, to show that the stack of LG-quasimaps is of finite type). 

\begin{pro}\label{prop:degree-is-positive} Given a lift $\lift$ of $\theta$, and any prestable LG-quasimap $\qmp = (\Ccal, \mrkp_1,\dots,\mrkp_k, \Pcal, \u, \spn)$ such that $\u$ maps the generic point of a component $\Ccal'$ of $\Ccal$ to a $\lift$-semistable point of $V$, then 
the degree of the pullback bundle $\u^*(\Lcal_{\lift})$ on $\Ccal'$ is nonnegative:
\[
\deg_{\Ccal'} \u^*(\Lcal_{\lift}) \ge 0.
\]
Moreover, $\deg_{\Ccal'} \u^*(\Lcal_{\lift}) = 0$ if and only if 
there are no $\lift$-basepoints on $\Ccal'$
and composing $\u$ with the natural map $[V\git{\lift}\Gamma]\to V\git{\lift}\Gamma$ induces a constant map $\Ccal' \to V\git{\lift}\Gamma$. 
\end{pro}
\begin{proof}
We may assume that $\Ccal$ is irreducible.  Since the generic point of $\Ccal$ maps to a $\Gamma$-semistable point of $V$ with respect to $\lift$, we must have some $n>0$ for which there exists a nonzero $f\in H^0(V,\LV_\lift^n)^\Gamma$ such that $f(\u(\mrkp)) \neq 0$ for some $\mrkp\in \Ccal$.  Thus $\u^*(f)$ is a nonzero element of $H^0(\Ccal,\u^*\Lcal_\lift^{\otimes n})$, and hence the degree of $\u^*\Lcal_\lift^{\otimes n}$ must be nonnegative.  

Moreover, if $\u$ has no basepoints, but $\u^*\Lcal_\lift$ has degree $0$, then the only global sections of $\u^*\Lcal_\lift^n$ are constant on $\Ccal$ for every $n>0$, hence the induced map $\Ccal \to V\git{\lift}\Gamma$ is constant.  The converse follows from Remark~\ref{rem:relation-among-lift-bundles}---if there are no basepoints and the induced map $\bar{\u}:\Ccal \to V\git{\lift}\Gamma$ is constant, then there is an ample line bundle $M$ on $V\git{\lift}\Gamma$ such that
$\u^*\Lcal_\lift^n = \bar{\u}^* M = \Ocal_\Ccal$. 

Finally, if $b$ is a $\lift$-basepoint of $\u$, then every section in $H^0(V,\LV_\lift^n)^\Gamma$ must vanish at $\u(b)$ and hence $\u^*(f)$ is a nonzero section of $\u^*\Lcal_\lift^{\otimes n}$ on $\Ccal$ that has at least one zero, and hence $\u^*\Lcal_\lift^{\otimes n}$ must have positive degree.
\end{proof}

\begin{cor}\label{bounded-over-M}(Compare to \cite[Cor 3.1.5]{CFKi:10})
The number of irreducible components of the underlying curve of a $k$-pointed, genus-$g$, $\ve$-stable LG-quasimap $\qmp = (\Ccal, \mrkp_1,\dots,\mrkp_k, \Pcal, \u, \spn)$  of degree $\beta$ is bounded in terms of $g$, $k$, and $\beta(\lift)$.
\end{cor}
\begin{proof}

Because the genus is bounded, the number of irreducible components of genus greater than zero is bounded.  Because the number of marked points is bounded, the number of genus-zero components with at least three points is also bounded.  It remains only to consider the components of genus zero with two or fewer marked points or those of genus one with no marked points.  

The existence of any unstable component (genus zero and two or fewer marked points, or genus 1 and no marks) for which $\deg_\Ccal \u^*\Lcal_{\lift}$ vanishes would contradict the conditions of stability.  This implies $\deg_\Ccal \u^*\Lcal_{\lift} >0$ on each such component.  By Proposition~\ref{prop:uniform-bound}, there is a uniform bound $\e$ such that $\deg_\Ccal \u^*\Lcal_{\lift} \ge \frac{1}{\e}$ on each such component, and hence the number of such components is bounded.
\end{proof}
\begin{rem}
The previous corollary also holds for $0+$ stable curves.  The only adjustment that must be made to the proof is that one may use any lift---not just a good lift---in the argument that $\deg_\Ccal \u^*\Lcal_{\lift}\ge \frac{1}{\e}$.
\end{rem}

\begin{thm}\label{bounded-over-A}
Fixing a prestable orbicurve  $\Ccal$, a polarization $\theta \in \Gh$, any character $\xi \in \Gammah_\QQ$ and a rational number $b$,  the family of prestable LG-quasimaps $\qmp$ from $\Ccal$ to $\Zst_\theta$ such that $\deg_{\Ccal} \sigma^*\Lcal_{\xi} = b$ is bounded.  
\end{thm}

The proof is similar to that of \cite[Thm 3.2.4]{CFKM:11}, with additional complications arising from the difference between $\Gamma$ and $G$ and from the fact that for any lift $\lift$, the set $V^{s}_\Gamma(\lift )$ may be empty, even if $\lift$ is a good lift of $\theta$.

It suffices to prove boundedness of the set $S$ of principal $\Gamma$ bundles $\Pcal$ over a fixed, irreducible, orbicurve $\Ccal$ with an isomorphism $\spn: \chiR(\Pcal) \to \pklogc$ which \emph{admit} an $\ve$-stable LG-quasimap $\u:\Pcal \to V$ of class $\beta$ to $[V\!\git{\theta} G]$ (but the particular choice of quasimap $\u$ is not fixed). We can also reduce to the case where $\Ccal$ is nonsingular because a principal $\Gamma$-bundle on a nodal orbicurve $\Ccal$ is given by a principal bundle $\widetilde{\Pcal}$ on the normalization $\widetilde{C}$ and a choice of an identification of the fibers $\widetilde{\Pcal}_{p_+} \cong \widetilde{\Pcal}_{p_-}$ over each node $p$, and for each
node, these identifications are parametrized by the group $\Gamma$. 

We first consider the case that $G$ is connected.  In this case $\Gamma$ is also connected, because there is a surjective map from $G\times\CC^*_R$ to $\Gamma$.

\begin{lem}
Let $G\le GL(V)$ be a connected reductive algebraic group.  Let $T'\le  G$ be a maximal torus containing $\genj$ and let $B'\le  G$ be a  Borel subgroup containing $T'$.   Let $B\le  \Gamma = G\CC^*_R$ be the subgroup of $\Gamma$ generated by $B'$ and $\CC^*_R$, and let $T\le  \Gamma$ be the subgroup generated by $T'$ and $\CC^*_R$.  

The group $\Gamma$ is reductive, and the subgroup $B$ is a Borel subgroup of $\Gamma$ containing $T$, which is a maximal torus of $\Gamma$.  Moreover, we have $B\cap G = B'$, and the unipotent radical $B_u'$ of $B'$ is the same as the unipotent radical $B_u$ of $B$.    
\end{lem}
\begin{proof}
First, $\Gamma$ is reductive because it is the quotient of the reductive group $G\times \CC^*_R$ by the finite subgroup $\genj$.

To see that $B$ is Borel in $\Gamma$, observe first that 
since $B'\triangleleft B$ is normal in $B$, and the quotient $B/B'$ is Abelian, then $B$ is solvable in $\Gamma$.  
If $C\subset \Gamma$ is any solvable subgroup in $\Gamma$ such that $B'\le  C\cap G$, we claim that $B' = C\cap G$.  To see this, note that given any subnormal series $\{1\} = C_0 \triangleleft C_1 \triangleleft \cdots \triangleleft C_n=C$ whose quotients $C_k/C_{k-1}$ are all Abelian, the corresponding series $C_0\cap G \triangleleft \cdots C_n\cap G$ shows that $C\cap G$ is solvable in $G$.  But $B'$ is Borel, hence is a maximal solvable subgroup in $G$ (since $G$ is connected, Borel subgroups are maximal among all solvable subgroups---not just among those that are Zariski-closed and connected---see \cite[11.17]{Bor:91}).  Thus since $B' \le C\cap G$ we must have $B' = C\cap G$.

To see that $B$ is Borel in $\Gamma$, it remains to show that $B$ is maximal among the solvable subgroups of $\Gamma$.  Assume that  $S \le \Gamma$ is a solvable subgroup of $\Gamma$ with $B\le S$.  Any element $s\in S\le \Gamma$ can be written as $s = gr$, where $g\in G$ and $r\in \CC^*_R$, and $g = sr^{-1} \in SB \le S$, so $g\in S\cap G$.  By the  previous paragraph, we have $S\cap G = B'$, so $g\in B'$ and $gr\in B'\CC^*_R = B$.  Therefore $S = B$, and $B$ is maximal among solvable subgroups of $\Gamma$, hence $B$ is Borel in $\Gamma$.  

The group $T$ is Abelian and contains $T'$, and the quotient $T/T'$ is isomorphic to $H \cong \CC^*$, by the map $\chiR:T \to H$ (See \eqref{eq:defzeta}).
By \cite[Corol.~pg.~149]{Bor:91}  we have that $T$ is also a torus.  Since $\genj \le T'$, we have $T'\cap \CC^*_R = \genj = B'\cap \CC^*_R$. So  the sequence \eqref{eq:chiR-seq} gives us $B/B' = T/T' = H$, and we have the following diagram of short exact sequences:
\[
\begin{diagram}
1 &\rTo &T'      &\rTo &T      &\rTo^{\chiR} &H          &\rTo &1\\
  & & \dTo^{\cong} & & \dTo    &           & \dTo^{\cong} \\
1 &\rTo &B'/B_u' &\rTo &B/B_u' &\rTo^{\chiR} &B/B'       &\rTo &1\\  
\end{diagram}
\]
where the leftmost vertical arrow is an isomorphism because $T'$ is the maximal torus of $B'$.  Thus we have $B/B'_u \cong T$.  The maximal torus $\widetilde{T}$ of $B$ must contain $T$ and is isomorphic to $B/B_u$.  Also $B_u$ contains $B'_u$, so we have 
\[
\begin{diagram}
T  & \rTo^{\cong} & B/B'_u\\
\dInto &          & \dOnto\\
\widetilde{T} & \rTo^{\cong} & B/B_u.
\end{diagram}
\]
Thus $T = \widetilde{T}$ must be the maximal torus and $B_u = B'_u$.
\end{proof}

We can now finish the proof of the theorem.  By \cite[\S2.11]{Ram:96a} (see also \cite[Thm 3.2.4]{CFKM:11}) we may choose a reduction to a principal $B$-bundle  $\Pcal'$ for each principal $\Gamma$-bundle $\Pcal$ in the set $S$.
Let 
\[
R = \{\overline{\Pcal}' = \Pcal'/B_u \mid \Pcal\in S\}.
\]
 For each $\overline{\Pcal}'$, let $d_{\overline{\Pcal}'}: \widehat{T} \to \QQ$ be given by \[
 d_{\overline{\Pcal}'}(\xi) = \deg_\Ccal (\overline{\Pcal}' \times_T \CC_\xi).\]
 By \cite[Lm 3.2.7]{CFKM:11} and \cite[Prop 3.1 and Lem 3.3]{HoNa:01} the set $S$ is bounded if the set $R$ is bounded, and $R$ is bounded if the set \[
D = \{d_{\overline{\Pcal}'}: \widehat{T} \to \QQ \mid \overline{\Pcal}' \in R\}
\]
is bounded.

The argument in the proof of \cite[Lm.~3.2.8]{CFKM:11} shows there is a $\QQ$-basis $\{\theta_1,\dots,\theta_m\}$ of $\widehat{T'} \otimes \QQ$ such that for each $\theta_i$ we have $V^{s}_{T'}({\theta}) \subseteq V^{ss}_{T'}(\theta_i)$, and $\theta = \sum_i a_i \theta_i$, with $a_i>0$ for every $i$.  
 For each choice of $\Ccal,\Pcal,\u$, the generic point of $\Ccal$ maps by $\u$ into 
\[
  V^{s}_{G}({\theta}) \subseteq V^{s}_{T'}(\theta) \subseteq V^{ss}_{T'}(\theta_i) 
\]

By Lemma~\ref{lem:three-lifts-suffice}, for each $i$ there are three standard lifts $\lift_-$, $\lift_0$ and $\lift_+$ of $\theta_i$ such that 
\[V^{ss}_{T'}(\theta_i) = V^{ss}_{T}(\lift_-) \cup V^{ss}_{T}(\lift_0) \cup V^{ss}_{T}(\lift_+),
\]
and such that the $\CC^*_R$-weight of $\lift_-$, $\lift_0$, and $\lift_+$  is $-1$, $0$, and $1$, respectively.
 
Therefore, for at least one of these three lifts (denote it simply by $\lift_i$), the generic point of $\Ccal$ must map to $V^{ss}_{\lift_i}(T)$.  That means  there must exist some section $s \in H^0(V,\LV_{\lift_i}^n)^T$ for some $n>0$. This induces a section of $\overline{\Pcal'} \times_{T} \LV_{\lift_i}^{n}$ that does not vanish on the generic point of $\Ccal$, and thus  $\deg_\Ccal \overline{\Pcal'} \times_{T} \LV_{\lift_i} \ge 0$.  

If $2g-2+k \ge 0$, let $\one$ be the unique lift of the trivial $T'$-character with $\CC^*_R$-weight $1$. The LG-quasimap structure $\Pcal \to \pklogc$ means that $\deg_\Ccal (\overline{\Pcal'} \times_{T} \CC_{\one}) = \deg_\Ccal \klogc = 2g-2+k \ge 0$. Similarly, if $2g-2+k<0$, let $\one$ be the unique lift of the trivial $T'$-character with $\CC^*_R$-weight $-1$.  Again we have $\deg_\Ccal (\overline{\Pcal'} \times_{T} \LV_{\one}) = \deg_\Ccal \klogc^{-1} = 2-2g-k > 0$.

The characters $\{\lift_1,\dots,\lift_n, \one\}$ form a basis for $\widehat{T}\otimes\QQ$.  Since the degree  $\deg_\Ccal \overline{\Pcal'} \times_{T} \LV_{\one} = |2g-2+k|$ is fixed by $g$ and $k$, it suffices to prove there are a finite number of possible values for each $\deg_\Ccal \overline{\Pcal'} \times_{T} \LV_{\lift_i}$. To do this, note that there is a unique $r\in \ZZ$ such that the character $\xi$ can be written as
\[
\xi = \sum_{i=1}^n a_i \lift_i + r\one.
\]  
This gives 
\[
\deg_{\Ccal}\sigma^*(\Lcal_\xi) - r|2g-2+k| = \sum_{i=1}^n a_i \deg_\Ccal \overline{\Pcal'} \times_{T} \LV_{\lift_i} 
\]
All the coefficients $a_i$ and $r$ are independent of $\Ccal$ and $\overline{\Pcal}'$ and depend only  on the action of $T$ on $V$ and on characters $\xi, \one,\lift_1,\dots, \lift_n\in \widehat{T}$. 
Since the $a_i$ are all positive, since every $\deg_\Ccal \overline{\Pcal'} \times_{T} \LV_{\lift_i}$ is nonnegative, and since the left-hand side of this equation is determined by $g,k,b$, the possible values for 
$\deg_\Ccal \overline{\Pcal'} \times_{T} \LV_{\lift_i}$ are all bounded.  Since these degrees must all lie in $\frac{1}{\e}\ZZ$, there can only be a finite number of them.  Hence the set  $D$ is bounded, and we have shown the theorem in the case that $G$ is connected. 

In the general case we assume that $G$ is reductive with identity component $G_0$ such that $G/G_0$ is finite, but $G$ is not necessarily connected.  Let $\Gamma_0$ be the component of $\Gamma$ containing the identity element.  Clearly $G_0\CC^*_R\subseteq \Gamma_0$, so the group $\Gamma/\Gamma_0$ is a quotient of $G/G_0$, and hence it is finite.

Given an LG-quasimap $\qmp = (\Pcal \to \pklogc \to \Ccal, \Pcal \rTo^\u V)$, the quotient $\Pcal/\Gamma_0$ is a principal $\Gamma/\Gamma_0$-bundle over $\Ccal$, hence it is a prestable orbicurve, which we denote by $\widetilde{\Ccal}$.  The morphism $\Pcal\to\pklogc$ induces a morphism $\Pcal \to \pklogc\times_\Ccal \widetilde{\Ccal} \cong \pklogct$, which, when combined with $\Pcal \rTo^\u V$, defines an LG quasimap $\widetilde{\qmp}$ with gauge group $G_0$ over $\widetilde{\Ccal}$ and with polarization $\theta|_{G_0}$ induced from $\theta$ by restriction to $G_0$. 
The degree $\deg^{G_0}_{\widetilde{\qmp}}(\theta|_{G_0})$ is just $\frac{1}{m}\Gdeg_\qmp \theta = b/m$, so by the proof of the theorem in the connected case, the subfamily of these LG quasimaps $\widetilde{\qmp}$ with gauge group $G_0$ over a fixed $\widetilde{\Ccal}$ is bounded.  But the number of  \'etale maps $\widetilde{\Ccal} \to \Ccal$ of fixed degree $|\Gamma/\Gamma_0|$ is 
finite, so the family of all such prestable LG-quasimaps over $\Ccal$ is bounded.

\begin{cor}
For any $\beta\in \Hom(\Gammah_\QQ,\QQ)$ the family of prestable LG-quasimaps $\qmp$ from $\Ccal$ to $\Zst_\theta$ of degree $\beta$ is bounded.  
\end{cor}

\subsection{Finite type Deligne-Mumford Stack}

To prove that ${\LGQ}_{g,k}^{\ve,\lift}(\Zst_\theta, \beta)$ is a Deligne-Mumford stack we first define an intermediate stack.

\begin{defn}\label{def:A}
Given $\Gamma \rTo^{\chiR} \CC^*$, let $\Af_{g,k} \to \bun_{\Gamma,g,k}$ denote the stack of tuples $(\Ccal, \mrkp_1,\dots,\mrkp_k,\Pcal, \spn)$ consisting of a $k$-pointed, genus-$g$ prestable orbicurve, a principal $\Gamma$-bundle $\Pcal$ on  $\Ccal$, and an isomorphism $\spn:\chiR_*(\Pcal) \to \pklogc$ with the property that the induced morphism $\Ccal \to \B\Gamma$ is representable.
\end{defn}
\begin{lem}\label{lm:A-is-smooth-Artin}
The stack $\Af_{g,k}$ is a smooth Artin stack, locally of finite type over $\CC$.
\end{lem}
\begin{proof} By \cite[Prop 2.1.1]{CFKM:11} the stack $\bun_\Gamma$ is a smooth Artin stack, locally of finite type over $\CC$.  Let $\mathfrak{C}$ denote the universal curve over $\bun_\Gamma$; let $\mathring{\mathfrak{w}}_{\log,\mathfrak{C}}$ denote the principal $\CC^*$-bundle associated to the log canonical bundle of $\mathfrak{C}$; and let $\mathfrak{P}$ denote the universal $\Gamma$-bundle over $\mathfrak{C}$.  As a stack $\Af_{g,k}$ is isomorphic to $\operatorname{Isom}_{\mathfrak{C}/\bun_{\Gamma}}\left(\chiR_*(\mathfrak{P}), \mathring{\mathfrak{w}}_{\log,\mathfrak{C}}\right)$, hence it is representable  over $\bun_\Gamma$.   This proves that $\Af_{g,k}$ is an Artin stack of finite type over $\bun_\Gamma$, hence locally of finite type over $\CC$.

To see that it is smooth, we use an argument similar to the proof that $\bun_\Gamma$ is smooth (see \cite[Prop 2.1.1]{CFKM:11}).  First note that since $\CC^*_R$ commutes with $G$, it lies in the center of $\Gamma$, therefore the subgroup $G\subset \Gamma$ inherits a $\Gamma$ action from the adjoint action of $\Gamma$ on itself.  Since $\ker(\chiR) = G$, the infinitesimal automorphisms of the data $(\Pcal,\spn)$ are precisely those automorphisms of $\Pcal$ over $\Ccal$ which are also automorphisms of $\Pcal$ as a $G$-bundle over $\pklogc$.  Over $\pklogc$ these are given by $\Hom_{\Gamma}(\Pcal,G) = H^0(\pklogc,\Pcal\times_\Gamma G)$.  

Assume we are given a family $\Ccal_0$ of prestable orbicurves over $\spec(A_0)$, where $A_0$ is a finitely generated $\CC$-algebra $A_0$, and that we are given the data $(\Pcal_0,\spn_0)$ over $\Ccal_0$.  Given a square-zero extension $A$ of $A_0$ with kernel $I$ and an extension $\Ccal$ of $\Ccal_0$ over $\spec(A)$, extensions of $(\Pcal_0,\spn_0)$ to $\Ccal$ are parametrized by $H^1(\pkk_{\log,\Ccal_0},\Pcal\times_\Gamma \mathfrak{g})\otimes_{A_0} I$, where $\mathfrak{g}$ is the Lie algebra of $G$. The obstruction to extending $(\Pcal_0,\spn_0)$ to $\Ccal$ lies in $H^2(\pkk_{\log,\Ccal_0},\Pcal\times_\Gamma \mathfrak{g})\otimes_{A_0} I$.  Since the fibers of the projection $q: \pkk_{\log,\Ccal_0} \to \Ccal$ are affine, the higher derived push forwards $R^i q_*\Pcal\times_\Gamma \mathfrak{g}$ vanish for $i>0$ and the Leray spectral sequence degenerates. So $H^2(\pkk_{\log,\Ccal_0},\Pcal\times_\Gamma \mathfrak{g}) = H^2(\Ccal_0,q_*(\Pcal\times_\Gamma \mathfrak{g}))$. Since $\Ccal_0$ is a family of curves over an affine scheme, $H^2(\Ccal_0,\Pcal\times_\Gamma \mathfrak{g})$ vanishes, and the deformations are unobstructed.  Hence $\Af_{g,k}$ is smooth over the stack $\Mf_{g,k}$ of prestable orbicurves, which is also smooth. 
\end{proof}

\begin{thm}\label{thm:DM-Stack}
Let $\beta\in \hom(\Gh,\QQ)$.  Fix either $\ve=0+$ or $\ve>0$ and a good lift $\lift$. Let $\LGQcal$ be 
\[
\LGQcal =\LGQ_{g,k}^{0+}([V\!\git{\theta}G], \beta) \quad \text{ or } \quad
\LGQcal =\LGQ_{g,k}^{\ve,\lift}([V\!\git{\theta}G], \beta).
\]
Let $\Mf = \Mf_{g,k}$ denote the stack of prestable orbicurves $\Mf_{g,k}$ and let $\Af = \Af_{g,k}$. 
The stack $\LGQcal$ is a Deligne-Mumford stack of finite type over $\Mf$.
And if  $Z\subset V$ is a closed subvariety with GIT quotient $\Zst_\theta = [Z\git{\theta}G]$, then  $\LGQ_{g,k}^{\ve,\lift}(\Zst_\theta, \beta)$ (or $\LGQ_{g,k}^{0+}(\Zst_\theta, \beta)$) is a closed substack of $\LGQcal$.

If $\Pcal \to \Ccal$ denotes the universal principal $\Gamma$-bundle $\Pcal$ on the universal curve $\pi:\Ccal\to \LGQcal$, and $\Ecal = \Pcal\times_\Gamma V$, then $\LGQcal\to \Af$ is representable and has a relative perfect obstruction theory
\begin{equation}\label{eq:rel-perf-obs}
\phi_{\LGQcal/\Af} : \mathbb{T}_{\LGQcal/\Af} \to \mathbb{E}_{\LGQcal/\Af} = R^{\bullet}\pi_* \Ecal, 
\end{equation}
where $\mathbb{T}_{\LGQcal/\Af}$ is the relative tangent complex (dual to the relative cotangent complex $\LL_{\LGQcal/\Af}$). 
\end{thm}

\begin{proof}
Let $\pi: \mathfrak{C}\to \Af$  be the universal curve, and let 
$\mathfrak{P}$ be the universal $\Gamma$-bundle on $\mathfrak{C}$.  Let $\mathfrak{E} = \mathfrak{P} \times_\Gamma V$ and let $q:\mathfrak{E} \to \mathfrak{C}$ be the projection. Chang-Li in \cite[\S2.1]{ChaLi:11} show that the direct image cone $\mathfrak{Q} = C(q_*\mathfrak{E})$, consisting of sections $\u$ of $\mathfrak{E}$ over $\mathfrak{C}$ is an Artin stack, and the projection $\mu: \mathfrak{Q} \to \Af$ is representable and quasiprojective with relative perfect obstruction theory 
\begin{equation}\label{eq:Artin-RelPOT}
\phi_{\mathfrak{Q}/\Af} : \mathbb{T}_{\mathfrak{Q}/\Af} \to \mathbb{E}_{\mathfrak{Q}/\Af} = R^{\bullet}\pi_* \sigma^* \Omega_{\mathfrak{E}/\mathfrak{C}}. 
\end{equation}

We can realize $\LGQcal$ as the open substack of $\mathfrak{Q}$ where the following conditions hold:
\begin{enumerate}
	\item The degree of $\u$ is $\beta$.
	\item The section $\u$ maps the generic points of components of $\Ccal$ to $V^{ss}_{\Gamma}(\lift)$.   
	\item The section $\u$ maps the nodes and the marked points to $V^{ss}_{\Gamma}(\lift)$.
	\item $\klogc \otimes \u^*\Lcal_{\lift}^\ve$ is ample.
	\item $\ve\ell(\mrkp) \le 1$ for all $\mrkp\in \Ccal$.
\end{enumerate}
Therefore $\LGQcal$ 
is an Artin stack with relative perfect obstruction theory \eqref{eq:Artin-RelPOT} over $\Af$.  Since $\mathfrak{E}$ is a vector bundle and $\sigma$ is a section, we have $\sigma^*\Omega^\vee_{\mathfrak{E}/\mathfrak{C}} = \mathfrak{E}$, giving the desired relative perfect obstruction theory \eqref{eq:rel-perf-obs} for $\LGQcal$.  
Note that the obstruction theory for a more general target $[Z\git{\theta}G]$ is not necessarily perfect.

The fact that $\LGQcal$ is of finite type over $\CC$ follows from the boundedness results of the previous sections, as follows.
Consider the obvious projection morphisms $\nu:\LGQcal \rTo \Mf$   and $\mu:\LGQcal \rTo \Af$.  By Corollary~\ref{bounded-over-M} the image of $\nu$ is contained in an open and closed substack $\Sf \subset \Mf$ of finite type.  By Theorem~\ref{bounded-over-A} $\mu$ factors thorough an open substack of finite type $\Af_\beta \subset \Af$ lying over $\Sf$. Since $\mu$ is quasiprojective, this implies that $\LGQcal$ is of finite type.  

The fact that $\LGQcal$ is Deligne-Mumford follows from the finiteness of the automorphism group. 
Finally,  the condition that the image of $\u$ lie in $Z$ is a closed condition, so  $\LGQ_{g,k}^{\ve,\lift}([Z\git{\theta}G],\beta)$ 
is a closed substack of $\LGQcal$.
\end{proof}

\subsection{Separatedness}

\begin{thm}
For any $\ve$ and for any closed subvariety $Z\subset V$ with GIT quotient $\Zst_\theta = [Z\git{\theta}G]$, the Deligne-Mumford stack $\LGQ_{g,k}^{\ve,\lift}(\Zst_{\theta}, \beta)$ (or $\LGQ_{g,k}^{0+}(\Zst_\theta, \beta)$) is a separated stack. 
\end{thm}

The proof of the theorem follows, by the valuative criterion, from the following lemma.
\begin{lem}\label{lem:separated}
Let $R$ be a discrete valuation ring over $\CC$.  Let $\eta$ be the generic point of $\spec(R)$, and let $0$ be the closed point.  Consider two prestable LG-quasimaps 
\[\qmp_1 = (\Ccal_1, \mrkp_{1,1},\dots,\mrkp_{1,k}, \Pcal_1, \u_1, \spn_1) \dsand \qmp_2 = (\Ccal_2, \mrkp_{2,1},\dots,\mrkp_{2,k}, \Pcal_2, \u_2, \spn_2)\]
 over $\spec(R)$ that are isomorphic over $\eta$.  Given a lift $\lift$ (not necessarily good), if for each $i\in\{1,2\}$ the quasimap $\qmp_i$ satisfies the stability condition that $\kk_{\log,\Ccal_i} \otimes\u_i^*(\Lcal_\lift)^\ve$ is ample on every fiber of $\Ccal_i$, then after possibly replacing $R$ with a cover ramified at $0$, the isomorphism of $\qmp_1$ with $\qmp_2$ over $\eta$ extends to an isomorphism over all of $R$.
\end{lem}
\begin{proof}
The proof is similar to that in \cite{MOP:11} and \cite{CFKM:11}, but with additional complications arising from the difference between $\Gamma$ and $G$.

If $C_1$ and $C_2$ are the coarse underlying curves of $\Ccal_1$ and $\Ccal_2$, respectively, then semistable reduction (see \cite[Prop 3.48]{HaMo:98}) guarantees that, after possibly replacing $R$ with a cover ramified over $0$, there is a prestable, $k$-pointed curve $C,\mrkp_1,\dots,\mrkp_k$ over $\Delta = \spec(R)$ and dominant morphisms $\pi_1:C\to C_1$ and $\pi_2:C\to C_2$ compatible with the sections and such that each $\pi_i$ is an isomorphism away from the nodes of the central fibers $(C_{i})_0$.  

The description of the universal deformation of twisted nodal curves in  \cite[Rem 1.11]{Ols:07} shows that one can define an orbicurve $\Ccal$ with coarse underlying space $C$, and $\Ccal$ is compatible with the maps $\pi_i$; that is, we have dominant maps $\widetilde{\pi}_i:\Ccal \to \Ccal_i$ such that the diagrams 
\begin{diagram}
\Ccal & \rTo^{\widetilde{\pi}_i} &  \Ccal_i\\
\dTo &                     & \dTo\\
C    &  \rTo^{\pi_i}       & C_i
\end{diagram}
commute for $i\in \{1,2\}$, and the maps $\widetilde{\pi}_i$ are  isomorphisms except possibly at the nodes of the central fibers $(\Ccal_{i})_0$.

Pulling back $\qmp_1$ and $\qmp_2$ to $\Ccal$ gives prestable LG-quasimaps on $\Ccal$ that are isomorphic over the generic fiber.  For each $i\in \{1,2\}$ let  $B_i$ denote the base locus in $\Ccal$ of $\qmp_i$.  Let $U = \Ccal\setminus B_1\cup B_2$.   The maps $\u_i$ induce maps $\bar{\u}_i: \pklogc|_U \to [V\!\git{\theta} G]$.  These maps agree on the generic fiber, and the target $[V\!\git{\theta} G]$ is separated, so $\bar{\u}_1 = \bar{\u}_2$ on $U$.  The isomorphism $f_\eta:\Pcal_1 \to \Pcal_2$ (which is $\Gamma$-equivariant over $\Ccal_\eta$) must, therefore extend to an isomorphism $f$ over $\Ccal_\eta \cup U$ in such a way that it is $G$-equivariant over $\pklogc$.

The question now is whether the $G$-equivariant morphism $f:\Pcal_1\to \Pcal_2$ over $\pklogc$ defines a $\Gamma$-equivariant morphism over $\Ccal_\eta \cup U$.  For any $p\in \Pcal_1$ over the special fiber and for any $r\in \CC^*_R$, there is a $\Delta$-valued point $\tilde{p}$ of $\Pcal_1$ which specializes to $p$.
Since $f$ is $\Gamma$-equivariant over the generic fibers, we have
\[
f(r\tilde{p}_\eta) = r f(\tilde{p}_\eta). 
\]
And since the space $\Pcal_2$ is separated, we must have
\[
f(r\tilde{p}) = r f(\tilde{p}), 
\]
and hence $f(rp) = rf(p)$.  Therefore, $f$ is $\Gamma$-equivariant and defines an isomorphism $\Pcal_1 \to \Pcal_2$ of principal $\Gamma$-bundles over $\Ccal_\eta \cup U$.

Since the base loci $B_i$ are disjoint from all nodes, the isomorphism is defined everywhere but a finite collection of (unorbifolded) points $(B_1\cup B_2)\cap \Ccal_0$ in the central fiber $\Ccal_0$, and by Hartog's Theorem it must extend to all of $\Ccal$.  Therefore, we may assume that on $\Ccal$ the bundles $\Pcal_i$ are isomorphic and the maps $\u_i$ are identified by that isomorphism.

The morphisms $\widetilde{\pi}_i$ must contract precisely those components of the special fiber $\Ccal_0$ for which  $\klogc\otimes \u^*(\Lcal_{\lift})^{\ve}$ is not ample. But this condition depends only on $\Pcal$ and $\u$, so the same components are contracted for each $i$, and the isomorphisms $(\Ccal_1)_\eta \to (\Ccal_2)\eta$ and $(\Pcal_1)_\eta \to (\Pcal_2)_\eta$ extend to isomorphisms $\Ccal_1 \to \Ccal_2$ and $\Pcal_1 \to \Pcal_2$, which gives the isomorphism $\qmp_1 \cong \qmp_2$. 
\end{proof}

\subsection{Properness}

\begin{thm}
If $Z\git{\theta}G \subset V\!\git{\theta} G$ is projective, then for every good lift $\lift$ of $\theta$, for every pair $g,k$, 
for every $\beta\in \hom(\Gammah,\QQ)$, and for every $\ve$, 
the stack $\LGQ_{g, k}^{\ve,\lift}(\Zst_{\theta}, \beta)$ (or $\LGQ_{g,k}^{0+}(\Zst_\theta, \beta)$, with any lift) is proper (over $\spec(\CC)$).  
\end{thm}
\begin{proof} 
To begin, note that $Z\git{\theta}G$ is always projective over $Z\aff{G}$ and it is projective (over $\spec(\CC)$)  if and only if $Z\aff{G} = \spec(\CC)$. And thus $Z\git{\theta}G$ is projective implies that $Z\git{\lift}\Gamma$ is projective as well (but the Artin stack $[Z\git{\lift}\Gamma]$ need not be separated). 

To prove properness of the stack $\LGQ_{g, k}^{\ve,\lift}(\Zst_{\theta}, \beta)$, we use the valuative criterion.   
 If $\Delta$ is the spectrum of a complete  discrete valuation ring with generic point $\eta$ and special point $0$, assume we have a $k$-pointed, genus-$g$, $\ve$-stable LG-quasimap $\qmp_\eta$:
 \begin{equation}\label{eq:genericLGQM}
 \Ccal_\eta \lTo \pklogceta \lTo \Pcal_\eta \rTo^{\u_\eta} Z
 \end{equation}
 over the generic fiber.   After possibly shrinking $\Delta$ and making a base change ramified only over $0$, we may assume that the base points of $\qmp_\eta$  are sections $b_i:\eta \to \Ccal_\eta$, with $i\in \{1,\dots,m\}$, and that the lengths $\ell(b_i)$ are constant.  We may also assume that the generic fiber $\Ccal_\eta$ is smooth and irreducible.
Let $C_\eta\to \eta$ be the coarse underlying curve of $\Ccal_\eta$ with corresponding sections $\bar{\mrkp}_1,\dots, \bar{\mrkp}_k$, and $\bar{b}_1,\dots,\bar{b}_m$. 

The first step of the proof is to choose a suitable open set $U'\subset \Ccal_\eta$ where the LG-quasimap is sufficiently well behaved that we can extend it to most of the central fiber.  Gluing this to the original LG-quasimap and using Hartog's theorem will allow us to extend this to an LG-quasimap on the entire curve.  For any semistable curve extending $\Ccal_\eta$ to all of $\Delta$, we also choose a section $\tau$ of the log-canonical bundle on a coarse, stable model (or on a special semistable model if $2g-2+k\le 0$).

We choose 
\[
U' = \Ccal_\eta \setminus \{c_1,\dots, c_n, b_1,\dots,b_m\},
\]
where the $c_1,\dots c_n$ are sections of $\Ccal_\eta$ chosen as follows.

If $\Ccal_\eta$ with sections $\mrkp_1,\dots,\mrkp_k,b_1,\dots,b_m$ is not stable as a pointed curve, then the genus $g$ of $\Ccal_\eta$ satisfies $2g-2+k+m \le 0$.  If $2g-2 + k + m =0$, then either $k=2$ and $\kk_{\log,\Ccal_\eta}$ is trivial, or $m>0$, and $\kk_{\log,\Ccal_\eta}$ is trivial on $\Ccal_\eta\setminus\{b_1,\dots, b_m\}$. Letting $\widetilde{\Ccal}\to \Delta$ be any semistable curve over $\Delta$ whose coarse generic fiber agrees with $\Ccal_\eta$, then repeatedly contracting all the $-1$ curves in the special fiber will give a new (unorbifolded) curve $C$ with no $-1$ curves in the special fiber.  Every component of this new curve $C$ will either have genus $1$ with no marked points, or have genus $0$ with two marked points or nodes, in either case, every component has a trivial log-canonical bundle. In this case we take no additional sections (that is, $n=0$), and we fix a section $\tau:C \to \pkk_{\log,C}$. 

Similarly, if $2g-2+k+m < 0$ and $m=1$, then $g=k=0$. Taking any semistable curve over $\Delta$ whose coarse generic fiber agrees with $\Ccal_\eta$ and repeatedly contracting $-1$-curves not containing $b_1$ gives a curve $C$ with only one component in the special fiber, and it must contain the point $b_1$.  The log-canonical bundle is trivial over $C \setminus \{b_1\}$.  Again, take no additional sections (that is, $n=0$), and fix a section $\tau:C\setminus\{b_1\}  \to \pkk_{\log,C}$. 

If $2g-2+k+m < 0$ and $m=0$, then $g=0$, and $0\le k\le 1$. Choose $c_1$ to be any section that is disjoint from the section $\mrkp_1$ (or let $c_1$ be any section if $k=0$). If $\widetilde{\Ccal}\to \Delta$ is any semistable curve (pointed with $c_1$ and with $\mrkp_1$ if $k=1$) over $\Delta$ whose coarse generic fiber agrees with $\Ccal_\eta$, then repeatedly contracting all the $-1$ curves in the special fiber (relative to both the $\mrkp_1$ and $c_1$) will give a new (unorbifolded) curve $C$ with no $-1$ curves in the special fiber.  This new curve $C$ will have trivial log-canonical bundle.  Fix a section $\tau:C \to \pkk_{\log,C}$. 
  
Finally, consider the case where $\Ccal_\eta$ with sections $\mrkp_1,\dots,\mrkp_k,b_1,\dots,b_m$ is a stable curve. Since the stack of stable curves is proper, there is a unique family of genus-$g$, $k+m$-pointed stable curves $\bar{C}\to \Delta$  extending $C_\eta \to \eta$.  We also denote by $\bar{\mrkp}_1,\dots, \bar{\mrkp}_k$ and $\bar{b}_1,\dots,\bar{b}_m$ the extensions to $C$ of the corresponding (coarse) sections of $C_\eta$.    

If the log-canonical bundle $\kk_{\log,\bar{C}}$ is trivial over $\bar{C}\setminus\{\bar{b}_1,\dots,\bar{b}_m\}$, then we need no additional sections, so $n=0$.  

If $\kk_{\log,\bar{C}}$ is not trivial over $\bar{C}\setminus\{\bar{b}_1,\dots,\bar{b}_m\}$, we may choose a finite set of sections $\bar{c}_1,\dots \bar{c}_n$  of $\bar{C}\to \Delta$, disjoint from the marks $\bar{\mrkp}_i$ and the basepoints $\bar{b}_i$  such that the corresponding $\CC^*$-bundle $\pkk_{\log,\bar{C}}$ is trivial on the complement 
\[
U = \bar{C} \setminus \{\bar{c}_1,\dots, \bar{c}_n, \bar{b}_1,\dots,\bar{b}_m\}.
\]
Let $c_1,\dots, c_n$ be the corresponding sections of $\Ccal_\eta$, and 
let $\tau:U \to \pkk_{\log,{\bar{C}}}$ be a section of the log-canonical bundle.

Now that we have chosen the $c_1,\dots, c_n$ in every case, we set
\[
U' = \Ccal_\eta \setminus \{c_1,\dots, c_n, b_1,\dots,b_m\}.
\]
Over $U'$ the morphism $\u_\eta:\Pcal_\eta \to Z$ has no base points, so its image lies entirely in $Z^{ss}_G(\theta)$, and it corresponds to a morphism $\bar{\u}': \pklogc|_{U'}  \to \Zst_\theta$.  Composing with $\tau:{U'} \to \pklogc|_{U'}$ gives a morphism $\alpha' = \bar{\u}' \circ \tau: {U'} \to \Zst_\theta$.  

The quotient $\Zst_\theta$ is Deligne-Mumford with a projective coarse moduli space, so by \cite[Lem 2.5]{CCFK:14} there is a unique orbicurve $\widetilde{\Ccal}_\eta$ constructed from $\Ccal_\eta$ by possibly adding additional orbifold structure at the points of 
$\Ccal_\eta \setminus U'_\eta = \{{b}_1,\dots,{b}_m,{c}_1,\dots {c}_n\}$, 
and a unique representable morphism $\alpha_\eta: \widetilde{\Ccal}_\eta \to \Zst_\theta$ such that $\alpha_\eta|_{U'_\eta} = \alpha'$.  The stability conditions on the generic fiber imply that the morphism $\alpha_\eta$ is a balanced twisted stable $k+m+n$-pointed map to $\Zst_\theta$.  By \cite[Thm 1.4.1]{AbVi:02} or \cite[Thm A]{ChRu:02} the stack $\Kcal^{bal}_{g,k+m+n}(\Zst_\theta)$ of such maps is proper, so $\alpha_\eta$ extends uniquely to a balanced twisted $k+m+n$-pointed stable map $\alpha:\Ccalt \to \Zst_\theta$.  

If $\ve \neq 0+$, we have a good lift $\lift$, and there is an obvious morphism $p:\Zst_\theta \to [Z\git{\lift}\Gamma]$, given by sending any $T \lTo Q \rTo^f Z$ to $T \lTo Q\times_G \Gamma \rTo^{\tilde{f}} Z$, where ${\tilde{f}}(q,\gamma) = \gamma f(q)$.  Composing with $\alpha$ we have $p\circ \alpha: \Ccalt \to [Z\git{\lift}\Gamma]$.  
It is straightforward to see that 
\begin{equation}\label{eq:l-beta}
\deg_{\Ccal_\eta}  \u^* \Lcal_\lift - \deg_{\Ccal_\eta}  (p\circ\alpha)^* \Lcal_\lift = \sum_{b\in\Ccal} \ell(b). 
\end{equation}

As in \cite[7.1.6]{CFKM:11}, for each subcurve $D$ of the special fiber of $\Ccalt$ we define 
\[
\deg(D,\Lcal_\lift) = \deg_D((p\circ\alpha)^* \Lcal_\lift) + \sum_{b_i \cap D \neq \emptyset} \ell(b_i).
\]

For each $-1$ curve $D$ of the special fiber $\Ccalt_0$ (i.e., an irreducible, rational component that does not contain any of the marked points $\mrkp_i$ and only intersects the rest of the special fiber in one point $z$), we contract this $-1$ curve if and only if 
\begin{equation}\label{eq:contraction}
\deg(D,\Lcal_\lift) \le \frac{1}{\ve}.
\end{equation}
Repeat this process until there are no $-1$ curves satisfying \eqref{eq:contraction}.  If $\ve=0+$ then we just contract all $-1$ curves.  We call the contracted curve $\Ccalo$. Denote the set of all the resulting points on $\Ccalo$  by  $\{z_1,\dots, z_s\}$.  For each $z_i$ denote by $\Psi_i$ the tree of rational curves in $\Ccalt$ that was contracted to $z_i$.
Let 
\[
\Uo = \Ccalo\smallsetminus\{b_1,\dots,b_m,c_1,\dots,c_n, z_1,\dots, z_s\},
\]

Note that the generic fiber of $\Ccalo$ is not the same as the original generic fiber $\Ccal_\eta$ because $\Ccalo$  may have additional orbifold structure at the points $\{{b}_1,\dots,{b}_m,{c}_1,\dots {c}_n\}$.  Nevertheless, these are equal on the open set $\Uo\cap \Ccal_\eta$.  

Let $\varrho:\Ccalo \to {C}$ be the natural map to the coarse underlying curve ${C}$ of $\Ccalo$. Forgetting the orbifold structure of $\Ccalo$ at all the sections $b_1,\dots,b_m,c_1,\dots,c_n$ gives a unique (balanced, prestable) orbicurve $\Ccal$ over $\Delta$ with coarse underlying curve ${C}$, with orbifold structure matching $\Ccal_\eta$ on the generic fiber, and with orbifold structure at the nodes of the central fiber matching $\Ccalo$.  From now on we will think of $\Uo$ as an open subset of $\Ccal$ rather than of $\Ccalo$.

If the generic fiber $\Ccal_\eta$ with its sections $\mrkp_1,\dots, \mrkp_k,b_1,\dots, b_m$ is stable, then let $f:{C} \to \bar{C}$ be the obvious contraction to the (coarse) stable model $\bar{C}$ of ${C}$.
Otherwise, let $f:{C} \to \bar{C}$ be the curve obtained by repeatedly contracting all rational curves that do not contain any marked points $\mrkp_i$, basepoints $b_i$, or additional sections $c_i$.

We now use the pullback of the trivialization $\tau$ along $f\circ \varrho$ to construct a trivialization of $\pkk_{\log,\Ccal}$ over $\Uo$.  

First note that the log-canonical bundles $\klog$ satisfy the following two properties:
\begin{enumerate}
\item For any prestable curve $f:{C} \to \bar{C}$ with sections ${\mrkp}_1,\dots,{\mrkp}_k$ lying over $\bar{\mrkp}_1,\dots,\bar{\mrkp}_k$, with generic fiber $C_\eta$ and \emph{no rational tails (with respect to the marks ${\mrkp}_1,\dots,{\mrkp}_k$) in the central fiber}, we have 
\[\kk_{\log,{C}} = f^*(\kk_{\log,\bar{C}}).\]  
\item For any pointed prestable orbicurve $\Ccal, {\mrkp}_1,\dots,{\mrkp}_k$ with the map to its coarse underlying curve ${C},{\mrkp}_1,\dots,{\mrkp}_k$ denoted by $\varrho: \Ccal \to {C}$, we have
\[
\varrho^*(\kk_{\log,{C}}) = \kk_{\log,\Ccal}.
\]
\end{enumerate}

Except on uncontracted $-1$ trees of the central fiber 
we have $\varrho^*f^*(\pkk_{\log,\bar{C}}) = \pkk_{\log,\Ccal}$, so in this case pulling back the trivialization $\tau$ along $f\circ \varrho$ immediately gives a trivialization of $\pkk_{\log,\Ccal}$. 

For each $-1$ tree $\Psi$ of the central fiber, there is a neighborhood $N$ of $\Psi$ in $\Ccal$ which is the result of a sequence of successive blowups of points of $\Psi$.  For each blowup, let $x$ be a local coordinate of the curve before blowing up (so the curve is locally of the form $\spec R[[x]]$).  Removing the section $x=0$ we can trivialize the log-canonical bundle (before blowing up) by $dx/x \mapsto 1 \in \Ocal$.  Removing the strict transform of $x=0$ from the blowup we have $\omega_{\log,\Ccal}$ is again trivialized by $dx/x \mapsto 1$.
Repeat this process for each blowup, and for each $-1$ tree, and denote the sections removed (the strict transform of each $x=0$)  by 
$\{d_1,\dots, d_t\}$.  We abuse notation and redefine $\Uo$ to be 
\[
\Uo = \Ccalo\smallsetminus\{b_1,\dots,b_m,c_1,\dots,c_n, z_1,\dots, z_s, d_1,\dots,d_t\}.
\] 
Combining the local trivialization with the pullback of $\tau$, we have now constructed a trivialization of $\pklogc$ on all of $\Uo$.

Let $\Uo \leftarrow \Qcal \to Z^{ss}_G(\theta)$ be the principal $G$-bundle on ${\Uo}$ corresponding to the morphism $\alpha:\Uo\to \Zst_\theta$.   The space $\Pcalt = \Qcal \times_G \Gamma$ is a principal $\Gamma$-bundle over $\Uo$ with a natural morphism $i:\Qcal \to \Pcalt$ (given by sending a point with local coordinate $(u,g) \in \Uo\times G$ to the point with local coordinate $((u,g),1)\in (\Uo\times G) \times_G \Gamma$), and we may construct a corresponding $\Gamma$-equivariant morphism $\ut:\Pcalt \to Z^{ss}_G(\theta)$ (by sending points of the form $i(q)\in \Pcalt$ to $\alpha(q)$ and extending $\Gamma$-equivariantly to the rest of $\Pcalt$).

Denote by $b_{i,0}$, $c_{i,0}$, and $d_{i,0}$ be the intersection of the sections $b_i$, $c_i$, and $d_i$, respectively, with the central fiber $\Ccal_0$ of $\Ccal$.  The bundles and morphisms $\Pcal_\eta,\u_\eta$ and $\Pcalt,\ut$ agree on the
 intersection $\Ccal_\eta\cap\Uo$, so they glue together to give a principal $\Gamma$-bundle $\Pcal$ and a morphism $\u:\Pcal\to Z$ defined on the open set $\Ccal_\eta\cup\widetilde{U} = \Ccal \smallsetminus\{b_{1,0},\dots,b_{m,0},c_{1,0},\dots,c_{n,0},
z_1,\dots, z_s, d_{1,0},\dots,d_{t,0}\}$.

By \cite[Lem~4.3.2]{CFKM:11} the principal $\Gamma$-bundle $\Pcal$ extends from $\Ccal_\eta\cup\Uo$ to all of $\Ccal$.  We also denote this extension by $\Pcal$.  By Hartogs' Theorem, the morphism $\u:\Pcal \to Z$ extends uniquely over all of $\Ccal$.  

Over $\widetilde{U}$ we have $\chiR_*: \Pcal' \to \CC^* \times \widetilde{U}$, and we may combine this with the trivialization $\varrho^*f^*\tau:\CC^* \times \widetilde{U} \to \pkk_{\log,\Ccal}|_{\widetilde{U}}$.  By construction this composition agrees with the LG-quasimap structure $\Pcal_\eta \to \pklogc$ on  $\Ccal_\eta\cap\widetilde{U}$, so these glue together to give a morphism $\Pcal \to \pklogc$ on all of $\Ccal_\eta\cup\Uo$.  Again, by Hartogs' Theorem this morphism extends uniquely to a morphism over all of $\Ccal$.  Thus we have constructed a family $\qmp$ of prestable LG-quasimaps 
 \[
\Delta \lTo \Ccal \lTo \pklogc \lTo \Pcal \rTo^{\u} Z
 \]
whose generic fiber is the stable LG-quasimap \eqref{eq:genericLGQM}.
It remains to show that the central fiber of $\qmp$ is $\ve$-stable. 

The rest of the proof is very similar to the corresponding part of the proof of \cite[Thm 7.1.6]{CFKM:11}, but we include a sketch here for completeness. Let $\Psi_1,\dots, \Psi_s$ be the $-1$ trees in $\Ccalt_0$ that were contracted (and the resulting points in the special fiber of $\Ccal_0$ are $z_1,\dots, z_s$).  The analog of \eqref{eq:l-beta} for the special fiber gives  
\[
\deg_{\Ccal_0} \u^*\Lcal_\lift  = \deg_{\Ccalt_0} (p\circ\alpha)^* \Lcal_\lift + \sum_{i=1}^s \ell(z_i)+
\sum_{b\in\Ccal_0, b\not\in \{z_1,\dots,z_s\}} \ell(b). 
\]
The degree of $\u^*\Lcal_\lift$ is constant in the fibers, and combining this with semicontinuity of $\ell$ and using the previous equation we obtain the following for each basepoint $b_i$
\[
\ell(b_{i,0}) = \ell(b_i)  \le \frac{1}{\ve} 
\]
and similarly, for each contracted $-1$-tree $\Psi_j$ we have
\[
\ell(z_j) =  \deg((p\circ\alpha)^*\Lcal_\lift|_{\Psi_j}) + \sum_{b\in \Psi_j} \ell(b) = \deg(\Psi_j,\Lcal_\lift)  \le \frac{1}{\ve}. 
\]

Finally we verify that the ampleness criterion holds.  First, any uncontracted $-1$-curve $D$ must have $\deg_D(\u^*\Lcal_\lift) > \frac{1}{\ve}$ by construction, hence $\klogc\otimes\u^*\Lcal_\lift^\ve$ is ample on $D$.  

Second, for any component $D$ with $\deg_D \klogc = 0$, if $D$ contains a $\lift$-basepoint, we must have $\deg_D \u^*\Lcal_\lift >0$ by Proposition~\ref{prop:degree-is-positive}. If there are no $\lift$-basepoints of $\u$ on $D$, then since $\lift$ is a lift of $\theta$ there must be no $\theta$ basepoints on $D$ and $D$ lies entirely in $\Uo$ and so on $D$ the bundle $\u^*(\Lcal_{\lift})$ is equal to $\alpha^*(\XLB_{\theta})$ (see Remark~\ref{rem:relation-among-lift-bundles}).  In that case, $\alpha:\Ccalt \to \Zst_\theta$ is a stable map that does not contract the component $D$, so  $\u^*(\Lcal_{\lift}) = \alpha^*(\XLB_\theta)$ is ample on the component $D$.  
\end{proof}

\section{The Virtual Cycle}

In this section, we construct the virtual cycle for the case where all insertions are of {compact type} (See Definition~\ref{def:compact-type}).  To do this, we use the cosection localization techniques of Kiem-Li \cite{KiLi:10} as applied in \cite{ChaLi:11, CLL:13}.
To use the cosection technique 
we need a relative perfect obstruction theory for $\LGQcal = \LGQ_{g,k}^{\ve,\lift}([\X_\theta],\beta)$ over $\Mf$ and a cosection 
\[
\Obs_\LGQcal \to \Ocal
\] 
whose degeneracy locus is $\LGQ_{g,k}^{\ve,\lift}(\crst_\theta,\beta) \subset \LGQcal$.

\subsection{Cosection and Virtual Cycle}

As shown in Theorem~\ref{thm:DM-Stack}, if $\Pcal \to \Ccal$ denotes the universal principal $\Gamma$-bundle on the universal curve $\pi:\Ccal\to \LGQcal$, and $\Ecal = \Pcal\times_\Gamma V$, then the map  $\LGQcal = \LGQ_{g,k}^{\ve,\lift}([V\!\git{\theta}G],\beta)\to \Af$ is representable and has a relative perfect obstruction theory 
\[
\phi_{\LGQcal/\Af} : \LL_{\LGQcal/\Af} \to \mathbb{E}_{\LGQcal/\Af} = R^{\bullet}\pi_* \Ecal,
\]
where $\Af = \Af_{g,k}$ is the smooth Artin stack of principal $\Gamma$-bundles $P$ on twisted, $k$-pointed, genus-$g$ prestable curves $\Ccal$ with an isomorphism $\chiR_*P \to \pklogc$ to the (punctured) log-canonical bundle, such that the corresponding morphism $\Ccal \to \B \Gamma$ is representable.

We wish to define a cosection, that is, a homomorphism $\Obs_{\LGQcal}  \to \Ocal_\LGQcal$ from the obstructions of $\LGQcal$ over the stack of prestable curves. To do this, we will proceed in several  steps.  First we define a relative cosection $\delta:\Obs_{\LGQcal/\Af} = R^1\pi_*\Ecal  \to \Ocal_\LGQcal$ from the relative obstruction space over $\Af$.  We then show this also induces a relative cosection $\Obs_{\LGQcal/\Mf_{g,k}}\to \Ocal_\LGQcal$.  Finally we show that this induces a cosection $\Obs_{\LGQcal}\to \Ocal_\LGQcal$.

To begin, note that the superpotential $W:V \to \CC$ is equivariant with respect to the homomorphism $\chiR:\Gamma \to \CC^*$, so for any LG-quasimap $(\Ccal, \mrkp_1,\dots,\mrkp, \Pcal,\u, \spn)$ the map  $W$  defines a morphism of vector bundles $W: \Ecal = \Pcal\times_\Gamma V \to \pklogc \times_{\CC^*} \CC = \klogc$.  Differentiating along the section $\u$ gives another morphism of vector bundles 
\[
dW_\u : T\Ecal|_\u  \to T\klogc|_\u .
\]
But we have canonical isomorphisms $T\Ecal|_\u  \cong \Ecal$ and $T\klogc|_\u \cong \klogc$, so this gives a map
\begin{equation}\label{eq:dW_u}
dW_\u:\Ecal \to \klogc.
\end{equation}

\begin{lem}
For any LG-quasimap $\qmp = (\Ccal, \mrkp_1,\dots,\mrkp_k, \Pcal, \u , \spn)$ into $[V\!\git{\theta} G]$, if $ev_i(\qmp)$ lies in  a narrow sector $\X_{\theta,\gel}$,  then the map 
$dW_\u: \Ecal \to \klogc$ factors through the obvious inclusion $\klogc(-\mrkp_i) \subset \klogc$.  Thus, if all the marked points are narrow, then 
$dW_\u$ factors through the canonical inclusion $\iota: \kc \to \klogc$.  
\end{lem}
\begin{proof}
To prove that $dW_\u$ factors through $\klogc(-\mrkp_i)\subset \klogc$ is a local problem, so it suffices to show that the map $dW|_{\u(\mrkp_i)}: V \to \CC$ vanishes.

Assume that $ev_i(\qmp)$ lies in a compact stack $\X_{\theta,\gel} = [V^{ss,g}\git{\theta} Z_G(\gel)] $, where $V^{ss,\gel}$ is the fixed point locus of $\gel$ in $V^{ss}$. In particular, $\u(\mrkp_i)\in V^{ss,\gel}$.  

As observed in Section~\ref{sec:state-space}, $\gel$ must have finite order, and since we are in characteristic zero, $\gel$ must be semisimple. Choose coordinates $x_1,\dots,x_N$ on $V$ to diagonalize $\gel$.  We may assume that for some $k$ the coordinates $x_1,\dots,x_k$ are fixed by $\gel$, while $x_{k+1},\dots,x_N$ are not fixed by $\gel$.  Therefore $V^{ss,\gel} = \{x_{k+1}= x_{k+2}= \cdots = x_N=0\}\cap V^{ss}$.  

If $W|_{V^{ss,\gel}}$ is not zero, then it defines a polynomial on ${V^{ss,\gel}}$.  But the $G$-invariance of $W$ implies that $W$ defines a function on $\X_{\theta,\cjcl{\gel}} = [V^{ss,\gel}/Z_G(\gel)]$.  And since $\X_{\theta,\cjcl{\gel}}$ is compact, this function must be constant, hence $W|_{V^{ss,\gel}}$ is constant.
        
Since $W|_{V^{ss,\gel}}$ is constant, every monomial of $W$ is either constant or contains at least one $x_j$ for some $j>k$.  Since $W$ is $G$-invariant, every monomial that contains such an $x_j$ must also contain another $x_\ell$ for $\ell>k$ (otherwise the monomial is not fixed by $\gamma$).  Therefore every monomial in each partial derivative $\frac{\partial W}{\partial x_i}$ must also contain at least one $x_j$ for $j>k$, and hence each $\frac{\partial W}{\partial x_i}$ and also $dW$ must vanish on $V^{ss,\gel}$.
\end{proof}

We can now define the homomorphism $\delta:  R^1\pi_*\Ecal = \Obs_{\LGQcal/\Af}  \to \Ocal_\LGQcal$ in any situation where $dW_\u$ factors through the canonical inclusion $\iota: \kc \to \klogc$.
\begin{defn}\label{def:the-cosection}
If $dW_\u$ factors through the canonical inclusion $\iota: \kc \to \klogc$,
let $\delta:\Ecal \to \kc$ be the homomorphism corresponding to that factorization:
\[
\begin{diagram}
\Ecal 	& \rTo^{dW_\sigma} 	& \klogc\\
		& \rdTo^{\delta}	& \uInto^{\iota}\\
		&					& \kc.
\end{diagram}
\]
By Serre duality, we have $\delta\in \hom(\Ecal,\kc) \cong H^0(\Ccal, \Ecal^{\vee}\otimes \kc) \cong H^1(\Ccal,\Ecal)^{\vee}$, hence $\delta$ defines a homomorphism $H^1(\Ccal,\Ecal) \to \Ocal_{\Ccal}$, and on the stack $\LGQcal$ we have $\delta: R^1\pi_* \Ecal \to \Ocal_\LGQcal$, as desired. 
\end{defn}

Next we discuss the case of a compact-type insertion whose Poincare dual is supported in $[H^{ss}/ Z_G(\gamma)]$ where $H\subset V^g$ is a $\Gamma$-invariant subspace. In this case, we replace the moduli space by the
substack $ev^{-1}([H^{ss}/ Z_G(g)])$. It carries a different obstruction theory. Our motivating example is the geometric phase of the quintic three-fold, where the moduli space consists of tuples $(s_1, \cdots, s_5, p)$, where $s_i\in H^0(\Ccal, \Lcal)$ and  $p\in H^0(\Ccal, \Lcal^{-5}\otimes \omega_{\log,\Ccal})$. The evaluation map is
\[
ev_i: ( (s_1, \cdots, s_5, p)\rightarrow (s_1(z_i), \cdots, s_5(z_i), \Res_{z_i}(p)).
\]
The compact-type sector is given by the locus $\Res_{z_i}(p)=0$, namely, the locus where $p\in H^0(\Ccal, \Lcal^{-5}\otimes \omega_{,\Ccal})$. This allows us to replace $\omega_{\log,\Ccal}$ by $\omega_{\Ccal}$ in the argument above.

In the general case, $\Ecal$ is an orbifold vector bundle and an element of $ev^{-1}_i([H^{ss}/ Z_G(\gamma)])$ consists of a pair $(P, s)$ where $P$ is a principal $\Gamma$-bundle and $s$ is a section of $\Ecal$ with certain vanishing conditions.
To make this precise consider the local chart near $z_i\in U$:
$$\Ecal_{U}\cong (U\times V)/G_{z_i},$$
where $G_{z_i}$ is generated by $g$.  Since $H \subset V^g$, we may choose a complementary subspace $H^{\perp}$ of $H$ inside $V^g$.  Moreover, since $G_{z_i}$ is finite we may choose a $g$-invariant complement $V^{\perp}$ of $V^g$ inside $V$.  Thus we have a $g$-invariant decomposition $V=H\times H^{\perp}\times (V^g)^{\perp}$. Now, glue $\Ecal(U-\{z_i\})$ with $(U\times V)/G_{z_i}$ by the map
\[
(z, v_1, v_2, v_3)\rightarrow (z, v_1, z v_2, v_3)
\]
to obtain a new bundle $\tilde{\Ecal}$. That is, modify the fiber of $\Ecal$ at $z_i$ by introducing a zero on $v_2$. It is clear that $ev^{-1}_i([H^{ss}/ Z_G(\gamma)])$ consists of $(P, \tilde{s})$ where $\tilde{s}\in H^0(\Ccal, \tilde{\Ecal})$
and its obstruction is given by $H^1(\Ccal, \tilde{\Ecal})$. 

The only case that is not covered by the previous argument is when $W$ contains a monomial of the form $p W_0$ where $W_0$ is a monomial of $H$ and $p$ is a coordinate of $H^{\perp}$. We may reduce to the case where all the monomials of $W$ are of this form. But then $dW$ defines a map from the total space of $\tilde{\Ecal}$ to $\omega_{\Ccal}$, as required.

\begin{pro} 
If $ev_i(\qmp)$ lies in either a narrow sector $\X_{\theta,\gel}$ or a compact substack of  $\X_{\theta,\gel}$ in the broad case, then the degeneracy locus of $\delta$ (the locus on $\LGQcal$ where $\delta$ vanishes) is precisely the closed substack $\LGQ_{g,k}^{\ve,\lift}(\crst_\theta,\beta) \subset \LGQcal$.
\end{pro}
\begin{proof}
The hypothesis guarantees that $dW_\u$ factors through the canonical inclusion $\iota: \kc \to \klogc$, and hence that $\delta$ is defined.
The stack $\LGQ_{g,k}^{\ve,\lift}(\crst_\theta,\beta)$ embeds in $\LGQcal$ as the locus where the image of $\u$ lies in $\crit(W)$, and this is, by definition, the locus where $dW_\u$ vanishes.  Since $dW_\u = \iota \circ \delta$, and $\iota$ is injective, this is precisely the locus where $\delta$ vanishes.
\end{proof}

Next we show that  $\delta$ induces a relative cosection  $\Obs_{\LGQcal/\Mf_{g,k}}\to \Ocal_\LGQcal$ by generalizing the arguments of \cite[\S3.3]{ClaDis:14}.  To reduce clutter in our notation, we denote $\Mf = \Mf_{g,k}$ and continue to use $\Af$ to denote $\Af_{g,k}$ and $\LGQcal$ to denote $\LGQ_{g,k}^{\ve,\lift}([V\!\git{\theta}G],\beta)$.  

\begin{lem}
The homomorphism $\delta:\Obs_{\LGQcal/\Af} \to \Ocal_\LGQcal$ induces a homomorphism $\Obs_{\LGQcal/\Mf}\to \Ocal_{\LGQcal}$ (which we also denote by $\delta$).
\end{lem}
\begin{proof}
If $p:\LGQcal \to \Af$ is the obvious forgetful morphism, then we have the 
deformation exact sequence
\[
T_{\Af/\Mf} \rTo^\tau \Obs_{\LGQcal/\Af} \rTo \Obs_{\LGQcal/\Mf} \rTo 0.
\]
So to verify that the cosection $\delta$ induces a cosection $\Obs_{\LGQcal/\Mf} \to \Ocal_\LGQcal$, we must verify that $\delta\circ\tau = 0$.

As we saw in the proof of Lemma~\ref{lm:A-is-smooth-Artin}, at any point $\mathcal{A} = (\Ccal, \mrkp_1,\dots,\mrkp_k,\Pcal \to \pklogc)$ of $\Af$ the deformation space  $T_{\Af/\Mf}$ at $\mathcal{A}$ is $H^1(\Ccal,\Pcal\times_\Gamma \mathfrak{g})$, where $\mathfrak{g}$ is the Lie algebra of $G$ with the adjoint action of $\Gamma$.  Let $e: \Pcal \times_\Gamma \mathfrak{g} \to  \Pcal\times_\Gamma V   = \Ecal \cong T\Ecal |_{\u}$ be given by sending $(z,\alpha)\in \Pcal \times_\Gamma \mathfrak{g}$ to $(z,\alpha\u(z))$.  Since $W$ is $G$ invariant, we have that  $dW|_\u \circ e = 0$, hence $\delta \circ e = 0$.
Fiberwise, over any point $(\Ccal, \mrkp_1,\dots,\mrkp_k,\Pcal \to \pklogc, \u:\Pcal \to V) \in \LGQcal$, the map $\tau$ is just $h^1(e)$, and hence $\delta\circ \tau = 0$.  
Thus  $\delta:\Obs_{\LGQcal/\Af} \to \Ocal_\LGQcal$ induces a homomorphism from the cokernel $\Obs_{\LGQcal/\Mf}$ of $\tau$ to $\Ocal_{\LGQcal}$.
\end{proof}

Finally, to apply the general theory of Kiem-Li \cite{KiLi:10}, we must show that the relative cosection $\delta: \Obs_{\LGQcal/\Af} \to \Ocal_\LGQcal$ induces an absolute cosection
$ \Obs_{\LGQcal} \to \Ocal_\LGQcal$, where $\Obs_\LGQcal$ is the absolute obstruction bundle, defined as the cokernel of a homomorphism $\eta$ as described below.

We have a distinguished triangle
\[
p^* \LL_{\Af} \rTo \LL_{\LGQcal} \rTo \LL_{\LGQcal/\Af} \rTo^{\partial}  p^* \LL_{\Af/\Mf}[1]. 
\]
Composing the dual $\partial^\vee$ of the connecting homomorphism and the map $\phi_{\LGQcal/\Af} $ gives
\[
\phi_{\LGQcal/\Af}  \circ \partial^\vee: p^*\TT_{\Af} \rTo  \EE_{\LGQcal/\Af}[1]
\]
and hence a map 
\[
\eta = h^0(\phi_{\LGQcal/\Af}  \circ \partial^\vee): H^0(p^* \TT_{\Af}) \rTo  \Obs_{\LGQcal/\Af}.
\]
We define $\Obs_{\LGQcal}$ to be the cokernel of $\eta$.

To extend $\delta: \Obs_{\LGQcal/\Af} \to \Ocal_\LGQcal$ to $\Obs_{\LGQcal/\Mf_{g,k}}$, we must verify that $\delta\circ\eta = 0$.  Because $\eta$ factors through $H^1(\TT_{\LGQcal/\Af})\to \Obs_{\LGQcal/\Af}$, the vanishing of $\delta\circ\eta$ follows from the following lemma.

\begin{lem}
If  $(\Ccal_\LGQcal, \Pcal \to \pklogc, \u)$ denotes the universal LG-quasimap structure on $\LGQcal$, and if $\Ecal$ is the sheaf of sections of the vector bundle $\Pcal\times_\Gamma V$,  then the composition 
\begin{equation}\label{eq:cosection-vanish}
H^1(\TT_{\LGQcal/\Af})\rTo \Obs_{\LGQcal/\Af} = R^1\pi_* \Ecal \rTo^\delta R^1\pi_* \kk_{\Ccal_\LGQcal}  =  \Ocal_\LGQcal
\end{equation}
is zero.
\end{lem}
\begin{proof}
The proof is very similar to the proofs of \cite[Lem 3.6]{ChaLi:11} and \cite[Lem 3.4.4]{ClaDis:14}.  We sketch the proof here and refer the reader to \cite{ChaLi:11,ClaDis:14} for more details.
Let $\omega_{\Ccal_\Af/\Af}$ be the relative dualizing sheaf of the universal curve $\Ccal_{\Af}$ over $\Af$, and let $\Vb(\omega_{\Ccal_\Af/\Af})$ denote its corresponding vector bundle.  Let 
$Q = C(\pi_*\omega_{\Ccal_\Af/\Af})$ be the direct image cone of $\omega_{\Ccal_\Af/\Af}$ parametrizing global sections of $\omega_{\Ccal_\Af/\Af}$ on curves in $\Af$ (see \cite[Def 2.1]{ChaLi:11}), and let $\Ccal_Q$ be the universal curve over $Q$.  Composing the function $W$ with the section $\u:\Ccal_\LGQcal \to \Ecal$ defines a section $\ve = W\circ\u \in H^0(\Ccal_\LGQcal,\kk_{\Ccal_\LGQcal})$, and hence a morphism $\LGQcal\to Q$.  Denote by $\widetilde{\Phi}_\ve:\Ccal_\LGQcal \to \Ccal_Q$ the morphism of curves induced by (lifting) $\Phi_\ve$. This gives a commutative diagram
\[
\begin{diagram}
\Ccal_\LGQcal 			& \rTo^\u 		& \Vb(\Ecal)\\
\dTo^{\Phi_\ve} 		& 			& \dTo^{W}\\
\Ccal_Q			& \rTo^{\efp}	& \Vb(\omega_{\Ccal_\Af/\Af}),
\end{diagram}
\]
where $\efp$ is the tautological morphism.  From this we see that the following diagram is  commutative:
\[
\begin{diagram}
\pi^* \TT_{\LGQcal/\Af} & \rEq & \TT_{\Ccal/\Ccal_\Af} & \rTo & \u^*\Omega^\vee_{\Vb(\Ecal)/\Ccal_\Af}\\
\dTo          		       &         & \dTo                            &         & \dTo^{dW}\\
\pi^* \widetilde{\Phi}_\ve^* \TT_{Q/\Af} & \rEq & \widetilde{\Phi}_\ve^* \TT_{\Ccal_Q/\Ccal_\Af} & \rTo & \widetilde{\Phi}_\ve^* {{\efp}^*}\Omega^\vee_{\Vb(\omega_{\Ccal_\Af/\Af})/\Ccal_\Af}\\                              
\end{diagram}
\]
Applying $R^1\pi_*$ to the bottom right-hand arrow gives a homomorphism
\[
H^1({\Phi}_\ve^* \TT_{Q/\Af}) \rTo \Phi_\ve^* R^1\pi_{Q *} \omega_{\Ccal_Q/Q},
\]
which vanishes because  $Q$ is a vector bundle over $\Af$ and $\Ccal_Q$ is smooth over $\Ccal_\Af$.  As described in \cite[Eqn.~(3.13)]{ChaLi:11}, this implies that the composition
\[
H^1(\TT_{\LGQcal/\Af}) \rTo R^1\pi_* \u^*\Omega^\vee_{\Vb(\Ecal)/\Ccal_\Af} \rTo R^1\pi_* \u^* W^* \Omega^\vee_{\Vb(\omega_{\Ccal_\Af/\Af})/\Ccal_\Af}
\]
is equal to the composition
\[
H^1(\TT_{\LGQcal/\Af}) \rTo H^1({\Phi}_\ve^* \TT_{Q/\Af}) \rTo^0 \Phi_\ve^* R^1\pi_{Q *} \omega_{\Ccal_Q/Q},
\]
and hence it vanishes.  Using $\u^* W^* \Omega^\vee_{\Vb(\omega_{\Ccal_\Af/\Af})/\Ccal_\Af} = \kk_{\Ccal_\LGQcal}$ we see that the composition \eqref{eq:cosection-vanish} vanishes. 
\end{proof}

Now we can apply the general cosection localization theory of Kiem-Li \cite{KiLi:10} to construct our virtual cycle.
\begin{defn}
Suppose that all the marked points have a narrow insertion or $ev_i(\qmp)$ lies in a compact substack of $\X_{\theta,\gel}$ in the broad case, and that the cosection $\delta$ is defined as in Definition~\ref{def:the-cosection}.  The virtual cycle of the stack $\LGQ_{g,k}^{\ve,\lift}(\crst_\theta,\beta)$ is defined as 
$$[\LGQ_{g,k}^{\ve,\lift}(\crst_\theta,\beta)]^{vir}=[\LGQcal]^{vir}_{loc},$$
taken with respect to the cosection $\delta$.

\end{defn}
\begin{lem}
$[\LGQ_{g,k}^{\ve,\lift}(\crst_\theta,\beta)]^{vir}$ has virtual dimension
$$\dim_{vir}=\int_{\beta} c_1(V\git{\theta} G)+(\chat_{W,G}-3)(1-g)+k-\sum_i (\age(\gamma_i-q).$$
\end{lem}
\begin{proof}
Cosection localization preserves the virtual dimension. Therefore, the virtual dimension is the sum of the dimension of stack of $\Gamma$-bundle and the index of vector bundle
$\Pcal\times_{\Gamma} V$. Namely
$$\dim_{vir}=3g-3+k+\dim(G)g+c_1(\Pcal\times_{\Gamma} V)-\sum_i \age(\gamma_i)+n(1-g)-\dim G$$
$$=(n-\dim (G)-3)(1-g)+k+c_1(\Pcal\times_{\Gamma} \det(V))-\sum_i \age(\gamma_i).$$
Note that $\Pcal\times_{\Gamma} \det(V)$ is defined by a $\Gamma$-character. We choose its zero lift $\widehat{\det(V)}$ and define 
$$\int_{\beta} c_1(V\git{\theta} G)=c_1(\widehat{\det(V)}).$$ 

By taking a higher multiple if necessary, we can assume $\Pcal\times_{\Gamma} \det(V)$ is a $G\times \CC^*_R$ bundle, i.e, the tensor product of $\widehat{\det(V)}$ and a $\CC^*_R$ bundle $\det(V)_{R}$.
The $\CC^*_R$ bundle has the property
$$\det(V)_{R}^d=\omega^{\sum_i c_i}_{\log}.$$
Hence, 
$$c_1(\det(V)_R)=q(2(g-1)-k).$$
We can put everything together to obtain
$$\dim_{vir}=\int_{\beta} c_1(V\git{\theta} G)+(\chat_{W,G}-3)(1-g)+k-\sum_i (\age(\gamma_i-q).$$
\end{proof}

\subsection{Correlators}
\begin{defn}\label{def:correlator}
Suppose that $\alpha_i \in \Hcal_{W,G, comp}$. We define correlator
$$\langle \tau_{l_1}(\alpha_1),\cdots, \tau_{l_k}(\alpha_k)\rangle=\int_{[\LGQ_{g,k}^{\ve,\lift}(\crst_\theta,\beta)]^{vir}}\prod_i ev_i^*(\alpha_i)\psi^{l_i}_i.$$
One can define the generating function in a standard fashion.
\end{defn}
These invariants satisfy the gluing axioms for nodes that are narrow. The argument is standard and we leave it to the reader (see, for example, the proof of \cite[Thm 4.8]{CLL:13}).  For insertions that are not narrow, but where the evaluation map factors through a compact substack of $\X_{\theta}$, a form of the gluing axioms should also hold.  We will treat this in a future paper.
We do not expect a forgetful morphism or string/dilaton equations to hold, except in the chamber where $\ve=\infty$.

\section{Examples}\label{sec:examples}

In this section we consider several more examples of the GLSM, including some important examples studied by Witten \cite{Wit:97}.  We begin with some general considerations about toric quotients.

\subsection{Toric Quotients}
The hypersurface in Examples~\ref{ex:hypersurf} and \ref{ex:hypersurf-geom} is a special case of a toric
quotient; that is, where the group  $G=(\CC^*)^\littlem$ is an algebraic torus.
The geometric and combinatorial properties of the polarization are encoded in the
weights of the $(\CC^*)^\littlem$ action. Let $\GQ=(\Gq_{i j})$ be the {gauge weight matrix}, as described in Section~\ref{sec:choice}.
Note that some $\Gq_{i j}$ could be negative, and hence the resulting quotient could be
fail to be compact, but we always assume that $\GQ$ is of maximal rank (i.e.,
rank $\littlem$).

An important case is that of \emph{Calabi-Yau weights},
 where $\sum_j \Gq_{i j}=0$ for all $i$.  In this case, the quotient $[V\!\git{\theta} G]$ is Calabi-Yau and cannot be compact.
 In fact, $[V\!\git{\theta} G]$ or $\XS_{\mmapvalue}$ is compact if and only if $\GQ^{-1}(0)\cap
\RR^{\smalln}_{\geq 0}=\{0\}$, meaning that if $\bGq_1, \cdots, \bGq_{\smalln}$ are the
column vectors of $\GQ$, then the only nonnegative solution $\alphab
= (\alpha_1,\dots,\alpha_{\smalln}) \in (\RR^{\geq 0})^{\smalln}$ to the linear
equation
$$\alpha_1\bGq_1+\cdots+\alpha_{\smalln} \bGq_{\smalln}=0$$ is the zero solution (See
\cite[\S2]{Pop:12}). Note that this condition is entirely independent
of the phase ($\theta$ or $\mmapvalue$).
  
If the above condition fails, the quotient is 
not compact. However, one can choose a maximal collection of column
vectors of $\GQ$ with the property above.  After possibly reindexing, we may write  $\CC^{\smalln}=\CC^{{\bigN}}\times \CC^{\bigK}$ with variables $x_1, \cdots,
  x_{\bigN}, \cf_1\cdots, \cf_\bigK$ such that $\CC^{\bigN}$ corresponds to the maximal
  collection of column vectors. In this case the subset
  $[(\CC^\bigN\times\{\0\})\git{\theta}(\CC^*)^\littlem]\subset
  [\CC^{\smalln}\git{\theta}(\CC^*)^\littlem]$ is compact and depends on a choice of phase ($\mmapvalue$ or $\theta$).  This compact piece may be empty, but if it is not empty, we call it a \emph{maximal compact piece}.  In general,
there may be several maximal compact pieces.

A particularly interesting case is when
$[\CC^{\smalln} \git{\theta}(\CC^*)^\littlem] = [(\CC^{{\bigN}}\times \CC^\bigK) \git{\theta}(\CC^*)^\littlem]$ 
is a toric vector bundle over the maximal compact piece $\X_{base} = [\CC^{{\bigN}} \git{\theta}(\CC^*)^\littlem]$.  
Each remaining variable $\cf_j$ defines a line bundle $\TLB_j \rightarrow \X_{base}$. Each corresponding column vector $\bGq_{\cf_j}$ of $\GQ$ can be written as  
\[
\bGq_{\cf_j} =\alpha_{1,j}\bGq_{x_1}+\cdots+\alpha_{{\bigN},j} \bGq_{x_n}
\]
for some choice of $\alpha_{i,j} \leq 0$. 
Letting $D_i$ be the toric divisor corresponding to $\bGq_{x_i}$, we have 
\[
c_1(\TLB_j)=\sum_{i=1}^{\bigN} \alpha_{i j} D_i, \qquad \text{or} \qquad \TLB_j = \bigotimes_{i=1}^{\bigN} \Ocal(\alpha_{i j} D_i).
\]

A very important subclass of the toric examples consists of the so-called \emph{hybrid models}.   \begin{defn}
For a torus $G=(\CC^*)^m$, a phase $\theta$ of $(W, G)$ is called a \emph{hybrid model} if 
\begin{enumerate}
\item The quotient $\X_{\theta}\rightarrow \X_{base}$ has the structure of a toric bundle over a compact base $\X_{base}$, and 
\item The $\CC^*_R$-weights of the base variables are all zero. 
\end{enumerate}
\end{defn}
Both the geometric and the LG phases of the hypersurface in Example~\ref{ex:hypersurf} were hybrid models.  Several examples of hybrid models have been worked out in detail by E.~Clader in \cite{Cla:13}.

\subsection{Complete Intersections}\label{exa:toricCI}

Suppose that $G = \CC^*$ and we have several quasihomogeneous polynomials $F_1, F_2,
\dots, F_M \in \CC[\cb_1,\dots,\cb_{\bigN}]$ of $G$-degree $(d_1, \dots, d_M)$, where each variable $\cb_i$ has $G$-weight $\Gq_i>0$. We assume that the $F_j$
 intersect transversely in $\WP(\Gq_1, \dots, \Gq_{\bigN})$ and define a
 complete intersection. Let
 $$W=\sum_i \cf_i F_i{\colon} \CC^{{\bigN}+M}\rightarrow \CC,$$ where we assign
 $G$-weight $-d_i$ to $\cf_i$. In the special case that 
 $\sum_i \Gq_i=\sum_j d_j$,
 then the complete intersection defined by $F_1=\cdots=F_M=0$
 is a Calabi-Yau orbifold in $\WP(\Gq_1,\dots,\Gq_{\bigN})$. One can view this as a \emph{toric LG-model  for the complete intersection}.  We do not assume the
  Calabi-Yau condition here.

The critical set of $W$ is defined by the following equations:
\begin{equation}\label{eq:crit-for-toricLG}
\partial_{\cf_j}W=F_j=0, \ \partial_{\cb_i}W=\sum_j \cf_j \partial_{\cb_i}
 F_j=0.
\end{equation} 

 The moment map is 
\[
  \mmap=\sum_i  \frac12 \Gq_i|x_i|^2-\frac12 \sum_j d_j|p_j|^2.
\] 

\subsubsection{Phases}

As with the hypersurface, there are two phases,
$\mmapvalue>0$ and $\mmapvalue<0$. When $\mmapvalue>0,$ any choice of $p_1,\dots
p_M$ determines a nontrivial ellipsoid $E \subset \CC^{\bigN}$ of 
points $(x_1, \dots, x_{\bigN})$ such that
  $(x_1,\dots,x_{\bigN},p_1,\dots,p_M)$ lies in $\mmap^{-1}(\mmapvalue)$.
  Quotienting by $U(1)$, the first projection $pr_1:E \times \CC^M \to
  E$ induces a map $\XS_\mmapvalue \to \X_{base} = \WP(\Gq_1, \dots, \Gq_{\bigN})$,
  corresponding to the maximal collection of column vectors $(\Gq_1,
  \cdots, \Gq_{\bigN})$.
The full quotient is $\XS_\mmapvalue = \bigoplus_j \Ocal(-d_j)$ over
$\X_{base}$. Similarly, for $\mmapvalue<0$, the toric variety is $\bigoplus_i
  \Ocal(-\Gq_i)$ over $\WP(d_1, \dots, d_M)$.
  
  \begin{enumerate}
\item[$\mmapvalue>0$]
The chamber $\mmapvalue>0$ is called the \emph{geometric phase}.  Here we have $(\cb_1, \dots, \cb_{\bigN})\neq (0,\dots,
 0)$. 
 In this case, we can choose our $\CC^*_R$ action to have weights $\Rq_{x_i}=0, \Rq_{\cf_j}=1$, which gives a hybrid model, and the trivial lift $\lift_0$ is a good lift of $\theta$.
The polynomial $W$ has $\CC^*_R$-degree $d=1$, and the element $J$ is trivial, so $\Gamma \cong G \times \CC^*_R$.  
 The critical locus is defined by Equation~\eqref{eq:crit-for-toricLG}.
Since the $F_i$ intersect transversely, the $d F_i$ are linearly
 independent for $(\cb_1, \dots, \cb_{\bigN})\neq (0, \dots, 0)$. Therefore,
 all the $\cf_i$ vanish, and the critical set is the complete intersection 
 \[
 \{F_1 = \dots = F_M = 0\} 
 \]
in the zero section of $\X_\mmapvalue \to \WP(\Gq_1,\dots,\Gq_{\bigN})$. 

\item[$\mmapvalue<0$] The chamber $\mmapvalue<0$ is called the \emph{LG phase}.  If we happen to have $d_1=\cdots=d_r = d$, we may take  $\Rq_{x_i}=\Gq_i,\Rq_{\cf_j}=0$, and we again have a hybrid model with good lift $\lift_0$.  In this hybrid model case, we have 
\[
J = (\exp(2\pi i \Rq_1/d), \dots, \exp(2\pi i \Rq_n/d), 1,\dots, 1),
\]
and we have 
\begin{align*}
\Gamma &= \{((st)^{\Gq_1}, \dots, (st)^{\Gq_n}, s^{-d}, \dots, s^{-d})| s,t \in \CC^*\}\\
& = 
\{(\alpha^{\Gq_1},\dots, \alpha^{\Gq_n}, \beta^{d},\dots,\beta^{d})| \alpha,\beta\in \CC^*\}
\end{align*}
with 
\[
\chiR(\alpha^{\Gq_1},\dots, \alpha^{\Gq_n}, \beta,\dots,\beta) = \alpha^{d}\beta.
\]

Again, the critical locus is defined by Equation~\eqref{eq:crit-for-toricLG}, but now we have $(\cf_1, \dots, \cf_r)\neq (0, \dots, 0)$.  This implies that $(\cb_1, \dots,
 \cb_{\bigN})=(0,\dots, 0)$.  So the critical set is the zero section of
 the corresponding quotient $\X_\mmapvalue= \bigoplus_i \Ocal({-\Gq_i})   \to \WP(d,\dots, d)$. 
 Thus, for each choice of 
$(\cf_1, \dots, \cf_n) \in \WP(d,\dots, d)$, we have a pure 
LG-model of superpotential
$\sum_i \cf_i F_i$.  One can view this as a family of pure LG-theories. 
\end{enumerate}

\subsubsection{LG-quasimaps into Complete Intersections}

Now assume that $d_1=\cdots=d_r = d$.  In the geometric phase, with the trivial lift $\lift_0$, the stack of LG-quasimaps is  
  $$\{(\Ccal, \Lcal, s_1, \cdots, s_\bigN, p_1, \cdots, p_r); s_i\in H^0(\Ccal,\Lcal), p_j\in H^0(\Ccal,\Lcal^{-d}\otimes \klogc \}$$
      satisfying the stability condition. 
We obtain a theory similar to that of the geometric phase of the hypersurface in Example~\ref{ex:hypersurf-geom} and the corresponding $p$-field theory. 

At the LG-phase with $\ve=\infty$, with the trivial lift $\lift_0$, the moduli space consists of $$\u=(p_1, \cdots, p_r): \Ccal \rightarrow W\PP^{r-1}(d,d,\dots,d),$$
  where $W\PP^{r-1}(d,d,\dots,d)$ is weighted projective space, corresponding to usual (unweighted) projective space with an order-$d$ gerbe, and $\Lcal^{-d}\otimes \klogc \cong \u^*\Ocal(1)$. 
  Similar to FJRW-theory, we have the condition
  $$\Lcal^d\cong \klogc \otimes \u^*\Ocal(-1).$$
  This is the hybrid theory constructed by Clader \cite{Cla:13}. 
  
\begin{rem}
  When the $F_j$'s have different degrees $d_j$, there is generally no good lift.
Moreover, the sections $p_j\in H^0(\Ccal,\Lcal^{-d_j}\otimes \klogc)$ are sections of different bundles, so we do not have a simple stable map description as before. Physicists have referred to this case  as a \emph{pseudohybrid model} \cite{AsPl:09}.  
  We will come back to this on a different occasion.
\end{rem}

\subsection{Hypersurface in a Product}
 
The previous examples all have a one-dimensional parameter space for $\tau$. We now give an example of multi-parameter model, namely a hypersurface of bidegree $(b,b')$ in a product of weighted projective spaces 
\[
\WP(\Gq_1, \dots, \Gq_{\bigN})\times \WP(\Gq'_1, \dots, \Gq'_M).
\]

 Consider the action of $\CC^*$ on $\CC^{\bigN}$ with positive weights
 $(\Gq_1, \dots, \Gq_{\bigN})$, and let $z_1,\dots, z_\bigN$ be the coordinates on $\CC^{\bigN}$. Its quotient is weighted projective space
 $\WP(\Gq_1, \dots, \Gq_{\bigN})$. Consider another weighted projective
 space given by a different $\CC^*$ acting on $\CC^M$ with weights
 $(\Gq'_1, \dots, \Gq'_M)$, and let $w_1,\dots, w_M$ be the coordinates on $\CC^{M}$.  We combine these by setting $G = \CC^*\times\CC^*$ and letting $G$ act on $\CC^{\bigN+M} \times \CC$ with weights
 $$\left(\begin{array}{ccccccc} \Gq_1,& \dots,& \Gq_{\bigN}& 0, &\dots, &0 &
   -\Gq\\ 0, & \dots, & 0, & \Gq'_1, &\dots, & \Gq'_M& -\Gq'
 \end{array}\right).$$
That is, if the last factor $\CC$ has coordinate $p$, then $p$ has bidegree $(\Gq,-\Gq')$.
 
Let $F$ be any bihomogeneous polynomial in $\CC[z_1,\dots, z_{\bigN}]\otimes \CC[w_1,\dots,w_M]$ of bidegree $(b,b')$ that is \emph{nondegenerate}, in the sense that if 
\[
\frac{\partial F}{\partial z_i} = 0 = \frac{\partial F}{\partial w_j} \quad \forall i \in \{1,\dots,\bigN\} \quad \forall j\in \{1,\dots,M\},
\]
then either $z_1 = \cdots = z_\bigN = 0$ or $w_1=\cdots = w_M =0$.
 As in Example~\ref{ex:hypersurf}  let  \[
W = pF,\] so that $W$ is $G$-invariant.
The critical locus of $W$ is 
\[
\left\{p\frac{\partial F}{\partial z_i} = 0, p\frac{\partial F}{\partial w_j} = 0, F=0\right\}. 
\]
The moment map $\mmap \colon \CC^{{\bigN}+M}\times \CC \to \ufrak(1)\oplus\ufrak(1) = \RR^2$ is
 $$\mmap_1=\frac12\left(\sum_i \Gq_i |z_i|^2-{\Gq}|p|^2\right); \qquad
 \mmap_2=\frac12\left(\sum_j \Gq'_j|w_{j}|^2-{\Gq}'|p|^2\right).$$ The critical
 loci are
 \begin{enumerate}[i.)]
 \item $\{z_1=\cdots=z_{{\bigN}} = 0,\ p=0\}$;
 \item $\{w_{1}=\cdots =w_{M}=0, \ p=0\}$;
 \item $\{z_1=\cdots=z_{\bigN} = w_1=\cdots =w_M=0\}$.
 \end{enumerate}
 The corresponding critical values are
 \begin{enumerate}[i.)]
 \item $\mmapvalue_1=0, \mmapvalue_2\geq 0$;
 \item $\mmapvalue_2=0, \mmapvalue_1\geq 0$;
 \item $\mmapvalue_1, \mmapvalue_2<0, \frac{\mmapvalue_1}{{\Gq}}=\frac{\mmapvalue_2}{{\Gq}'}.$
 \end{enumerate}
 These divide $\RR^2$ into three phases.

\begin{enumerate}
\item[$\mmapvalue_1, \mmapvalue_2>0$:]  In this phase we have $(z_1, \dots, z_{\bigN})\neq (0, \dots, 0)$, and 
 $(w_{1}, \dots, w_{M})\neq (0, \dots, 0)$.  The maximal
 collection is
 $$\left(\begin{array}{cccccc} \Gq_1,& \dots,& \Gq_{\bigN}& 0, &\dots, &0
   \\ 0, & \dots, & 0, & \Gq'_1, &\dots, & \Gq'_M
 \end{array}\right)$$
 The quotient can be expressed as the total space of the line
 bundle
\[
 \Ocal_1(-{\Gq})\otimes \Ocal_2(-{\Gq}') = K_{\WP(\Gq_1, \dots, \Gq_{\bigN})}\otimes K_{\WP(\Gq'_1,
   \dots, \Gq'_M)}
\] of bidegree $(-{\Gq},
 -{\Gq}')$ over $\WP(\Gq_1, \dots, \Gq_{\bigN})\times \WP(\Gq'_1,
 \dots, \Gq'_M)$.

In the GIT formulation, let $\LV_\theta$ have a generating section $\ell$, and let $\theta$ have $G$-weights $(-e,-e')$, with $e,e'>0$.  Any $G$-invariant section of $\LV_\theta$ is given by a polynomial in the $z_i$ and  $\ell$, and can be written as a sum of $G$-invariant monomials in the $z_i$ and $\ell$, so to find the unstable and semistable points it suffices to consider only the $G$-invariant monomials of the form $\prod_{i=1}^{\bigN} z_i^{a_i} \prod_{j=1}^M z_{\bigN+j}^{a'_j} \ell^k$.  Since both $e$ and $e'$ are positive, any $G$-invariant monomial must have at least one $a_i$ and at least one $a'_j$ not vanishing.  This implies that the locus $\{z_1=z_2=\cdots=z_{\bigN+M}\}$ is unstable.  But any monomial of the form $z_i^{e}w_j^{e'}\ell$ will be $G$ invariant and will vanish only on the locus $z_i = w_j =0$.  Letting $i$ and $j$ range over all possible values shows that every point that is not in $\{z_1=z_2=\cdots=z_{\bigN+M}\}$ is semistable,  

Choose $\CC^*_R$ to have weights $(0,\dots, 0,1)$, so that $W$ has $\CC^*_R$-weight $1$.  Let $\lift_0$ be the lift of $\theta$ with $\CC^*_R$-weight $0$.  Every monomial of the form $z_i^{e}w_j^{e'}\ell$ is also $\CC^*_R$-invariant, so $\lift_0$ is a good lift of $\theta$.

The semistable points of the critical locus of $W$ are given by the equations 
\[
p=0, F=0,
\]
which is the hypersurface defined by the vanishing of $F$ in the image of the zero section of 
$\Ocal_1(-{\Gq})\otimes \Ocal_2(-{\Gq}')  \to \WP(\Gq_1, \dots, \Gq_{\bigN})\times \WP(\Gq'_1,
 \dots, \Gq'_M).$
This is the geometric phase.

\item [$\mmapvalue_2<0, \frac{\mmapvalue_1}{{\Gq}}>\frac{\mmapvalue_2}{{\Gq}'}$:]

In this phase a similar
 analysis implies $(z_1, \dots, z_{\bigN})\neq (0, \dots, 0), p \neq
 0$. 
 
 The maximal collection is
 $$\left(\begin{array}{cccccc} \Gq_1,& \dots,& \Gq_{\bigN}, &-{\Gq} \\ 0, &
   \dots, & 0, & - {\Gq}'
 \end{array}\right)$$
 We can quotient by $(\CC^*)^2$, but since the two actions on $z_{{\bigN}+M+1}$
 intertwine, we do not obtain $\WP(\Gq_1, \dots, \Gq_{\bigN})\times \B
 \ZZ_{\Gq}$. Instead, we obtain a nontrivial gerbe over $\WP(\Gq_1,
 \dots, \Gq_{\bigN})$.
 
 To be more specific, dividing by the first copy of $\CC^*$, we
 obtain $\Ocal_1({-{\Gq}})\rightarrow \WP(\Gq_1, \dots, \Gq_{\bigN})$, where
 $\Ocal_1(1)$ is the standard $\CC^*$ bundle associated with the first $\CC^*$.  Then,
 we quotient out the second $\CC^*$. We obtain a
nontrivial $\B\ZZ_{{\Gq}'}$-bundle over $\WP(\Gq_1, \dots, \Gq_{\bigN})$, called a \emph{gerbe}. We denote it by
 $\WP(\Gq_1, \dots, \Gq_{\bigN})^{-\frac{{\Gq}}{{\Gq}'}}$.  Our quotient
 is the total space of $\bigoplus_i \Ocal_2(-\Gq'_i)$.

We choose the $\CC^*_R$ action in this phase to have weights $(0, \dots, 0, b_1',\dots, b_M', 0)$.  Again, the lift $\lift_0$ is a good lift of $\theta$.

The semistable points of the critical locus are those with $w_1=\cdots=w_M = 0 = F$, so this phase gives us a mixture of LG and geometric phases with the $w$-directions corresponding to an LG model and the $z$-directions corresponding to a geometric model.
 
\item[$\{\mmapvalue_1<0,\frac{\mmapvalue_1}{{\Gq}}<\frac{\mmapvalue_2}{{\Gq}'}\}$]  The analysis for this phase is similar to the previous one and yields a different mixture of LG and geometric phase with the $z$-directions now corresponding to an LG model and the $w$-directions corresponding to a geometric model.
\end{enumerate}

\subsection{Non-Abelian Examples}

The subject of gauged linear sigma models for non-Abelian groups is a
very active area of research in physics and is far from complete.
Here, we discuss the example of complete intersection of Grassmannian
varieties. One should be able to discuss everything in the
setting of complete intersections of quiver varieties, although the
details have not been worked out. It would be very interesting to explore  
mirror symmetry among Calabi-Yau complete intersections in quiver
varieties.

\subsubsection{Complete Intersection in a Grassmannian}

Consider a complete intersection in the Grassmannian
$\Gr(k,{n})$. The space $\Gr(k, {n})$ can be constructed as the GIT  quotient $M_{k,{n}}\git{}\GL(k, \CC)$, where $M_{k,{n}}$ is the space of
$k\times {n}$ matrices and $\GL(k, \CC)$ acts as matrix multiplication on
the left.  

The Grassmannian $\Gr(k, {n})$ can also be embedded into
$\PP^{{\bigN}}$ for ${\bigN}=\frac{{n}!}{k!({n}-k)!}-1$ by the Pl\"ucker embedding
$$A \mapsto (\dots, \det(A_{i_1,\dots, i_k}), \dots),$$ where
$A_{i_1, \cdots, i_{k}}$ is the $(k\times k)$-submatrix of $A$ consisting of the columns $i_1, \dots,
i_k$. 

The group $G = \GL(k,\CC)$ acts on the Pl\"ucker coordinates $B_{i_1, \cdots, i_k}(A)=\det(A_{i_1,\cdots, i_k})$ by the determinant, that is, for any $U\in G$, and $A\in M_{k,{n}}$ we have
$$B_{i_1, \cdots, i_k}(UA)=\det(U)B_{i_1, \cdots, i_k}(A),$$

Let $F_1, \dots, F_s\in \CC[B_{1,\dots,k}, \dots, B_{{n}-k+1,\dots,{n}}]$ be
degree-$d_j$ homogeneous polynomials such that the zero loci $Z_{F_j}=\{F_j=0\}$ and the Pl\"ucker embedding of  $\Gr(k,{n})$ all intersect transversely in $\PP^{{\bigN}}$.  We let 
$$Z_{d_1, \cdots, d_s}=\Gr(k,{n})\cap \bigcap_{j} Z_{F_j}$$
denote the corresponding complete intersection.
 
The analysis of $Z_{d_1, \cdots, d_s}$ is similar to the Abelian case. Namely, let
$$W=\sum_j \cf_j F_j{\colon} M_{k,{n}}\times \CC^s\rightarrow \CC$$
be the superpotential.
We assign an action of $G = \GL(k, \CC)$ on $\cf_j$ by $\cf_j\rightarrow
\det(U)^{-d_j}$. 

The phase structure is similar to that of a complete intersection in projective space.  The moment map is given by $\mmap(A,p_1,\dots,p_s) = \frac{1}{2} (A\bar{A}^T - \sum_{i=1}^s d_i |p_i|^2)$.  Alternatively, to construct a linearization for GIT, the only characters of $GL(k,\CC)$ are powers of the determinant, so $\theta(U) = \det(U)^{-\thetaweight}$ for some $\thetaweight$, and $\mmapvalue$ will be positive precisely when $\thetaweight$ is positive.

Let $\ell$ be a generator of $\CC[\LV^*]$ over $\CC[V^*]$.   Any element of $H^0(V,\LV_\theta)$ can be written as a sum of monomials in the Pl\"ucker coordinates $B_{i_1,\dots,i_k}$ and the $p_j$ times $\ell$.  Any $U\in G$ will act on a monomial of the form $\prod B_{i_1,\dots,i_k}^{b_{i_1,\dots,i_k}} \prod p_j^{a_j}\ell^m$ by multiplication by $\det(U)^{\sum b_{i_1,\dots,i_k} -\sum d_j a_j -m\thetaweight}$.

\begin{enumerate}
\item[$\thetaweight>0$:]
In order to be $G$-invariant, a monomial must have $\sum b_{i_1,\dots,i_k}>0$, which implies that any point with every $B_{i_1,\dots,i_k} =0$ must be unstable, but for each $m>0$ and each $k$-tuple $(i_1,\dots,i_k)$ the monomial $B_{i_1,\dots,i_k}^{m\thetaweight} \ell^m$ is $G$ invariant, so every point with at least one nonzero $B_{i_1,\dots,i_k}$ must be $\theta$-semistable.  
Thus $[V\!\git{\theta} G]$ is isomorphic to the bundle $\bigoplus_j \Ocal(-d_j)$ over $\Gr(k,{n})$. 

Furthermore, $W$ is quasihomogeneous of degree one with respect to the following compatible $\CC^*_R$ action
$$\lambda(A, \cf_1, \cdots, \cf_s)=(A, \lambda \cf_1, \cdots, \lambda
\cf_s).$$ 
The trivial lift $\lift_0$ is a good lift because each monomial of the form $B_{i_1,\dots,i_k}^{m\thetaweight} \ell^m$ is $\Gamma$ invariant for the action induced by $\lift_0$.

As in Example~\ref{exa:toricCI}, the critical locus in this phase is given by $\cf_1 = \cdots = \cf_s = 0 = F_1 = \cdots = F_s$, so we recover the complete intersection $F_1 = \cdots = F_s$ in $\Gr(k,{n})$.

As in the toric case, we call this phase the \emph{geometric phase}.

\item[$\thetaweight<0$:]
We call the case where $\thetaweight<0$ the \emph{LG-phase}.  In this case, in order to be $G$-invariant a monomial $\prod B_{i_1,\dots,i_k}^{b_{i_1,\dots,i_k}} \prod p_j^{a_j}\ell^m$ must have $\sum a_j>0$, which implies that any point with every $p_j =0$ must be unstable, but for each $m>0$ and each $j$ the monomial $p_j^{m\thetaweight} \ell^{m d_j}$ is $G$-invariant, so every point with at least one nonzero $p_j$ is $\theta$-semistable.  Therefore $V^{ss}_G(\theta) = M_{k,n} \times (\CC^s\setminus \{\0\})$.  Again, since the $F_j$ and the image of the Pl\"ucker embedding are transverse,  the equations 
$\partial_{B_{i_1,\dots,i_k}}W=\sum_j \cf_j \partial_{B_{i_1,\dots,i_k}} F_j=0$ imply that the critical locus is
$\left[(\{\0\} \times (\CC^s\setminus \{\0\}))/\GL(k,\CC)\right]$ inside $[V\!\git{\theta} G] = \left[(M_{k,n} \times (\CC^s \setminus \{\0\}))/\GL(k,\CC)\right]$.
This phase does not immediately fit into our theory because we have an infinite stabilizer $\SL(k, \CC)$ for any points of the form $(\0, \cf_1, \dots , \cf_s)$. This means that the quotient $[V\!\git{\theta} G]$ is an Artin stack.

Hori-Tong \cite{HoTo:07} have analyzed the gauged linear sigma model of the Calabi-Yau complete
intersection $Z_{1, \dots, 1}\subset \Gr(2,7)$ which is defined by $7$
linear equations in the Pl\"ucker coordinates. They gave a physical
derivation that its LG-phase is equivalent to the Gromov-Witten theory of
the so-called \emph{Pfaffian variety}
$$\Pf(\bigwedge^2 \CC^7)=\{A\in \bigwedge^2 \CC^7; A\wedge A\wedge
A=0\}.$$ It is interesting to note that the Pfaffian $\Pf(\bigwedge^2 \CC^7)$ is not a complete intersection. For additional work on this example, see \cite{Rod:00, Kuz:08, HoKo:09, ADS:13} 
\end{enumerate}

\subsubsection{Complete Intersections in a Flag Variety}

Another class of interesting examples is that of complete intersections in partial flag varieties.  The partial flag variety $\Fl(d_1,\cdots, d_k)$
parametrizes the space of partial flags
$$0\subset V_1\subset \cdots V_i\subset\cdots V_k=\CC^n$$ such that
$\dim V_i=d_i.$ The combinatorial structure of the equivariant cohomology
of $\Fl(d_1, \cdots, d_k)$ is a very interesting subject in its own right. 

For our purposes, $\Fl(d_1, \cdots, d_k)$ can be constructed as
a GIT or symplectic quotient  of the vector space 
\[
 \prod_{i=1}^{k-1} M_{d_i,d_{i+1}}
\]
by the group 
\[
G = \prod_{i=1}^{k-1} \GL(d_i,\CC).
\]

The moment map sends the element $(A_1,\dots,A_{k-1})\in \prod_{i=1}^{k-1} M_{i,i+1}$ 
to the element $\frac{1}{2}(A_1\bar{A}_1^T,\dots,A_{k-1}\bar{A}_{k-1}^T) \in \prod_{i=1}^{k-1}{\ufrak}({d}_i)$. 

Let the $\chi_i$ be the character of $\prod_j \GL({d}_j)$ given by the determinant of
$i$th factor.  Each character $\chi_i$ defines a line bundle on the vector space $ M_{d_1,d_2}\times \cdots \times M_{d_{k-1},d_k}$, which descends to a line bundle $\XLB_i$ on $\Fl(d_1,\cdots, {n}_k)$. A hypersurface of multidegree $(\ell_1, \dots, \ell_k)$ is a section of $\bigotimes_j \XLB^{\ell_j}_j$. To consider the gauged linear sigma model for the complete intersection $F_1=\cdots =F_s=0$ of such sections, we again consider the vector space 
\[
V = \prod_{i=1}^{k-1} M_{d_i,d_i+1} \times \CC^s,
\]
with coordinates $(\cf_{1},\dots,\cf_{s})$ on $\CC^s$ and 
superpotential
\[
W=\sum_{j=1}^s \cf_{j} F_j.
\]
We define an action of $G$ on $\cf_i$ by $(g_1,\dots,g_{k-1}) \in G$ acts on $\cf_i$ as $\prod_{j=1}^{k-1} \det(g_j)^{-\ell_{ij}}$, where $\ell_{ij}$ is the $j$th component of the multidegree degree of $F_i$. 

We may describe the polarization as 
\[
\theta = \prod_{i=1}^{k-1} \det(g_i)^{-\thetaweight_i},
\]
or the moment map as 
\[
\mmap(A_1,\dots,A_{k-1},\cf_1,\dots,\cf_s) =\frac{1}{2}(A_1\bar{A}_1^T - \sum_{i=1}^s \ell_{1j}|\cf_j|^2,\dots,A_{k-1}\bar{A}_{k-1}^T - \sum_{i=1}^s \ell_{k-1,j}|\cf_j|^2 ).
\] 
this gives a phase structure similar to the complete intersection in a product of projective spaces. 

For example, when $\thetaweight_i>0$ for all $i\in \{1,\dots,k-1\}$ we can choose a compatible $\CC^*_R$ action with weight $1$ on $\cf_j$ and weight $0$
on each $A_i$, and the trivial lift $\lift_0$ is a good lift of $\theta$ in this phase.

This example should be easy to generalize to complete intersections in quiver varieties.  It would be very interesting to calculate the details of our theory for these examples.

\subsection{Graph Spaces and Generalizations}

\subsubsection{Graph Spaces}\label{sec:RR}

  The graph moduli space is very important in Gromov-Witten theory.  It is used to define the  $I$-function and prove genus zero mirror symmetry (see, for example, \cite{Giv98}).   We can construct it in the GLSM setting as follows.
  Suppose that we have a phase $\theta$ of a GLSM $W: \CC^n/G\rightarrow \CC$ with a certain $R$-charge $\CC^*_R$, defining $\Gamma$ and a good lift $\lift$ of $\theta$. 
  
We construct a new GLSM  as follows.  
Let $V' = V\times \CC^2$, and let $\CC^*$ act on $\CC^2$ with weights $(1,1)$. Let $G' = G\times \CC^*$ act on $V'$ with the product action, so $G$ acts trivially on the last two coordinates and $\CC^*$ acts trivially on the first $n$ coordinates.
Let $\theta':G'\to \CC^*$ be given by sending any $(g,h) \in G\times \CC^*$ to $\theta'(g,h) = \theta(g) h^{-e}$ for some $e>0$.  The GIT quotient is the product $[V'\git{\theta'}G'] = [V\git{\theta}G]\times \PP^1$.  Let $W'$ be defined on $V'$ by the same polynomial as $W$, so that the critical locus of $W'$ is $\CC^2$ times the the critical locus of $W$, and the GIT quotient of the critical locus is $\PP^1$ times the corresponding quotient in the original GLSM.

Keeping the same $R$-charge (that is, letting $\CC^*_R$ acts trivially on the last two coordinates of $V'$), we have $\Gamma' = \Gamma \times \CC^*$, and we construct a  lift $\lift'$ of $\theta'$ by sending $(\gamma,h) \in \Gamma \times \CC^*$ to $\lift'(\gamma,h) = \lift(\gamma)h^{-e}$. It is easy to see that $\lift'$ is a good lift of $\theta'$ if $\lift$ is a good lift of $\theta$.

In the $\ve = \infty$ case, the last two coordinates $(z_1, z_2)$ induce a stable map $\Ccal\rightarrow \PP^1$. 
For other $\ve$-stable case, we choose $e>\!\!>0$ such that stability condition for the second $\CC^*$ is always in the $\infty$-chamber.
 There is no base point for $(z_1, z_2)$ which induces a stable map $\Ccal \rightarrow \PP^1$. Therefore, it can be reformulated as usual 
GLSM moduli space of $[V\git{\theta}G]$ with additional data of a stable map $f: \Ccal \rightarrow \PP^1$.

\subsubsection{Generalization of the graph space}\label{sec:RRS}

  We can generalize slightly the graph moduli space to obtain a new moduli space with a remarkable property. Let's take the quintic GLSM as an example. Now, we consider a new GLSM on $\CC^{6+2}/(\CC^*)^2$ with charge matrix
  $$\left(\begin{array}{cccccccc}
  1&1&1&1&1&-5&d&0\\
  0&0&0&0&0&0&1&1
  \end{array}\right).$$
  for an integer $d>0$. 
  
  Let's look at its chamber structure. The moment maps are
  $$\mu_1=\frac{1}{2}(|x_1|^2+|x_2|^2+|x_3|^2+|x_4|^2+|x_5|^2-5|p|^2+d|z_1|^2),\     \mu_2=\frac{1}{2}(|z_1|^2+|z_2|^2).$$
  It has three chambers. We are interested in the chamber $0<\mu_1<d\mu_2$.  This corresponds to a character $\theta$ of $G=\CC^*\times \CC^*$ with weights $(-e_1,-e_2)$ and $0<e_1<d e_2$.  The unstable locus for this character is 
  $$\{(x_1,x_2, x_3,x_4,x_5,z_1)= (0,0,0,0,0,0)\}\cup \{(p,z_2)= (0,0)\}\cup\{(z_1,z_2)= (0,0)\}.$$

Taking the superpotential $W = \sum_{i=1}^5 x_i^5$ and the $R$-charge of weight $(0,0,0,0,0,1,0,0)$, we have $\Gamma = G\times \CC^*_R = \{(a,a,a,a,a,w,ba^d,b)\mid 
a,b,w\in \CC^*\}$ and the map $\chiR$ takes  $(a,a,a,a,a,w,ba^d,b)$ to $wa^5$.
There is no good lift of $\theta$, so we restrict to the case of $\ve = 0+$.  We must choose some lift for the stability condition, so we take $\lift(a,b,w) = a^{-e_1}b^{-e_2}$.  Any other lift will give the same stability conditions.

The resulting moduli problem consists of   
\begin{align*}
 \{(\Ccal,\mrkp_1,\dots,\mrkp_k, \Acal, \Bcal, x_1, \cdots, x_5, p, z_1,z_2)\mid   x_i\in &H^0(\Ccal,\Acal), p\in H^0(\Ccal,\Acal^{-5}\otimes \klogc) \\
 & z_1\in H^0( \Ccal,\Acal^{d}\otimes \Bcal), 
  z_2\in H^0( \Ccal,\Bcal)\}
  \end{align*}
satisfying the stability condition that $\u^*\Lcal_{\lift} =\Acal^{-e_1}\Bcal^{-e_2}$ is ample on all components where $\klogc$ has degree $0$.

  This GLSM admits a $\CC^*$ action on $z_2$. The induced action on the moduli space has three types of fixed point loci: the Gromov-Witten locus, FJRW-locus
  and the theory of a point.  This remarkable property gives us the hope that we can extract a relation between Gromov-Witten theory and FJRW-theory
 geometrically by using localization techniques on this moduli space. A program is being carried out right now for the $\ve=0^+$ theory  \cite{RRS,CJR}.
A theory based on the same GIT-quotient, but with a different stability condition, was discovered and the localization argument was carried out independently by Chang-Li-Li-Liu \cite{CLL:15}.



\def\cprime{$'$}

\end{document}